\documentclass{amsart}
\usepackage{fullpage}
\usepackage{colonequals}
\usepackage[alphabetic]{amsrefs}
\usepackage{enumitem,amssymb}
\usepackage{stmaryrd}
\usepackage{mathtools}
\usepackage{xcolor,tikz}
\usetikzlibrary{decorations.pathreplacing}

\usetikzlibrary{positioning,calc}
\usetikzlibrary{arrows.meta}
\usepackage{hyperref}
\usepackage{tikz-cd}
\usepackage{adjustbox}
\newtheoremstyle{hypothesisstyle}
  {\topsep} 
  {\topsep} 
  {\itshape} 
  {} 
  {\bfseries} 
  {.} 
  {.5em} 
  {} 
\theoremstyle{hypothesisstyle}
\newtheorem{hypothesis}{Hypothesis}

\usepackage[normalem]{ulem}
\hypersetup{
    colorlinks=true,
    linkcolor=blue,
    citecolor=blue,
    filecolor=magenta,      
    urlcolor=cyan,
    pdftitle={Coloring graphs as complete graph invariants},
    pdfpagemode=FullScreen,
    }

\newcommand{\set}[1]{\{#1\}}

\newtheorem{theorem}{Theorem}[section]
\newtheorem{alg}{Algorithm}

\newtheorem*{theorem*}{Theorem}
\newtheorem*{conjecture*}{Conjecture}

\allowdisplaybreaks
\theoremstyle{definition}
\newtheorem{definition}[theorem]{Definition}
\newtheorem{algorithm}[alg]{Algorithm}

\newtheorem{lemma}[theorem]{Lemma}

\newtheorem{corollary}[theorem]{Corollary}

\theoremstyle{remark}
\newtheorem{remark}[theorem]{Remark}

\numberwithin{equation}{section}

\newcommand{\abs}[1]{\lvert#1\rvert}

\newcommand{\C}{\mathcal C}

\makeatletter
\@namedef{subjclassname@2020}{%
  \textup{2020} Mathematics Subject Classification}
  \makeatother

\title
{Coloring graphs as complete graph invariants}

\author{Shamil Asgarli}
\address{Department of Mathematics and Computer Science \\ Santa Clara University \\ 500 El Camino Real \\ Santa Clara, CA 95053}
\email{sasgarli@scu.edu}

\author{Sara Krehbiel}
\address{Department of Mathematics and Computer Science \\ Santa Clara University \\ 500 El Camino Real \\ Santa Clara, CA 95053}
\email{skrehbiel@scu.edu}

\author{Howard W. Levinson}
\address{Department of Computer Science \\ Oberlin College \\ 10 N Professor St \\ Oberlin, OH 44074}
\email{hlevinso@oberlin.edu}

\subjclass[2020]{Primary: 05C15; Secondary: 05C60, 05C85}
\keywords{coloring graphs, reconfiguration systems, graph invariants}

\begin{document}

\begin{abstract} We investigate the extent to which the $k$-coloring graph $\mathcal{C}_{k}(G)$ uniquely determines the base graph $G$ and the number of colors $k$. The vertices of $\mathcal{C}_{k}(G)$ are the proper $k$-colorings of $G$, and edges connect colorings that differ on exactly one vertex. There are nonisomorphic graphs $G_1$ and $G_2$ with isomorphic coloring graphs, so $\mathcal{C}_{k}(G)$ is not a complete invariant in general. However, for color palettes with \emph{surplus colors} (when the number of colors $k$ is greater than the chromatic number), we prove that the coloring graph \textit{is} a complete invariant. Specifically, provided that $k_1 > \chi(G_1)$, we show that $\mathcal{C}_{k_1}(G_1)\cong \mathcal{C}_{k_2}(G_2)$ implies $G_1\cong G_2$ and $k_1=k_2$. Thus, there is a natural bijection between pairs $(G, k)$ with $k > \chi(G)$ and their coloring graphs $\mathcal{C}_k(G)$. Furthermore, no coloring graph of the form $\mathcal{C}_{\chi(G)}(G)$ is isomorphic to a coloring graph with surplus colors. Our constructive proof provides a method to decide whether a coloring graph is generated with surplus colors, although the resulting algorithms are inefficient.
\end{abstract}

\maketitle

\section{Introduction}

The difficulty of the graph isomorphism problem motivates the study of \emph{graph invariants}: properties preserved under isomorphism. Graph colorings provide a rich source of such invariants. A \emph{proper $k$-coloring} of a graph $G=(V, E)$ is a function $c\colon V \to \{1, \dots, k\}$ such that $c(u) \neq c(v)$ whenever $uv \in E$.  The minimum such $k$ is the \emph{chromatic number}, denoted $\chi(G)$. One graph invariant is the \emph{chromatic polynomial} $\pi_G(k)$, which counts the number of proper $k$-colorings. This paper focuses on a richer structure: the \emph{$k$-coloring graph}, denoted $\C_k(G)$. The vertices of $\C_k(G)$ are the proper $k$-colorings of $G$, and two vertices are adjacent if the corresponding colorings differ on exactly one vertex of $G$. This construction places $\C_k(G)$ within the broader context of a \emph{reconfiguration graph}, where edges represent a minimal transition between valid states.

We study the conditions under which $\C_k(G)$ is a complete graph invariant, uniquely determining both $G$ (up to isomorphism) and the number of colors $k$. If $k < \chi(G)$, then $\C_k(G)$ is the null graph and yields no information. If $k=\chi(G)$, then $\C_k(G)$ is not a complete invariant in general. For instance, any two uniquely $k$-colorable graphs share the same $k$-coloring graph $N_{k!}$ (where $N_r$ denotes the graph with $r$ vertices and $0$ edges). This leaves the third case $k>\chi(G)$, namely, when there are \emph{surplus colors}. The question of complete invariance of coloring graphs was formalized in our joint work \cite{AKLR25} with Russell:

\begin{center}
Does there exist a function $f\colon \mathsf{Graphs} \to \mathbb{N}$ such that $G$ can be uniquely reconstructed from $\mathcal{C}_{f(G)}(G)$?
\end{center}

Hogan, Scott, Tamitegama, and Tan \cite{HSTT24} answered this question affirmatively by showing that the function $f(G)=5 |V(G)|^2+1$ works. We refine their result by proving that $f(G)=\chi(G)+1$ suffices. Our main theorem asserts an even stronger conclusion: coloring graphs with surplus colors are distinguishable from all other coloring graphs, including those of the form $\mathcal{C}_{\chi(G)}(G)$.

\medskip 

\begin{theorem}\label{thm:main-intro}
If an abstract graph $\mathcal{C}$ satisfies  $\C\cong\C_k(G)$ for some graph $G$ with $k>\chi(G)$, then for \emph{any} graph $G'$ and positive integer $k'$ with $\C\cong\C_{k'}(G')$, we have $G\cong G'$ and $k=k'$.
\end{theorem}

\medskip 

We emphasize that Theorem~\ref{thm:main-intro} is logically equivalent to the combination of two separate results:

\begin{enumerate}[label=(\alph*),ref=\alph*]
    \item \label{item:intro-thm-easy} If $\mathcal{C}_{k_1}(G_1) \cong \mathcal{C}_{k_2}(G_2)$ with $k_1>\chi(G_1)$ and $k_2>\chi(G_2)$, then $G_1\cong G_2$ and $k_1=k_2$. 
    \item \label{item:intro-thm-hard} No two graphs $G$ and $G'$ can satisfy $\mathcal{C}_{k}(G) \cong \mathcal{C}_{\chi'}(G')$ with $k>\chi(G)$ and $\chi'=\chi(G')$.
\end{enumerate}

Part~\eqref{item:intro-thm-hard} turns out to be much more difficult than \eqref{item:intro-thm-easy}. We prove \eqref{item:intro-thm-easy} in Theorem~\ref{thm:reconstruction} using a reconstruction procedure outlined in Algorithm~\ref{algo:link-vertex}. Proving~\eqref{item:intro-thm-hard} (in Theorem~\ref{thm:main}) requires two additional algorithms and several other novel ideas. We also exhibit that $k=\chi(G)$ is not sufficient for $\mathcal{C}_k(G)$ to be a complete invariant by proving the following result. 

\begin{theorem}\label{thm:townhouse}
For every $\chi_1,\chi_2\ge 3$, there exist graphs $G_1,G_2$ with $\chi(G_1)=\chi_1$ and $\chi(G_2)=\chi_2$ such that $\C_{\chi_1}(G_1)\cong \C_{\chi_2}(G_2)$. 
\end{theorem}

The proof of Theorem~\ref{thm:townhouse} relies on a family of {\em $n$-townhouse graphs}, $\operatorname{TH}_n$, whose $3$-coloring graph $\mathcal{C}_3(\operatorname{TH}_n)$ is a disjoint union of six paths. By modifying these graphs, we can construct desired graphs $G_1$ and $G_2$ that satisfy the theorem's conclusion.

\begin{center}
\begin{tikzpicture}
    \coordinate (A1) at (0,0);
    \coordinate (B1) at (1,0);
    \coordinate (C1) at (1,1);
    \coordinate (D1) at (0,1);
    \coordinate (R1) at (0.5,1.5); 

    \coordinate (A2) at (1,0);
    \coordinate (B2) at (2,0);
    \coordinate (C2) at (2,1);
    \coordinate (D2) at (1,1);
    \coordinate (R2) at (1.5,1.5); 

    \draw (A1) -- (B1) -- (C1) -- (D1) -- cycle;
    \draw (D1) -- (R1) -- (C1);
    \draw (A2) -- (B2) -- (C2) -- (D2) -- cycle;
    \draw (D2) -- (R2) -- (C2);
    \draw (C1) -- (D2);
    \draw (R1) -- (R2);
    
    \foreach \point in {A1, B1, C1, D1, R1, A2, B2, C2, D2, R2}
        \fill (\point) circle (2pt);

    \node at (1, -0.5) {$\operatorname{TH}_2$};

    \coordinate (A3) at (4,0);
    \coordinate (B3) at (5,0);
    \coordinate (C3) at (5,1);
    \coordinate (D3) at (4,1);
    \coordinate (R3) at (4.5,1.5);

    \coordinate (A4) at (5,0);
    \coordinate (B4) at (6,0);
    \coordinate (C4) at (6,1);
    \coordinate (D4) at (5,1);
    \coordinate (R4) at (5.5,1.5); 

    \coordinate (A5) at (6,0);
    \coordinate (B5) at (7,0);
    \coordinate (C5) at (7,1);
    \coordinate (D5) at (6,1);
    \coordinate (R5) at (6.5,1.5);

    \draw (A3) -- (B3) -- (C3) -- (D3) -- cycle;
    \draw (D3) -- (R3) -- (C3);
    \draw (A4) -- (B4) -- (C4) -- (D4) -- cycle;
    \draw (D4) -- (R4) -- (C4);
    \draw (A5) -- (B5) -- (C5) -- (D5) -- cycle;
    \draw (D5) -- (R5) -- (C5);
    \draw (C3) -- (D4);
    \draw (C4) -- (D5);
    \draw (R3) -- (R4);
    \draw (R4) -- (R5);
    
    \foreach \point in {A3, B3, C3, D3, R3, A4, B4, C4, D4, R4, A5, B5, C5, D5, R5}
        \fill (\point) circle (2pt);

    \node at (5.5, -0.5) {$\operatorname{TH}_3$};

    \coordinate (A6) at (9,0);
    \coordinate (B6) at (10,0);
    \coordinate (C6) at (10,1);
    \coordinate (D6) at (9,1);
    \coordinate (R6) at (9.5,1.5); 

    \coordinate (A7) at (10,0);
    \coordinate (B7) at (11,0);
    \coordinate (C7) at (11,1);
    \coordinate (D7) at (10,1);
    \coordinate (R7) at (10.5,1.5); 

    \coordinate (A8) at (11,0);
    \coordinate (B8) at (12,0);
    \coordinate (C8) at (12,1);
    \coordinate (D8) at (11,1);
    \coordinate (R8) at (11.5,1.5); 
    
    \coordinate (A9) at (12,0);
    \coordinate (B9) at (13,0);
    \coordinate (C9) at (13,1);
    \coordinate (D9) at (12,1);
    \coordinate (R9) at (12.5,1.5); 

    \draw (A6) -- (B6) -- (C6) -- (D6) -- cycle;
    \draw (D6) -- (R6) -- (C6);
    \draw (A7) -- (B7) -- (C7) -- (D7) -- cycle;
    \draw (D7) -- (R7) -- (C7);
    \draw (A8) -- (B8) -- (C8) -- (D8) -- cycle;
    \draw (D8) -- (R8) -- (C8);
    \draw (A9) -- (B9) -- (C9) -- (D9) -- cycle;
    \draw (D9) -- (R9) -- (C9);
    \draw (C6) -- (D7);
    \draw (C7) -- (D8);
    \draw (C8) -- (D9);
    \draw (R6) -- (R7);
    \draw (R7) -- (R8);
    \draw (R8) -- (R9);

    \foreach \point in {A6, B6, C6, D6, R6, A7, B7, C7, D7, R7, A8, B8, C8, D8, R8, A9, B9, C9, D9, R9}
        \fill (\point) circle (2pt);

    \node at (11, -0.5) {$\operatorname{TH}_4$};
\end{tikzpicture}
\end{center}

We call a coloring graph $\mathcal{C}_{\chi(G)}(G)$ a \emph{$\chi$-coloring graph}. If $k>\chi(G)$, then $\mathcal{C}_k(G)$ is called a \emph{surplus coloring graph}. To prove Theorem~\ref{thm:main-intro}, we find structures that are unique to surplus coloring graphs. We define the concept of an \emph{abstract link vertex}, whose presence can be algorithmically checked in an arbitrary coloring graph (without knowing $G$ or $k$). We develop three algorithms: the first reconstructs the base graph’s adjacency matrix using a square-counting technique (inspired by \cite{HSTT24}); the second extracts the underlying coloring partition at a link vertex; and the third organizes equivalent link vertices into a labeled link graph. Collectively, these algorithms prove that the presence of an abstract link vertex is a certificate for surplus coloring graphs. A graph containing such a vertex is generated by at most one pair $(G,k)$ with $k>\chi(G)$, \emph{and} by no pair $(G',k')$ where $k' = \chi(G')$. 

\begin{remark}
Theorem~\ref{thm:reconstruction} (a weaker version of Theorem~\ref{thm:main-intro}) was independently and concurrently discovered by Berthe, Brosse, Hearn, van den Heuvel, Hoppenot, and Pierron \cite{BBHvdHHP25} in their arXiv preprint posted on April 28, 2025. The first version of our paper (arXiv:2504.20978) was posted on April 29, 2025. The approach for reconstructing the base graph is similar, but the two papers also diverge in other aspects. Specifically, our Theorem~\ref{thm:main-intro} is stronger in its conclusion. Namely, \cite{BBHvdHHP25}*{Theorem 1.2} proves part~\eqref{item:intro-thm-easy}, while our methods establish both part~\eqref{item:intro-thm-easy} and part~\eqref{item:intro-thm-hard}, 
with the proof of part~\eqref{item:intro-thm-hard} being more sophisticated and forming the bulk of this paper. Moreover, Theorem~\ref{thm:townhouse} and its proof are novel to this work. On the other hand, Berthe et al. establish the insufficiency of $k=\chi(G)$ by constructing an arbitrarily large collection $\mathcal{F}$ of nonisomorphic graphs all with chromatic number $\chi$ and no frozen vertices such that $\mathcal{C}_{\chi}(G_1)\cong \mathcal{C}_{\chi}(G_2)$ for all $G_1, G_2\in \mathcal{F}$ provided that $\chi\geq 6$ \cite{BBHvdHHP25}*{Theorem 1.4}. Their paper also contains further results concerning reconfiguration graphs of Kempe-colorings and independent sets. \end{remark}

We conclude the introduction with a brief discussion of related work and the context for our main question. Reconfiguration graphs have been studied extensively from structural and algorithmic perspectives \cites{vdH13,Nis18}, with particular attention to graph colorings \cites{CVJ08, BC09, BFHRS16, BB18}. The question we consider here also fits naturally within the broader theme of graph reconstruction. Classical reconstruction asks whether a graph is recoverable from its vertex-deleted deck \cite{LS16}, a problem closely tied to graph isomorphism \cite{HHRT07}. Beyond this, one may ask when a graph is determined by other associated structures such as line graphs \cite{Whi32-line-graphs}, graph powers \cite{RH60}, chromatic polynomial \cites{KT90,KT97}, Tutte polynomial \cite{dMN04}, or spectrum \cite{vDH03}.

\textbf{Structure of the paper.} Section~\ref{sec:def} reviews fundamental definitions and introduces our notation. We prove Theorem~\ref{thm:townhouse} in Section~\ref{sec:chromatic-level}. The proof of our main result, Theorem~\ref{thm:main-intro}, is developed in Section~\ref{sec:unique-if-surplus}, which is organized into four subsections with a series of supporting lemmas. In Section~\ref{subsect:reconstruction}, we introduce Algorithm~\ref{algo:link-vertex} and prove Theorem~\ref{thm:reconstruction} (part~\eqref{item:intro-thm-easy} of Theorem~\ref{thm:main-intro}). Section~\ref{subsect:partition} and Section~\ref{subsect:LLG} introduce Algorithm~\ref{algo:partitions} and Algorithm~\ref{algo:LLG}, respectively. The final Section~\ref{subsect:final} proves Theorem~\ref{thm:main} (part~\eqref{item:intro-thm-hard} of Theorem~\ref{thm:main-intro}). 

\section{Definitions and Notation}\label{sec:def}

Given a graph $G$ and $k\geq \chi(G)$, the \emph{$k$-coloring graph} of $G$, denoted $\C_k(G)$, is defined as follows. The vertices of $\C_k(G)$ represent proper $k$-colorings of $G$. Two vertices in $\C_k(G)$ (representing two distinct $k$-colorings) are adjacent if the two colorings differ on exactly one vertex $v\in V(G)$. Formally, we have a bijection $\Phi\colon V(\mathcal{C}_k(G))\to  \{\text{proper } k\text{-colorings of } G \}$ defined by $\Phi(\alpha)=\alpha_{G}$ such that $\alpha\beta\in E(\mathcal{C}_k(G))$ if and only if $\alpha_G(v)\neq \beta_G(v)$ for exactly one $v\in V(G)$. That is, we use $\alpha$ to denote a vertex in $\mathcal{C}_k(G)$ and $\alpha_G\colon V(G)\to [k]$ for the corresponding coloring of $G$. For simplicity, we typically omit the symbol $\Phi$ and use the shorthand $\alpha\mapsto \alpha_G$ for this mapping. 

While $\C_k(G)$ is fundamentally an abstract graph, we can use the underlying bijection between $V(\C_k(G))$ and the $k$-colorings of $G$ to produce an {\em edge-labeled} coloring graph. Each edge is labeled to show the unique vertex $v\in V(G)$ where the two corresponding colorings differ and the specific colors assigned to $v$ by each coloring. For example, an edge $\alpha\beta\in E(\C_k(G))$ might have the label $cvc'$, which signifies that $\alpha_G(v)=c, \beta_G(v)=c'$, and $\alpha_G(u)=\beta_G(u)$ for all $u\in V(G)\backslash\set{v}$. We emphasize that we have access to neither edge labelings nor the $\alpha\mapsto\alpha_G$ mapping when given an abstract coloring graph $\mathcal{C}\cong \mathcal{C}_{k}(G)$.

 Figure~\ref{fig:coloring-graph-labeling} depicts the relationship between $3$-colorings of $P_3$ by labeling the vertices and edges of $\C\cong\C_3(P_3)$. The red vertices in $\mathcal{C}$, which use the minimum number of colors and are termed \emph{abstract link vertices}, play a central role in our paper. In general, when a graph $G$ has $d$ connected components, the coloring graph $\C_k(G)$ with $k>\chi(G)$ admits a natural action of $S_k^d$, the $d$-fold product of the symmetric group $S_k$ on $k$ symbols. 
 We will show that the orbit-equivalent link vertices in the coloring graph are connected via an identifiable sequence of hypercubes. For $\mathcal{C}_3(P_3)$, the six link vertices form a single equivalence class. A \emph{labeled link graph} organizes a link vertex equivalence class into a new graph whose edges are hypercubes in $\mathcal{C}_k(G)$. These graphs, formalized in Section~\ref{sec:unique-if-surplus}, help establish $P_3$ as the unique base graph yielding $\mathcal{C}$ as a coloring graph.

\begin{figure}[ht]
    
 \begin{center}
     \begin{tikzpicture}[baseline=0cm, scale=0.44]
\draw[style=thick] (-2,0)--(2,0);
\draw[radius=.5,fill=white](-2,0)circle node {u};
\draw[radius=.5,fill=white](0,0)circle node{v};
\draw[radius=.5,fill=white](2,0)circle node{w};
    \end{tikzpicture} 
 \end{center}
 
    \begin{tabular}{cc}
    \begin{tikzpicture}[baseline=0cm, scale=0.44]
\node (a) at (0,0) {};
\node (b) at (-2,-2) {};
\node (c) at (0,-4) {};
\node (d) at (2,-2) {};
\node (e) at (-5,-6) {};
\node (f) at (5,-6) {};
\draw[style=thick] (b)--(e)--(-3,-8)--(-5,-10)--(5,-10)--(7,-8)--(5,-6)--(d);
\draw[style=thick] (e)--(-7,-8)--(-5,-10);
\draw[style=thick] (5,-6)--(3,-8)--(5,-10);
\draw[style=thick] (-5,-6)--(-3,-8)--(-5,-10);
\draw[style=thick] (a)--(b)--(c)--(d)--(a);
\draw[fill= white] (-1.5,-.5) rectangle (1.5,.5);
\draw[fill=white](-1,0)--(1,0);
\draw[radius=.35,fill=white](-1,0)circle node {1};
\draw[radius=.35,fill=white](0,0)circle node{2};
\draw[radius=.35,fill=white](1,0)circle node{3};

\draw[red,style=thick,fill=white] (-3.5,-2.5) rectangle (-.5,-1.5);
\draw[style=thick](-3, -2)--(-1,-2);
\draw[radius=.35,fill=white](-3,-2)circle node{3};
\draw[radius=.35,fill=white](-2,-2)circle node{2};
\draw[radius=.35,fill=white](-1,-2)circle node{3};

\draw[red,style=thick,fill=white] (3.5,-2.5) rectangle (.5,-1.5);
\draw[style=thick](3, -2)--(1,-2);
\draw[radius=.35,fill=white](3,-2)circle node{1};
\draw[radius=.35,fill=white](2,-2)circle node{2};
\draw[radius=.35,fill=white](1,-2)circle node{1};

\draw[fill=white] (-1.5,-4.5) rectangle (1.5,-3.5);
\draw[style=thick](-1,-4)--(1,-4);
\draw[radius=.35,fill=white](-1,-4)circle node{3};
\draw[radius=.35,fill=white](0,-4)circle node{2};
\draw[radius=.35,fill=white](1,-4)circle node{1}; 

\draw[red,style=thick,fill=white] (-6.5,-6.5) rectangle (-3.5,-5.5);
\draw[style=thick](-6,-6)--(-4,-6);
\draw[radius=.35,fill=white](-6,-6)circle node{3};
\draw[radius=.35,fill=white](-5,-6)circle node{1};
\draw[radius=.35,fill=white](-4,-6)circle node{3};

\draw[fill=white] (-8.5,-8.5) rectangle (-5.5,-7.5);
\draw[style=thick](-8, -8)--(-6,-8);
\draw[radius=.35,fill=white](-8,-8)circle node{2};
\draw[radius=.35,fill=white](-7,-8)circle node{1};
\draw[radius=.35,fill=white](-6,-8)circle node{3};

\draw[fill=white] (-4.5,-8.5) rectangle (-1.5,-7.5);
\draw[style=thick](-2, -8)--(-4,-8);
\draw[radius=.35,fill=white](-2,-8)circle node{2};
\draw[radius=.35,fill=white](-3,-8)circle node{1};
\draw[radius=.35,fill=white](-4,-8)circle node{3};

\draw[red,style=thick,fill=white] (-6.5,-10.5) rectangle (-3.5,-9.5);
\draw[style=thick](-6,-10)--(-4,-10);
\draw[radius=.35,fill=white](-6,-10)circle node{2};
\draw[radius=.35,fill=white](-5,-10)circle node{1};
\draw[radius=.35,fill=white](-4,-10)circle node{2}; 

\draw[red,style=thick,fill=white] (3.5,-6.5) rectangle (6.5,-5.5);
\draw[style=thick](4,-6)--(6,-6);
\draw[radius=.35,fill=white](4,-6)circle node{1};
\draw[radius=.35,fill=white](5,-6)circle node{3};
\draw[radius=.35,fill=white](6,-6)circle node{1};

\draw[fill=white] (1.5,-8.5) rectangle (4.5,-7.5);
\draw[style=thick](2, -8)--(4,-8);
\draw[radius=.35,fill=white](2,-8)circle node{2};
\draw[radius=.35,fill=white](3,-8)circle node{3};
\draw[radius=.35,fill=white](4,-8)circle node{1};

\draw[fill=white] (5.5,-8.5) rectangle (8.5,-7.5);
\draw[style=thick](8, -8)--(6,-8);
\draw[radius=.35,fill=white](6,-8)circle node{1};
\draw[radius=.35,fill=white](7,-8)circle node{3};
\draw[radius=.35,fill=white](8,-8)circle node{2}; 

\draw[red,style=thick,fill=white] (3.5,-10.5) rectangle (6.5,-9.5);
\draw[style=thick](4,-10)--(6,-10);
\draw[radius=.35,fill=white](4,-10)circle node{2};
\draw[radius=.35,fill=white](5,-10)circle node{3};
\draw[radius=.35,fill=white](6,-10)circle node{2}; 
\end{tikzpicture}\begin{tikzpicture}[baseline=0cm, scale=0.5]
\end{tikzpicture} & 
\begin{tikzpicture}[baseline=0cm, scale=0.44]
\node (a) at (0,0) {};
\node (b) at (-2,-2) {};
\node (c) at (0,-4) {};
\node (d) at (2,-2) {};
\node (e) at (-5,-6) {};
\node (f) at (5,-6) {};
\draw[style=thick] (b)--(e)node[above,pos=0.5,sloped]{1v2}--(-3,-8)node[above,pos=0.5,sloped]{3w2}--(-5,-10)node[below,pos=0.5,sloped]{\phantom{x}2u3}--(5,-10)node[above,pos=0.5,sloped]{1v3}--(7,-8)node[below,pos=0.5,sloped]{2u1}--(5,-6)node[above,pos=0.5,sloped]{1w2}--(d)node[above,pos=0.5,sloped]{2v3};
\draw[style=thick] (e)--(-7,-8)node[above,pos=0.5,sloped]{2u3}--(-5,-10)node[below,pos=0.5,sloped]{3w2};
\draw[style=thick] (5,-6)--(3,-8)node[above,pos=0.5,sloped]{2u1}--(5,-10)node[below,pos=0.5,sloped]{1w2\phantom{x}};
\draw[style=thick] (a)--(b)node[above,pos=0.5,sloped]{3u1}--(c)node[below,pos=0.5,sloped]{3w1}--(d)node[below,pos=0.5,sloped]{3u1}--(a)node[above,pos=0.5,sloped]{3w1};
\draw[radius=.35,fill=white](0,0)circle node{};
\draw[radius=.35,fill=red](-2,-2)circle node{};
\draw[radius=.35,fill=red](2,-2)circle node{};
\draw[radius=.35,fill=white](0,-4)circle node{};
\draw[radius=.35,fill=red](-5,-6)circle node{};
\draw[radius=.35,fill=white](-7,-8)circle node{};
\draw[radius=.35,fill=white](-3,-8)circle node{};
\draw[radius=.35,fill=red](-5,-10)circle node{};
\draw[radius=.35,fill=red](5,-6)circle node{};
\draw[radius=.35,fill=white](3,-8)circle node{};
\draw[radius=.35,fill=white](7,-8)circle node{};
\draw[radius=.35,fill=red](5,-10)circle node{};
\draw[draw=none,fill= white] (-8.5,-.5) rectangle (-5.5,.5);
\draw[draw=none,fill= white] (8.5,-.5) rectangle (5.5,.5);
\end{tikzpicture}
\end{tabular}

\vspace{1em}

\begin{center}\usetikzlibrary{fit}

    \begin{tikzpicture}[baseline=0cm, scale=0.9,
    vertex/.style={},
    edge_label/.style={
        sloped,   
        midway,   
        auto,    
        font=\small 
    },
    edge_style/.style={
        <->, 
        >=Latex 
    }, every fit/.style={draw,inner sep=-1.5pt}
  ]
  
  \def\hexagonradius{2.1cm}

  \node[vertex] (alpha)  at (120:\hexagonradius) {$323$};
  \node[vertex] (alpha1) at (60:\hexagonradius) {$121$};
  \node[vertex] (beta)   at (0:\hexagonradius) {$131$};
  \node[vertex] (gamma)  at (-60:\hexagonradius) {$232$};
  \node[vertex] (delta)  at (-120:\hexagonradius) {$212$};
  \node[vertex] (alpha2) at (180:\hexagonradius) {$313$}; 
  \draw[edge_style] (alpha) -- node[edge_label, swap] {$3 V_2 1$} (alpha1);
  \draw[edge_style] (alpha1) -- node[edge_label, swap] {$2V_13$} (beta);
  \draw[edge_style] (beta) -- node[edge_label] {$2V_21$} (gamma);
  \draw[edge_style] (gamma) -- node[edge_label] {$1 V_1 3$} (delta);
  \draw[edge_style] (delta) -- node[edge_label] {$3V_22$} (alpha2);
  \draw[edge_style] (alpha2) -- node[edge_label, swap] {$1 V_1 2$} (alpha);
  \node[draw, red, style=thick, fit=(alpha)] {};
  \node[draw, red, style=thick, fit=(alpha1)] {};
  \node[draw, red, style=thick, fit=(alpha2)] {};
  \node[draw, red, style=thick, fit=(beta)] {};
  \node[draw, red, style=thick, fit=(gamma)] {};
  \node[draw, red, style=thick, fit=(delta)] {};
\end{tikzpicture}
\end{center}

\caption{Three organizations of the 3-colorings of $P_3$, with link colorings highlighted red. Top left is coloring graph $\mathcal C_3(P_3)$ with vertex labels indicating the underlying bijection $\Phi\colon V(\C_3(P_3))\to \set{\text{proper 3-colorings of $P_3$}}$; top right depicts the corresponding edge labelings; bottom center presents the labeled link graph organizing the 2-colorings consistent with bipartition $V_1=\set{v},V_2=\set{u,w}$ of $V(P_3)$.}\label{fig:coloring-graph-labeling}
\end{figure}

We use standard notation for common graphs (see \cite{Wes01}): $K_n$, $C_n$, and $P_n$ are the complete, cycle, and path graphs on $n$ vertices, respectively; $K_{r, s}$ is the complete bipartite graph with two parts of cardinality $r$ and $s$; $N_r$ is the graph with $r$ vertices and $0$ edges. The \emph{house graph} is obtained by adding a chord to a $5$-cycle. 

Next, we describe notation related to integer and set partitions. Let $\mathbb{N}$ denote the set of positive integers. For $k\in\mathbb{N}$, let $[k]$ be the set $\{1, 2, \dots, k\}$. A \emph{partition} $\lambda=(\lambda_1, \dots, \lambda_t)$ of $n\in\mathbb{N}$, denoted $\lambda\vdash n$, means that $\lambda_i\in\mathbb{N}$ for all $i$ and $\lambda_1+ \dots + \lambda_t = n$. For example, $(3, 1, 7, 1, 6)\vdash 18$.

For two disjoint subsets $A$ and $B$ of a set $X$, their union $A\cup B$ is denoted by $A\uplus B$ to emphasize disjointness. A \emph{set partition} of $X$ is a tuple $(X_1, \dots, X_t)$ of pairwise disjoint subsets whose union is $X$, so $X=X_1\uplus X_2\uplus \dots \uplus X_t$. By convention, each part $X_i$ is nonempty. It is sometimes convenient to initially define parts that may be empty; such empty parts are then disregarded in the final partition. For instance, a proper $k$-coloring $\alpha\colon V(G)\to [k]$ of a graph $G$ induces a partition $P=(P_1, P_2, \dots, P_k)$ of $V(G)$ where each part $P_c=\alpha^{-1}(c)$ consists of vertices colored $c\in [k]$. If the coloring $\alpha$ uses fewer than $k$ colors, some parts $P_c$ in this partition are empty. The number of nonempty parts in a set partition $P$ is its \emph{length}, $\ell(P)$.

\section{Coloring Graphs at the Chromatic Level May Have Multiple Base Graphs}\label{sec:chromatic-level}

Our main result, established in Section~\ref{sec:unique-if-surplus}, is twofold: first, surplus coloring graphs uniquely determine their base graphs; second, surplus coloring graphs are structurally distinct from $\chi$-coloring graphs. In contrast, two nonisomorphic graphs can produce isomorphic $\chi$-coloring graphs. For instance, all uniquely $k$-colorable graphs share the same $k$-coloring graph, namely the graph $N_{k!}$ consisting of $k!$ isolated vertices. This observation leads to a natural question: can $\mathcal{C}_{\chi(G_1)}(G_1)$ be isomorphic to $\mathcal{C}_{\chi(G_2)}(G_2)$ when $\chi(G_1) \neq \chi(G_2)$? In this section, we affirmatively answer this question by proving Theorem~\ref{thm:townhouse}. We begin by constructing a $3$-chromatic graph with a prescribed number of vertices in its $3$-coloring graph.

\begin{lemma}\label{lem:townhouse} Let $r\geq 1$. There exists a graph $G$ with $\chi(G)=3$ such that $\mathcal{C}_3(G)$ has exactly $6r$ vertices.
\end{lemma}

\begin{proof} For the cases $r=1$ and $r=2$, the graphs $K_3$ and $K_{1, 3}+e$ (\emph{the paw graph}) provide the required examples, as can be verified by evaluating their chromatic polynomials at $k=3$. Assume $r\geq 3$. Consider the \emph{townhouse graph} $\operatorname{TH}_{n}$ for $n\ge 1$, which consists of a sequence of $n$ 5-node house graphs in which consecutive houses share a wall and have an additional edge connecting the tops of their gables. By definition, $\operatorname{TH}_1$ is the house graph. Below, we illustrate townhouse graphs for $n = 2, 3, 4$. 

\begin{center}
\begin{tikzpicture}
    \coordinate (A1) at (0,0);
    \coordinate (B1) at (1,0);
    \coordinate (C1) at (1,1);
    \coordinate (D1) at (0,1);
    \coordinate (R1) at (0.5,1.5); 
    \coordinate (A2) at (1,0);
    \coordinate (B2) at (2,0);
    \coordinate (C2) at (2,1);
    \coordinate (D2) at (1,1);
    \coordinate (R2) at (1.5,1.5); 
    \draw (A1) -- (B1) -- (C1) -- (D1) -- cycle;
    \draw (D1) -- (R1) -- (C1);
    \draw (A2) -- (B2) -- (C2) -- (D2) -- cycle;
    \draw (D2) -- (R2) -- (C2);
    \draw (C1) -- (D2);
    \draw (R1) -- (R2);
    \foreach \point in {A1, B1, C1, D1, R1, A2, B2, C2, D2, R2}
        \fill (\point) circle (2pt);
    \node at (1, -0.5) {$\operatorname{TH}_2$};
    \coordinate (A3) at (4,0);
    \coordinate (B3) at (5,0);
    \coordinate (C3) at (5,1);
    \coordinate (D3) at (4,1);
    \coordinate (R3) at (4.5,1.5); 
    \coordinate (A4) at (5,0);
    \coordinate (B4) at (6,0);
    \coordinate (C4) at (6,1);
    \coordinate (D4) at (5,1);
    \coordinate (R4) at (5.5,1.5);
    \coordinate (A5) at (6,0);
    \coordinate (B5) at (7,0);
    \coordinate (C5) at (7,1);
    \coordinate (D5) at (6,1);
    \coordinate (R5) at (6.5,1.5); 
    \draw (A3) -- (B3) -- (C3) -- (D3) -- cycle;
    \draw (D3) -- (R3) -- (C3);
    \draw (A4) -- (B4) -- (C4) -- (D4) -- cycle;
    \draw (D4) -- (R4) -- (C4);
    \draw (A5) -- (B5) -- (C5) -- (D5) -- cycle;
    \draw (D5) -- (R5) -- (C5);
    \draw (C3) -- (D4);
    \draw (C4) -- (D5);
    \draw (R3) -- (R4);
    \draw (R4) -- (R5);
    \foreach \point in {A3, B3, C3, D3, R3, A4, B4, C4, D4, R4, A5, B5, C5, D5, R5}
        \fill (\point) circle (2pt);
    \node at (5.5, -0.5) {$\operatorname{TH}_3$};
    \coordinate (A6) at (9,0);
    \coordinate (B6) at (10,0);
    \coordinate (C6) at (10,1);
    \coordinate (D6) at (9,1);
    \coordinate (R6) at (9.5,1.5); 
    \coordinate (A7) at (10,0);
    \coordinate (B7) at (11,0);
    \coordinate (C7) at (11,1);
    \coordinate (D7) at (10,1);
    \coordinate (R7) at (10.5,1.5); 
    \coordinate (A8) at (11,0);
    \coordinate (B8) at (12,0);
    \coordinate (C8) at (12,1);
    \coordinate (D8) at (11,1);
    \coordinate (R8) at (11.5,1.5);
    \coordinate (A9) at (12,0);
    \coordinate (B9) at (13,0);
    \coordinate (C9) at (13,1);
    \coordinate (D9) at (12,1);
    \coordinate (R9) at (12.5,1.5); 
    \draw (A6) -- (B6) -- (C6) -- (D6) -- cycle;
    \draw (D6) -- (R6) -- (C6);
    \draw (A7) -- (B7) -- (C7) -- (D7) -- cycle;
    \draw (D7) -- (R7) -- (C7);
    \draw (A8) -- (B8) -- (C8) -- (D8) -- cycle;
    \draw (D8) -- (R8) -- (C8);
    \draw (A9) -- (B9) -- (C9) -- (D9) -- cycle;
    \draw (D9) -- (R9) -- (C9);
    \draw (C6) -- (D7);
    \draw (C7) -- (D8);
    \draw (C8) -- (D9);
    \draw (R6) -- (R7);
    \draw (R7) -- (R8);
    \draw (R8) -- (R9);
    \foreach \point in {A6, B6, C6, D6, R6, A7, B7, C7, D7, R7, A8, B8, C8, D8, R8, A9, B9, C9, D9, R9}
        \fill (\point) circle (2pt);
    \node at (11, -0.5) {$\operatorname{TH}_4$};
\end{tikzpicture}
\end{center}

We claim that $\mathcal{C}_3(\operatorname{TH}_n)$ contains precisely $6(n+2)$ vertices. By letting $n = r - 2$, this completes the proof of the lemma. When constructing $\operatorname{TH}_{n}$ from $\operatorname{TH}_{n-1}$, three new vertices $u, v, w$ are added, with $v$ and $w$ serving as the \emph{roof vertices} as pictured below.

\begin{center}
\begin{tikzpicture}
    \coordinate (A1) at (0,0);
    \coordinate (B1) at (1,0);
    \coordinate (C1) at (1,1);
    \coordinate (D1) at (0,1);
    \coordinate (R1) at (0.5,1.5); 
    \coordinate (A2) at (1,0); 
    \coordinate (B2) at (2,0);
    \coordinate (C2) at (2,1);
    \coordinate (D2) at (1,1);
    \coordinate (R2) at (1.5,1.5); 
    \coordinate (A3) at (4,0); 
    \coordinate (B3) at (5,0);
    \coordinate (C3) at (5,1);
    \coordinate (D3) at (4,1);
    \coordinate (R3) at (4.5,1.5);
    \coordinate (A4) at (5,0);
    \coordinate (B4) at (6,0);
    \coordinate (C4) at (6,1);
    \coordinate (D4) at (5,1);
    \coordinate (R4) at (5.5,1.5);
    \draw (A1) -- (B1) -- (C1) -- (D1) -- cycle; 
    \draw (D1) -- (R1) -- (C1); 
    \draw (A2) -- (B2) -- (C2) -- (D2) -- cycle; 
    \draw (D2) -- (R2) -- (C2); 
    \draw (A3) -- (B3) -- (C3) -- (D3) -- cycle; 
    \draw (D3) -- (R3) -- (C3); 
    \draw (A4) -- (B4) -- (C4) -- (D4) -- cycle; 
    \draw (D4) -- (R4) -- (C4); 
    \draw (C1) -- (D2); 
    \draw (C3) -- (D4); 
    \draw (R1) -- (R2); 
    \draw (R3) -- (R4); 
    \node at (3,1) {\dots}; 
    \node at (3,1.5) {\dots};
    \node at (3, 0) {\dots}; 
    \foreach \point in {A1, B1, C1, D1, R1, A2, B2, C2, D2, R2, A3, B3, C3, D3, R3, A4, B4, C4, D4, R4}
        \fill (\point) circle (2pt);

    \node[above] at (R4) {$w$};
    \node[right] at (C4) {$v$};
    \node[right] at (B4) {$u$};

     \draw [decorate,decoration={brace,mirror,amplitude=10pt}]
        ([yshift=-6pt]A1.south west) -- ([yshift=-6pt]A4.south east)
        node [midway, below=10pt] {TH$_{n-1}$};
\end{tikzpicture}
\end{center}

We compute the number of $3$-colorings of $\operatorname{TH}_{n}$. First, there are $3$ choices for the color of vertex $v$ and $2$ choices for the color of vertex $w$. Once these colors are assigned, the colors of all the roof vertices (three per roof) are determined automatically. Hence, the number of $3$-colorings of $\operatorname{TH}_{n}$ equals $6 \cdot t_n$, where $t_n$ represents the number of $3$-colorings of the path graph $P_{n+1}$ with specific constraints described below.

Without loss of generality, suppose $v$ is colored $1$ and $w$ is colored $3$. The \emph{floor vertices} of the houses form a path graph $P_{n+1}$, consisting of $u_{n+1} \mathord{-} u_{n} \mathord{-} \cdots \mathord{-} u_2 \mathord{-} u_1$ where $u=u_1$. The value $t_n$ is the number of $3$-colorings of $P_{n+1}$ such that the color of $u_i$ differs from $i \pmod{3}$ for every $1 \leq i \leq n+1$. We use the set $\{1, 2, 3\}$ as representatives modulo $3$ to correspond to the color labels.

We claim that $t_n = n+2$. To prove this formula by induction on $n$, it suffices to show that $t_1 = 3$ and that $t_n = t_{n-1} + 1$ for all $n \geq 2$.

\textbf{Base case.} For $n = 1$, we count the colorings $c$ of $P_2 = u_2 \mathord{-} u_1$ such that $c(u_1) \neq 1$ and $c(u_2) \neq 2$. If $c(u_2) = 3$, then $c(u_1) = 2$ is the only valid choice. If $c(u_2) = 1$, then $c(u_1)$ can be either $2$ or $3$. Thus, there are $3$ valid colorings, giving $t_1 = 3$.
  
\textbf{Recursive formula.} By definition, $t_{n-1}$ is the number of valid colorings of $P_n = u_{n+1} \mathord{-} u_n \mathord{-} \cdots \mathord{-} u_2$. We extend these colorings to $P_{n+1}$ while ensuring that $c(u_1) \neq 1$. The $t_{n-1}$ valid colorings of $P_n$ split into two groups:

\begin{itemize}
    \item \textbf{Type A:} $c(u_2) = 3$. In this case, $c(u_1) = 2$ is the only possible extension, providing a unique valid coloring of $P_{n+1}$. 
    \item \textbf{Type B:} $c(u_2) = 1$. The constraints force a specific coloring pattern on $u_3, \dots, u_{n+1}$, determined by $c(u_i) = (i-1) \pmod{3}$. However, for $u_1$, we have two possible extensions: $c(u_1) = 2$ or $c(u_1) = 3$.
\end{itemize}

Thus, there is one additional valid coloring of Type B, giving $t_n = t_{n-1} + 1$. By induction, $t_n = n + 2$ for all $n \geq 1$. Consequently, $\operatorname{TH}_n$ has a total of $6(n+2)$ distinct $3$-colorings for each $n \geq 1$. \end{proof}

The proof of Lemma~\ref{lem:townhouse} can be adapted to fully describe the $3$-coloring graph of townhouse graphs. For $n\geq 1$, the graph $\mathcal{C}_3(\operatorname{TH}_n)$ is isomorphic to $6P_{n+2}$, a disjoint union of $6$ path graphs, each on $n+2$ vertices.

\begin{corollary}\label{cor:3-chromatic-isolated} For each $r\geq 1$, there exists a connected graph $G$ with $\chi(G)=3$ such that $\mathcal{C}_3(G)\cong N_{6r}$, the edgeless graph on $6r$ vertices.
\end{corollary}

\begin{proof} The cases $r = 1$ and $r = 2$ are addressed by considering the graphs $K_3$ and $F_2$ (the \emph{friendship graph}), which consists of two copies of $K_3$ joined at a common vertex. 

Assume $r\geq 3$ and set $n=r-2\geq 1$. We modify the townhouse graph $\operatorname{TH}_n$ from the proof of Lemma~\ref{lem:townhouse} by attaching new \emph{basement vertices} to each house to ensure that every vertex belongs to a triangle. This construction yields a new graph $\underline{\operatorname{TH}_{n}}$, illustrated below for $n = 3$. 

\begin{center}
    
\begin{tikzpicture}
    \coordinate (A3) at (4,0);
    \coordinate (B3) at (5,0);
    \coordinate (C3) at (5,1);
    \coordinate (D3) at (4,1);
    \coordinate (R3) at (4.5,1.5); 
    \coordinate (A4) at (5,0);
    \coordinate (B4) at (6,0);
    \coordinate (C4) at (6,1);
    \coordinate (D4) at (5,1);
    \coordinate (R4) at (5.5,1.5); 
    \coordinate (A5) at (6,0);
    \coordinate (B5) at (7,0);
    \coordinate (C5) at (7,1);
    \coordinate (D5) at (6,1);
    \coordinate (R5) at (6.5,1.5); 
    \coordinate (Q1) at (4.5,-0.5); 
    \coordinate (Q2) at (5.5,-0.5); 
    \coordinate (Q3) at (6.5,-0.5); 
    \draw (A3) -- (B3) -- (C3) -- (D3) -- cycle; 
    \draw (D3) -- (R3) -- (C3); 
    \draw (A4) -- (B4) -- (C4) -- (D4) -- cycle; 
    \draw (D4) -- (R4) -- (C4); 
    \draw (A5) -- (B5) -- (C5) -- (D5) -- cycle; 
    \draw (D5) -- (R5) -- (C5); 
    \draw (C3) -- (D4);
    \draw (C4) -- (D5); 
    \draw (R3) -- (R4); 
    \draw (R4) -- (R5); 
    \draw (A3) -- (Q1) -- (B3); 
    \draw (A4) -- (Q2) -- (B4);
    \draw (A5) -- (Q3) -- (B5); 
    \foreach \point in {A3, B3, C3, D3, R3, A4, B4, C4, D4, R4, A5, B5, C5, D5, R5, Q1, Q2, Q3}
        \fill (\point) circle (2pt);
    \node at (5.5, -1) {$\underline{\operatorname{TH}_3}$};
\end{tikzpicture}
\end{center}

The addition of the basement vertices does not affect the number of $3$-colorings as their colors are determined automatically. Consequently, $\mathcal{C}_3(\underline{\operatorname{TH}_n})$ still contains $6(n+2)=6r$ vertices. Furthermore, because every vertex participates in a triangle, $\mathcal{C}_3(\underline{\operatorname{TH}_n})$ has no edges, completing the proof. \end{proof}

\begin{corollary}\label{cor:k-chromatic-isolated} For each $r\geq 1$ and $k\geq 3$, there exists a connected graph $G$ with $\chi(G)=k$ such that $\mathcal{C}_k(G)\cong N_{k!\cdot r}$.
\end{corollary}

\begin{proof} We proceed by induction on $k$. The base case $k=3$ follows from Corollary~\ref{cor:3-chromatic-isolated}. For the inductive step, assume there exists a graph $G$ with $\chi(G) = k$ such that $\mathcal{C}_k(G)$ consists of $k! \cdot r$ isolated vertices. Consider a new graph $G'$ formed by adding a new vertex $v$ and connecting $v$ to all vertices of $G$. Then $\chi(G')=k+1$ and $\mathcal{C}_{k+1}(G')$ consists of $(k+1)\cdot (k!\cdot r)=(k+1)!\cdot r$ isolated vertices.
\end{proof}

We are now ready to prove one of the theorems from the introduction.

\begin{proof}[Proof of Theorem~\ref{thm:townhouse}] The case $\chi_1=\chi_2=k$, mentioned at the beginning of this section, is straightforward. Taking $G_1=K_k$ and $G_{2}=K_{k+1}-e$, we obtain $\mathcal{C}_{k}(G_1)\cong \mathcal{C}_{k}(G_2)\cong N_{k!}$. Thus, we assume $\chi_2>\chi_1\geq 3$. Let $r=\dfrac{\chi_2!}{\chi_1!}\in\mathbb{N}$. By Corollary~\ref{cor:k-chromatic-isolated}, there exists a graph $G_1$ such that $\chi(G_1)=\chi_1$ and $\mathcal{C}_{\chi_1}(G_1)$ consists of $\chi_1!\cdot r = \chi_2!$ isolated vertices. Let $G_2=K_{\chi_2}$ which satisfies $\chi(G_2)=\chi_2$, and $\mathcal{C}_{\chi_2}(G_2)$ also consists of $\chi_2!$ isolated vertices. As $\chi(G_1)\neq \chi(G_2)$, the graphs $G_1$ and $G_2$ are nonisomorphic, but $\mathcal{C}_{\chi_1}(G_1)\cong \mathcal{C}_{\chi_2}(G_2)$. \end{proof}

We discuss another consequence of Corollary~\ref{cor:k-chromatic-isolated}. While $\mathcal{C}_{\chi(G)}(G)$ does not generally allow for the unique reconstruction of the base graph $G$, it is worth considering what information about the base graph can still be derived from its $\chi$-coloring graph. The next observation, which strengthens Theorem~\ref{thm:townhouse}, shows that the $\chi$-coloring graph cannot even detect connectivity of its base graph. 

\begin{theorem}
For $\chi_2\geq \chi_1\geq 3$, there exists a disconnected $\chi_1$-chromatic graph $G_1$ and a connected $\chi_2$-chromatic graph $G_2$ such that $\mathcal{C}_{\chi_1}(G_1)\cong \mathcal{C}_{\chi_2}(G_2)$. 
\end{theorem} 

\begin{proof}
By Corollary~\ref{cor:k-chromatic-isolated} applied with $k=\chi_1$ and $r=\chi_2!$, there exists a graph $H$ with $\chi(H)=\chi_1$ such that $\mathcal{C}_{\chi_1}(H)\cong N_{\chi_1!\cdot \chi_2!}$. Consider the disjoint union graph $G_1=H\uplus H$, which is disconnected, and satisfies $\chi(G_1)=\chi_1$ and $\mathcal{C}_{\chi_1}(G_1) \cong N_{\chi_1!^2 \chi_2!^2}$. By Corollary~\ref{cor:k-chromatic-isolated} applied with $k=\chi_2$ and $r=\chi_2!\cdot \chi_1!^2$, there exists a connected graph $G_2$ with $\chi(G_2)=\chi_2$ such that $\mathcal{C}_{\chi_2}(G_2)\cong N_{\chi_2!\cdot (\chi_2!\cdot \chi_1!^2)} = N_{\chi_1!^2 \chi_2!^2}$. 
\end{proof}

These results show that, in general, a graph cannot be uniquely reconstructed from its $\chi$-coloring graph; however, exceptions exist. Our next result demonstrates that the $3$-coloring graph of the $5$-cycle is unique among all $\chi$-coloring graphs.

\begin{theorem}\label{thm:C3(C5)}
If $G$ is a graph satisfying $\mathcal{C}_{\chi(G)}(G)\cong\mathcal{C}_3(C_5)$, then $G\cong C_5$.
\end{theorem}
\begin{proof}
    Observe that $\C \colonequals \mathcal{C}_3(C_5)$ is the disjoint union of two copies of $C_{15}$. For any graph $H$, the number $\abs{V(\C_k(H))}$ is divisible by $k!$ when $k=\chi(H)$. Since $\abs{V(\C_{\chi(G)}(G))}=\abs{V(\C)}=30$ is not divisible by $k!$ for $k\geq 4$, it follows that $\chi(G)\leq 3$. Moreover, since $\abs{V(\C)}=30$ is not a power of $2$, we have $\chi(G)>2$. Thus, $\chi(G)=3$. 

     First, we show that $|V(G)|\geq 5$. As $G$ is \emph{not} bipartite, $G$ contains an odd cycle. If $G$ contained a triangle, then $\C$ would have at least $3!=6$ connected components; indeed, a $3$-coloring of a triangle is locked and cannot change within a connected component of $\C$. However, $\C$ has only two components. Thus, $G$ contains an $\ell$-cycle where $\ell\geq 5$, so $|V(G)|\geq 5$.
    
    Consider one of the $15$-cycles in $\C$. Let $\{v_1, \dots, v_r\}\subseteq V(G)$ be the subset of vertices in $G$ that achieve at least two different colors in the chosen 15-cycle in $\C$. Let $\lambda_i\geq 2$ be the number of times vertex $v_i$ changes color in this cycle, so that $(\lambda_1, \dots, \lambda_r)$ forms an integer partition of $15$. We claim that $\lambda_i\geq 3$ for every $i$. If $\lambda_j=2$ for some $j$, then vertex $v_j$ is recolored (in this $15$-cycle) between only two colors, say $1\leftrightarrow 2$. Permuting these two colors with the third available color generates two \emph{additional} $15$-cycles in $\mathcal{C}$: one corresponding to $v_j$ switching between $1\leftrightarrow 3$, and the other corresponding to $v_j$ switching between $2\leftrightarrow 3$. Together with the original $15$-cycle, this would force $|V(\mathcal{C})|\ge 3\cdot 15 = 45$, a contradiction.
    
   Combining $\lambda_i\geq 3$ for all $i$ and $\lambda_1+\dots + \lambda_r=15$, we get $r\leq 5$. We claim that $r=5$. If $r\leq 4$, then $|V(G)|\geq 5$ implies that there is a vertex $u\in V(G)$ whose color is fixed in the $15$-cycle, say color $1$. Changing its color to $2$ or $3$ results in two additional $15$-cycles, implying $|V(\C)|\geq 45$, a contradiction. 

Thus, $r=5$, which combined with $\lambda_i\geq 3$ for all $i$, forces $\lambda_i=3$ for $1\leq i\leq 5$. We claim that $|V(G)|=5$. Indeed, any vertex $w\in V(G)\setminus \{v_1, \dots, v_5\}$ would have a fixed color in the $15$-cycle, say color $1$. Changing its color to $2$ or $3$ results in two additional $15$-cycles, implying $|V(\C)|\geq 45$, a contradiction. Therefore, $|V(G)|=5$ and since $G$ contains an $\ell$-cycle with $\ell\geq 5$, we have $G\cong C_5$. \end{proof}

Later, Theorem~\ref{thm:main} will show that $\mathcal{C}_3(C_5)\not\cong \mathcal{C}_k(G)$ for any graph $G$ with $k>\chi(G)$. Combining this fact with Theorem~\ref{thm:C3(C5)}, we conclude that $C_5$ is the only graph having a coloring graph isomorphic to $\mathcal{C}_3(C_5)$. More generally, it seems difficult to characterize graphs $G$ that can be uniquely reconstructed from $\mathcal{C}_{\chi(G)}(G)$.

\section{Any Coloring Graph With Surplus Colors Has a Unique Base Graph}\label{sec:unique-if-surplus}

In this section, we establish Theorem~\ref{thm:main-intro} (our main result) by proving Theorems~\ref{thm:reconstruction} and~\ref{thm:main}, which respectively show that we can reconstruct base graphs from their coloring graphs with surplus colors, and that $\chi$-coloring graphs cannot be isomorphic to surplus coloring graphs. These results rely on the notion of a {\em link $k$-coloring} (see Definition~\ref{def:link-coloring}). We show that link $k$-colorings of the base graph correspond to {\em abstract link vertices} of a coloring graph, which are vertices that have certain structural properties identified by Algorithm~\ref{algo:link-vertex}. Theorem~\ref{thm:reconstruction} shows that for any $\C$, there is at most one pair $(G, k)$ with $k>\chi(G)$ such that $\C\cong\mathcal C_{k}(G)$. A useful strategy in our proofs is to extract information from link colorings and their corresponding abstract link vertices. 

Showing that $\C_k(G)\cong \C_{\chi'}(G')$ is impossible when $k>\chi(G)$ and $\chi'=\chi(G')$ (Theorem~\ref{thm:main}) requires several new ideas. The presence of abstract link vertices in $\C\cong \C_k(G)\cong \C_{\chi'}(G')$ imposes specific constraints on how $G$ can relate to a potential $G'$. The first step is to show that any such $G'$ must contain the connected components of $G$ as vertex-disjoint induced subgraphs. Next, Algorithm~\ref{algo:partitions} shows that at each abstract link vertex, this subgraph of $G$ in $G'$ must be colored according to the partition of the coloring in the surplus coloring case. Algorithm~\ref{algo:LLG} then extracts from $\C$ an auxiliary structure, termed \emph{labeled link graph}, in a base-graph agnostic manner. These algorithms explain how different abstract link vertices in $\mathcal{C}$ are interconnected. Finally, interpreting the shared graph $\mathcal{C}$ as $\mathcal{C}_{\chi'}(G')$ requires it to contain strictly more abstract link vertices than are present when interpreting it as $\mathcal{C}_{k}(G)$, contradicting the possibility that $\mathcal{C}_k(G)\cong  \mathcal{C}_{\chi'}(G')$. A flowchart (Figure~\ref{fig:flowchart}) at the end of the paper shows the progression of Section~\ref{sec:unique-if-surplus}.

All of our algorithms take a graph $\mathcal C$ as input and seek to interpret $\mathcal C$ as a coloring graph.
We briefly note that the runtime of these algorithms is polynomial in the number of vertices of $\mathcal C$ except for a graph isomorphism check at the end of Algorithm~\ref{algo:link-vertex}. However, the order of a coloring graph is exponential in the order of its base graph in general, so even excluding the graph isomorphism check, our algorithms are not efficient with respect to the base graph.

\subsection{Reconstructing the base graph via link vertices}\label{subsect:reconstruction} This subsection proves Theorem~\ref{thm:reconstruction}: if an abstract coloring graph $\mathcal{C}$ is isomorphic to $\mathcal{C}_k(G)$ for some graph $G$ and an integer $k>\chi(G)$, then this pair $(G, k)$ is unique and algorithmically recoverable from $\mathcal{C}$. The core idea is to identify abstract link vertices (Definition~\ref{def:link-vertex}) within the graph $\mathcal{C}$. We first define the corresponding concept in the base graph, the link $k$-coloring (Definition~\ref{def:link-coloring}), a coloring whose image size within each connected component is minimal.

\begin{definition}\label{def:link-coloring}
    Given a graph $G$ with connected components $H_1,\dots,H_d$  and an integer $k>\chi(G)$, a \emph{link~$k$-coloring of $G$} is a $k$-coloring $\alpha_G\colon V(G)\to [k]$ such that the image of $V(H_i)$ under $\alpha_G$ has size $\chi(H_i)$ for all $i\in[d]$. 
\end{definition}

Equivalently, a link $k$-coloring maximizes the number of common free colors for all vertices within each connected component.

We use Algorithm~\ref{algo:link-vertex} to find vertices in an abstract graph $\mathcal{C}$ that correspond to link $k$-colorings, each having local structural properties in $\mathcal C$ that imply an adjacency matrix that defines an underlying base graph. When $\mathcal{C}\cong \mathcal{C}_k(G)$ with $k>\chi(G)$, this process succeeds and allows us to reconstruct the unique pair $(G, k)$. The vertices successfully identified by this algorithm are called \emph{abstract link vertices}, formally defined in Definition~\ref{def:link-vertex}.

\begin{algorithm}[Reconstruction and Link Vertex Identification]\label{algo:link-vertex} 
    Given a graph $\C$, identify a subset of vertices $A\subseteq V(\C)$, a graph $G$, and an integer $k$ as follows:
\begin{enumerate}
    \item\label{step:clique-neighborhood} Let $A$ be the set of vertices $\alpha\in V(\C)$ where the neighborhood of $\alpha$ is a disjoint union of cliques. If $A\ne V(\mathcal C)$, abort. Otherwise, for each $\alpha\in V(\mathcal C)$, let $r$ be the number of cliques in its neighborhood, and label them arbitrarily $v_1, \dots, v_r$.
    \item\label{step:all-vertices-free} Let $n$ be the maximum number of cliques in the neighborhood of any vertex in $A$. Remove from $A$ any vertex that does not attain this maximum (i.e., if $r<n$).
    \item\label{step:all-edges-have-common-color} For each $\alpha\in A$, let $t_1,\dots,t_n$ be the sizes of the $n$ cliques in its neighborhood. Let $m_\alpha$ be the number of pairs $1\le i<j\le n$ with strictly fewer than $t_it_j$ squares spanning the $i$th and $j$th clique in the neighborhood of $\alpha$. Let $m$ be the maximum value of $m_\alpha$ over all $\alpha\in A$. Remove from $A$ any vertex $\alpha$ for which $m_{\alpha}<m$.
    \item\label{step:adj-matrix} For each $\alpha\in A$, let $M_\alpha\in \set{0,1}^{n\times n}$ be the matrix such that $M_{ij}=1$ if and only if $i\ne j$ and there are strictly fewer than $t_it_j$ squares spanning the $i$th and $j$th clique in the neighborhood of $\alpha$.
    \item\label{step:find-common-colors} For each $\alpha\in A$ and connected component $H_\ell$ induced by $M_\alpha$, let $f_\ell$ be the maximum number of ways to select one vertex from each clique in the neighborhood of $\alpha$ that corresponds to a vertex in $H_\ell$ such that there is no square spanning the vertices corresponding to $v_i,v_j$ for any $i,j$ with $M_{i,j}=1$. For each of these $f_\ell$ sets of $\abs{V(H_\ell)}$ neighbors of $\alpha$, label these edges with a distinct $c\in[f_\ell]$. 
    \item\label{step:max-common-colors} Let $f$ be the maximum value of $\sum_\ell f_\ell$ over all $\alpha\in A$. Remove from $A$ any $\alpha$ for which $\sum_\ell f_\ell < f$.
    \item\label{step:formula-for-k} Let $G$ be the graph with connected components $H_1,\dots,H_d$, as determined by an arbitrary $M_{\alpha}$ from Step~\ref{step:adj-matrix} interpreted as an adjacency matrix. Calculate $k=\frac 1 d \cdot \left(f+\sum_{\ell\in[d]}\chi(H_\ell)\right)$, where $f$ is the maximum value from Step~\ref{step:max-common-colors}. 
    \item\label{step:alg1-return} If $\C\not\cong \C_k(G)$, abort. Otherwise, return vertex set $A$, graph $G$, and integer $k$.
\end{enumerate}
\end{algorithm}

\begin{definition}\label{def:link-vertex}
    Given a graph $\C$, an \emph{abstract link vertex of $\C$} is any vertex $\alpha$ belonging to the subset $A\subseteq V(\C)$ returned by Algorithm~\ref{algo:link-vertex}. 
    In other words, an abstract link vertex is one that survives Step~\ref{step:max-common-colors}, provided the algorithm successfully identifies a graph $G$ and an integer $k$ such that $\C\cong\C_k(G)$.
\end{definition}

The next lemma bridges link $k$-colorings and abstract link vertices when $\mathcal{C}\cong \mathcal{C}_k(G)$ for $k > \chi(G)$.

\begin{lemma}\label{lem:link-equivalence}
    For a graph $G$ and $k>\chi(G)$, there is a natural bijection between the link $k$-colorings of $G$ and the abstract link vertices of $\C_k(G)$.
\end{lemma}

\begin{proof} We show that the natural bijection $\Phi\colon V(\C_k(G))\to\set{\text{$k$-colorings of $G$}} $ restricts to a bijection between the abstract link vertices of $\C_k(G)$ and the link $k$-colorings of $G$ whenever $k>\chi(G)$. It suffices to prove that the image of the abstract link vertices under $\Phi$ contains only link $k$-colorings and that every link $k$-coloring is achieved by some abstract link vertex. Recall that using the $\alpha\mapsto\alpha_G$ notation from Section~\ref{sec:def}, the symbol $\alpha$ represents an element of $V(\C_k(G))$ and $\alpha_G=\Phi(\alpha)$ is the corresponding $k$-coloring of base graph $G$. 

Let $\alpha_G$ be a link $k$-coloring of $G$. Suppose $G$ has $d$ connected components $H_1,\dots,H_d$. We show that $\alpha=\Phi^{-1}(\alpha_G)$ remains in set $A$ after all steps of Algorithm~\ref{algo:link-vertex}. The neighborhood of $\alpha$ in $\C_k(G)$ consists of the disjoint union of $n=\abs{V(G)}$ cliques, because $k>\chi(G)$ implies that each vertex of $G$ has at least one available color in a link $k$-coloring. Moreover, all vertices in $V(\C)$ satisfy the condition in Step~\ref{step:clique-neighborhood} for any coloring graph $\C$. Indeed, the edges of any clique in a coloring graph correspond to a single vertex changing color, and any coloring has a finite number of vertices with free colors, yielding a neighborhood of cliques. In particular, no vertex of $\C$ can have a neighborhood with more than $n$ cliques, so $\alpha$ survives Step~\ref{step:all-vertices-free}. 

Next, consider Step~\ref{step:all-edges-have-common-color}. For every edge $v_iv_j\in E(G)$, the corresponding cliques in the neighborhood of $\alpha$ have fewer than $t_it_j$ squares. This is because $v_i$ and $v_j$ have at least one common free color in a link $k$-coloring, and a square corresponding to assigning this same color to both vertices is necessarily missing since $v_iv_j\in E(G)$. Conversely, for any $v_iv_j\not\in E(G)$, the corresponding cliques have exactly $t_it_j$ squares, as any combination of free colors for $v_i$ and $v_j$ produces a valid $k$-coloring. Hence, $m_\alpha=|E(G)|$ for this $\alpha$. Additionally, no $\beta\in V(\C)$ can have $m_\beta>\abs{E(G)}$, because this would require some $i,j$ with fewer than $t_it_j$ squares but $v_iv_j\not\in E(G)$, a contradiction. Therefore, $\alpha$ achieves the maximum $m$ and survives Step~\ref{step:all-edges-have-common-color}. 

Now consider Step~\ref{step:max-common-colors}. By definition of link $k$-coloring, for each $\ell\in[d]$, every vertex in $V(H_\ell)$ shares exactly $k-\chi(H_\ell)$ common free colors in $\alpha_G$. Each of these common free colors is associated with $n_\ell=\abs{V(H_\ell)}$ neighbors of $\alpha$ which pairwise have spanning squares if and only if the pair of corresponding vertices in $G$ are non-adjacent. Thus, for each component $H_{\ell}$, the value of $f_{\ell}$ calculated in Step~\ref{step:find-common-colors} is exactly $f_\ell=k-\chi(H_\ell)$. Moreover, for any $\beta\in V(\C_k(G))$, a missing square between any two neighbors of  $\beta$ implies that the corresponding two vertices of $G$ take on the same color moving from $\beta$ to the respective neighbors. Thus, maximality of the definition of link $k$-colorings ensures that $\alpha$ survives Step~\ref{step:max-common-colors} and is included in the set $A$ returned by the algorithm.

Conversely, we show that any abstract link vertex $\alpha$ in $\mathcal{C}_k(G)$ corresponds to some link $k$-coloring $\alpha_G$ of $G$. For $\alpha$ to be in the final set $A$:
\begin{itemize}
    \item Surviving Step~\ref{step:all-vertices-free} means that \emph{every} vertex $v_i\in V(G)$ has at least one free color in $\alpha_G$.
    \item Surviving Step~\ref{step:all-edges-have-common-color} implies $m_{\alpha}=|E(G)|$. Hence, for every edge $v_i v_j\in E(G)$, the vertices $v_i$ and $v_j$ share at least one common free color in $\alpha_G$.
    \item Surviving Step~\ref{step:max-common-colors} requires $\sum_{\ell} f_{\ell}$ to be maximal. This corresponds precisely to colorings where each component $H_{\ell}$ uses exactly $\chi(H_{\ell})$ colors, maximizing the number of common free colors within that component. By Definition~\ref{def:link-coloring}, we deduce that $\alpha_G$ is a link $k$-coloring. 
\end{itemize}
Therefore, $\Phi$ restricted to the set of abstract link vertices is a bijection onto the set of link $k$-colorings. 
\end{proof}

Figure~\ref{fig:alg1-attrition} presents three non-link colorings that are rejected at different steps of Algorithm~\ref{algo:link-vertex} and one link coloring that survives. The labels within braces next to a given vertex indicate its free available colors.

\begin{figure}[ht]
\centering
\begin{tabular}{cc}
\begin{tikzpicture}[scale=1,
  every node/.style={draw=none, fill=none, minimum size=8mm, font=\small}]
  \node at (0,0) {$\delta_G$};
  \node[fill=red!30,    circle, draw] (v1) at ( 90:1)  {1};
  \node[fill=green!30,  circle, draw] (v2) at ( 18:1)  {2};
  \node[fill=blue!30,   circle, draw] (v3) at (-54:1)  {3};
  \node[fill=purple!30, circle, draw] (v4) at (-126:1) {4};
  \node[fill=green!30,  circle, draw] (v5) at ( 162:1) {2};
  \draw (v1) -- (v2) -- (v3) -- (v4) -- (v5) -- (v1);
  \node[fill=red!30, circle, draw] (v6) at ($(v3)+(1,0)$) {1};
  \draw (v3) -- (v6);
  \node[anchor=south] at (v1.north) {\{3, 4\}};
  \node[anchor=west]  at (v2.east)  {\{4\}};
  \node[anchor=east]  at (v4.west)  {\{1\}};
  \node[anchor=east]  at (v5.west)  {\{3\}};
  \node[anchor=west]  at (v6.east)  {\{2, 4\}};
\end{tikzpicture}
&
\begin{tikzpicture}[scale=1,
  every node/.style={draw=none, fill=none, minimum size=8mm, font=\small}]
  \node at (0,0) {$\gamma_G$};
  \node[fill=red!30,    circle, draw] (v1) at ( 90:1) {1};
  \node[fill=green!30,  circle, draw] (v2) at ( 18:1) {2};
  \node[fill=blue!30,   circle, draw] (v3) at (-54:1) {3};
  \node[fill=purple!30, circle, draw] (v4) at (-126:1){4};
  \node[fill=green!30,  circle, draw] (v5) at (162:1) {2};
  \draw (v1) -- (v2) -- (v3) -- (v4) -- (v5) -- (v1);
  \node[anchor=south] at (v1.north) {\{3, 4\}};
  \node[anchor=west]  at (v2.east)  {\{4\}};
  \node[anchor=west]  at (v3.east)  {\{1\}};
  \node[anchor=east]  at (v4.west)  {\{1\}};
  \node[anchor=east]  at (v5.west)  {\{3\}};
\end{tikzpicture}
\\[1.2em]
\begin{tikzpicture}[scale=1,
  every node/.style={draw=none, fill=none, minimum size=8mm, font=\small}]
  \node at (0,0) {$\beta_G$};
  \node[fill=red!30,    circle, draw] (v1) at ( 90:1) {1};
  \node[fill=green!30,  circle, draw] (v2) at ( 18:1) {2};
  \node[fill=blue!30,   circle, draw] (v3) at (-54:1) {3};
  \node[fill=purple!30, circle, draw] (v4) at (-126:1){4};
  \node[fill=green!30,  circle, draw] (v5) at (162:1) {2};
  \draw (v1) -- (v2) -- (v3) -- (v4) -- (v5) -- (v1);
  \node[anchor=south] at (v1.north) {\{3, 4, 5\}};
  \node[anchor=west]  at (v2.east)  {\{4, 5\}};
  \node[anchor=west]  at (v3.east)  {\{1, 5\}};
  \node[anchor=east]  at (v4.west)  {\{1, 5\}};
  \node[anchor=east]  at (v5.west)  {\{3, 5\}};
\end{tikzpicture}
&
\begin{tikzpicture}[scale=1,
  every node/.style={draw=none, fill=none, minimum size=8mm, font=\small}]
  \node at (0,0) {$\alpha_G$};
  \node[fill=red!30,   circle, draw] (v1) at ( 90:1) {1};
  \node[fill=green!30, circle, draw] (v2) at ( 18:1) {2};
  \node[fill=blue!30,  circle, draw] (v3) at (-54:1) {3};
  \node[fill=red!30,   circle, draw] (v4) at (-126:1){1};
  \node[fill=green!30, circle, draw] (v5) at (162:1) {2};
  \draw (v1) -- (v2) -- (v3) -- (v4) -- (v5) -- (v1);
  \node[anchor=south] at (v1.north) {\{3, 4, 5\}};
  \node[anchor=west]  at (v2.east)  {\{4, 5\}};
  \node[anchor=west]  at (v3.east)  {\{4, 5\}};
  \node[anchor=east]  at (v4.west)  {\{4, 5\}};
  \node[anchor=east]  at (v5.west)  {\{3, 4, 5\}};
\end{tikzpicture}
\end{tabular}
\caption{The vertex $\delta\in V(\C_4(G))$ corresponding to coloring $\delta_G$ does not survive Step~\ref{step:all-vertices-free} of Algorithm~\ref{algo:link-vertex} because the vertex colored $3$ has no free color (free colors are within braces); $\gamma\in V(\C_4(C_5))$ survives Step~\ref{step:all-vertices-free} but not Step~\ref{step:all-edges-have-common-color} of Algorithm~\ref{algo:link-vertex} because each vertex in $\gamma_G$ has a free color but not all adjacent vertices share a common free color; $\beta\in V(\C_5(C_5))$, survives Step~\ref{step:all-edges-have-common-color} but not Step~\ref{step:max-common-colors} because every pair of adjacent vertices in $\beta_G$ has a common free color but there is only one common free color for all vertices; $\alpha_G$ is a link 5-coloring of $C_5$ such that $\alpha\in V(\C_5(C_5))$ survives Algorithm~\ref{algo:link-vertex} because colors 4 and 5 are free for each vertex and no other coloring can have more than $5-\chi(C_5)=2$ common free colors.}
\label{fig:alg1-attrition}
\end{figure}

In the proof of Lemma~\ref{lem:link-equivalence}, we saw that the vertices $\gamma$ that survive Step~\ref{step:all-vertices-free} correspond to $k$-colorings in which each base graph vertex has a free color. Similarly, the vertices $\beta$ that survive Step~\ref{step:all-edges-have-common-color} represent those colorings in which every pair of adjacent base graph vertices share a common free color. All such colorings yield the same adjacency matrix in Step~\ref{step:adj-matrix}. This shared adjacency matrix is central to Theorem~\ref{thm:reconstruction}, which guarantees the uniqueness of the graph $G$ and the value of $k$ produced by the algorithm.

\begin{theorem}\label{thm:reconstruction}
    For any graph $G$ and $k>\chi(G)$, there is no graph $G'$ and $k'>\chi(G')$ with $G\not\cong G'$ or $k\ne k'$ such that $\C_k(G)\cong\C_{k'}(G')$. Moreover, Algorithm~\ref{algo:link-vertex} produces the unique (up to graph isomorphism) $G$ and $k$.
\end{theorem}

\begin{proof}
    Suppose that $\C_k(G)\cong\C_{k'}(G')$ for graphs $G$ and $G'$ where $k>\chi(G)$ and $k'>\chi(G')$. We show that $G\cong G'$ and $k=k'$. The isomorphism $\C\cong\C_k(G)\cong\C_{k'}(G')$ means that Algorithm~\ref{algo:link-vertex}, run on $\mathcal{C}$, identifies the \emph{same} set $A$ of abstract link vertices, regardless of whether we consider $\mathcal{C}$ as $\C_k(G)$ or $\C_{k'}(G')$. By Lemma~\ref{lem:link-equivalence}, this set $A$ is in bijection with the link $k$-colorings of $G$, and also in bijection with the link $k'$-colorings of $G'$. 
    
    Algorithm~\ref{algo:link-vertex} (Step~\ref{step:all-vertices-free}) identifies $n$ as the maximum number of cliques in the neighborhood of any abstract link vertex. This number $n$ corresponds to $|V(G)|$ when interpreting $\mathcal{C}$ as $\mathcal{C}_k(G)$ and to $|V(G')|$ when interpreting $\mathcal{C}$ as $\mathcal{C}_{k'}(G')$. Therefore, $|V(G)|=|V(G')|=n$. Algorithm~\ref{algo:link-vertex} (Steps~\ref{step:all-edges-have-common-color} and~\ref{step:adj-matrix}) derives an adjacency matrix $M_{\alpha}$ based on the pattern of missing squares between neighborhood cliques of abstract link vertices. Since this pattern is solely determined by the structure of $\C$, it yields the same $M_{\alpha}$ (up to permutation), regardless of whether $\mathcal{C}$ is interpreted as $\C_k(G)$ or $\C_{k'}(G')$. As $M_{\alpha}$   represents the adjacency matrix of the base graph in both interpretations (because $k>\chi(G)$ \emph{and} $k'>\chi(G')$), it follows that $G\cong G'$. 
    
    Furthermore, the algorithm calculates the maximum value $f=\sum_{\ell} f_{\ell}$ in Step~\ref{step:max-common-colors}. Using the formula from Step~\ref{step:formula-for-k}, we have $k=\frac{1}{d}\left(f + \sum_{\ell\in [d]}\chi(H_{\ell})\right)$ and $k'=\frac{1}{d'}\left(f + \sum_{\ell\in [d']}\chi(H'_{\ell})\right)$. Since $G\cong G'$, the connected components of the two graphs are isomorphic. After relabeling the component indices, we have $d=d'$ and $H_{\ell}\cong H'_{\ell}$ for each $\ell\in [d]$. As $f$ is also derived uniquely from $\C$, it follows that $k=k'$.
    \end{proof}

\begin{remark} The adjacency matrix of $G$ can be reconstructed not only from abstract link vertices in $V(\C_k(G))$, but also from certain other vertices that we can completely characterize. Specifically, for any coloring in which every pair of adjacent vertices in $G$ shares a common free color, the corresponding vertex in $\C_k(G)$ will also yield the correct adjacency matrix. We refer to such colorings as {\em weak link $k$-colorings} of $G$. Unlike link $k$-colorings, weak link $k$-colorings do not require that \emph{all vertices} within a given connected component share a common free color; they only require that \emph{each pair} of adjacent vertices in $G$ shares a common free color (which may differ between pairs). For example, the coloring $\beta_{G}$ in Figure~\ref{fig:alg1-attrition} is a weak link $5$-coloring, but not a link $5$-coloring. The concept of link $k$-colorings, along with Algorithm \ref{algo:link-vertex}, provides a systematic framework for identifying vertices from which the graph $G$ can be uniquely reconstructed. In contrast, it remains unclear how the vertices corresponding to weak link $k$-colorings can be reliably identified using only the structure of the coloring graph.
\end{remark}

In the following subsections, we will extract more information from the coloring graph. We remark that the definition of abstract link vertices applies to any abstract graph $\C$, including those of the form $\C_{\chi(G')}(G')$. This is true even though the concept of ``link $\chi(G')$-colorings" is not defined analogously to link $k$-coloring where $k>\chi(G)$. However, it is not yet clear whether Algorithm~\ref{algo:link-vertex} can successfully run to completion on a graph $\C_{\chi(G')}(G')$ without aborting. For example, consider the $5$-vertex graph $G'$ formed by connecting a vertex of $K_2$ to a vertex of $K_3$ by a single edge. Its 3-coloring graph, $\C=\mathcal{C}_3(G')$, has $24$ vertices: $\mathcal{C}$ is the disjoint union of six paths of length $3$ (that is, on four vertices). The degree-2 vertices in $\C$ survive Step~\ref{step:max-common-colors} of Algorithm~\ref{algo:link-vertex}, yielding the parameters $n=2,m=1,f=1$ for a potential base graph. The resulting matrix $M$ in Step~\ref{step:adj-matrix} corresponds to $G=P_2$, and Step~\ref{step:formula-for-k} then calculates $k=f+\chi(P_2)=1+2=3$. However, the algorithm aborts in Step~\ref{step:alg1-return} because the input graph $\mathcal{C}$ is \emph{not} isomorphic to $\mathcal{C}_3(P_2)\cong C_6$.  

The proof that $\C_{\chi(G')}(G')$ can \emph{never} survive Algorithm~\ref{algo:link-vertex} will follow from Theorem~\ref{thm:main}; for now, we observe that the properties of abstract link vertices already place structural constraints on $G'$. These constraints relate to the number and adjacency of vertices in $G'$ that have free colors available at the corresponding coloring. In the previous example, consider an interior vertex $\alpha$ of one of the copies of $P_4$ in $\mathcal{C}$. Let $\beta_1, \beta_2$ be the two neighbors of $\alpha$. The edges $\alpha\beta_1$ and $\alpha\beta_2$ in $\mathcal{C}$ are neither part of a common clique nor part of an induced square. This necessarily represents a coloring in $G'$ in which only two vertices, say $u$ and $v$, have free colors, and all other vertices have their colors locked (that is, cannot be recolored). Furthermore, the lack of squares indicates that $u$ and $v$ must be adjacent in $G'$. In other words, if $\C_k(G')$ is isomorphic to the disjoint union of six copies of $P_4$, then $G'$ contains $P_2$ as a subgraph. This is related to the fact that the adjacency matrix produced by Step~\ref{step:adj-matrix} corresponds to $P_2$. Lemma~\ref{lem:subgraphs} generalizes this observation.

\begin{lemma}\label{lem:subgraphs}
    If $\C\cong \C_k(G)\cong \C_{\chi'}(G')$ where $k>\chi(G)$ and $\chi'=\chi(G')$, then $G'$ contains $G$ as a subgraph. Moreover, if $G$ consists of connected components $H_1,\dots,H_d$, then $G'$ contains the connected components of $G$ as vertex-disjoint induced subgraphs. That is, there exists an injection $\phi\colon V(G)\to V(G')$ such that for any $i\in[d]$ and $u,v\in H_i$, we have $uv\in E(G)$ if and only if $\phi(u)\phi(v)\in E(G')$. 
\end{lemma}

\begin{proof}
    By Theorem~\ref{thm:reconstruction}, we know that Algorithm~\ref{algo:link-vertex} run on $\mathcal{C}$ produces a unique (up to permutation) adjacency matrix $M_\alpha$  associated with any abstract link vertex. The isomorphism $\C\cong\C_{\chi(G')}(G')$ induces a mapping $\alpha\mapsto\alpha_{G'}$ so that each abstract link vertex $\alpha\in V(\mathcal{C})$ corresponds to a $\chi'$-coloring $\alpha_{G'}$ of $G'$. The neighborhood structure around $\alpha$ in $\C$, specifically the cliques and the pattern of missing induced squares analyzed in Step~\ref{step:all-edges-have-common-color} and Step~\ref{step:adj-matrix} of Algorithm~\ref{algo:link-vertex}, implies that the corresponding $\chi'$-coloring $\alpha_{G'}$ in $G'$ has exactly $n=|V(G)|$ vertices with at least one free color. Furthermore, a pair of these vertices corresponding to missing induced squares in $\mathcal{C}$ must be adjacent in $G'$. Since $M_\alpha$ is the adjacency matrix of $G$ by Theorem~\ref{thm:reconstruction}, it follows that $G$ is a subgraph of $G'$. We now make the inclusion $G\hookrightarrow G'$ explicit.

    Label the $n$ vertices in $G$ as $v_1,\dots,v_n$, corresponding to the $n$ cliques in the neighborhood of $\alpha$ in $\mathcal{C}$. Via the isomorphism $\mathcal{C}\cong\mathcal{C}_{\chi'}(G')$, each clique associated with $v_j$ in the $\mathcal{C}_k(G)$ interpretation corresponds to changing the color of a unique vertex $\phi(v_j)$ in $G'$. This mapping $\phi$ is well-defined because each clique in the neighborhood of a vertex in $\C_{\chi'}(G')$ corresponds to the set of colorings obtained by changing the color of a single vertex in $G'$. In addition, the resulting map $\phi\colon V(G)\to V(G')$ is an injection. 
    
    As $G'$ contains $G$ as a subgraph, $\phi(u)\phi(v)\in E(G')$ for every edge $uv\in E(G)$. To show that each component $\phi(H_i)$ is an induced subgraph of $G'$, suppose for a contradiction that there exists an edge $\phi(u)\phi(v)\in E(G')$ with $uv\notin E(G)$ for some $u, v\in V(H_i)$. As $uv\notin E(G)$, by Theorem \ref{thm:reconstruction}, every spanning square is present between the cliques in $\C$ associated with base graph vertices $u$ and $v$. Since $u$ and $v$ are in the same connected component $H_i$ of $G$, there exists a path $u=w_1,w_2,\dots, w_{r-1},w_r=v$ within $H_i$ such that $w_jw_{j+1}\in E(G)$ for each $j\in[r-1]$. Thus, for each adjacent pair ($w_j$, $w_{j+1}$) on this path, there is at least one missing spanning square between their corresponding cliques in $\C$. The structure of missing squares along the path (when interpreting $\mathcal{C}$ via $G$) implies that the vertices $w_1, \dots, w_r$ must share at least one common free color in the link $k$-coloring in $G$, $\alpha_G$, associated with $\alpha$. Interpreting $\mathcal{C}$ via $G'$, the same structural properties hold: squares are missing between cliques for $\phi(w_j)$ and $\phi(w_{j+1})$, while all squares are present between cliques for $\phi(w_1)=\phi(u)$ and $\phi(w_r)=\phi(v)$. By the required structure of the missing spanning squares, there must be a common free color for all vertices $\phi(w_1),\dots,\phi(w_r)$.  In particular, this implies that $\phi(u)$ and $\phi(v)$ have a common free color. Since all spanning squares exist between the cliques for $\phi(u)$ and $\phi(v)$, any combination of assigning a free color $f_{u}$ to $\phi(u)$ and a free color $f_{v}$ to $\phi(v)$ results in a valid coloring. Since they share a common free color, we can set $f_u=f_{v}$, resulting in a proper coloring of $G'$ where $\phi(u)$ and $\phi(v)$ have the same color. This contradicts the assumption that $\phi(u)\phi(v)\in E(G')$.   \end{proof}

Under the same hypothesis of Lemma~\ref{lem:subgraphs}, we will show that $G'$ has strictly more vertices than $G$. We first recall the following folklore lemma \cite{Wes01}*{Exercise 5.2.7}.

\begin{lemma}\label{lem:rainbow-star} Given a proper $k$-coloring of a graph $H$ with $\chi(H)=k$, for each $1\leq i\leq k$, there exists a vertex $v$ with color $i$ such that $v$ is adjacent to vertices of the other $k-1$ colors.
\end{lemma}

\begin{proof}
Fix any color $i$, and let $X_i$ be the set of vertices colored $i$. Suppose, for contradiction, that \emph{every} $v \in X_i$ is missing at least one color in its neighborhood $N(v)$. That is, for each $v\in X_i$, there is some color $j\neq i$ not present in $N(v)$. We recolor $v$ with $j$ while still maintaining a proper coloring. By repeating this process for every $v \in X_i$, we eliminate color $i$ entirely, resulting in a proper $(k-1)$-coloring of $H$. This contradicts the assumption that $\chi(H)=k$. Thus, there is at least one vertex $v_i \in X_i$ whose neighborhood includes all $k-1$ other colors.
\end{proof}

\begin{lemma}\label{lem:subgraphs:order-comparison} 
If $\C\cong \C_k(G)\cong \C_{\chi'}(G')$ where $k>\chi(G)$ and $\chi'=\chi(G')$, then $|V(G')|\geq |V(G)|+\chi'$. \end{lemma}

\begin{proof}
Under the interpretation of $\mathcal{C}$ as  $\mathcal{C}_{k}(G)$, let $\alpha \in V(\mathcal{C})$ be an abstract link vertex, and let $\alpha_G$ be the corresponding link $k$-coloring of $G$. The neighborhood of $\alpha$ consists of $n=|V(G)|$ distinct cliques associated with vertices in $G$ changing their current color at $\alpha_{G}$ to a free color. Under the interpretation of $\mathcal{C}$ as $\C_{\chi'}(G')$, the same vertex $\alpha$ represents a proper $\chi'$-coloring $\alpha_{G'}$ of $G'$. The same $n$ cliques around $\alpha$ must arise from $n$ distinct vertices in $G'$ that can be recolored from their current color at $\alpha_{G'}$. By Lemma~\ref{lem:rainbow-star}, there are at least $\chi'$ vertices of $G'$ that are \emph{locked} in the coloring $\alpha_{G'}$, that is, vertices in $G'$ that cannot be recolored. Consequently, $|V(G')|\geq |V(G)|+\chi'$. 
\end{proof}

Since $\chi'>0$, Lemma~\ref{lem:subgraphs:order-comparison} has the following immediate consequence.

\begin{corollary}\label{cor:equal-number-of-vertices}
   Given a graph $G$ and $k>\chi(G)$, there is no graph $G'$ with $|V(G)|=|V(G')|$ such that $\C_k(G)\cong \C_{\chi(G')}(G')$.
\end{corollary}

This corollary demonstrates that coloring graphs with surplus colors are complete graph invariants, when restricted to base graphs of the same order. Our main result Theorem~\ref{thm:main} strengthens Corollary~\ref{cor:equal-number-of-vertices} by removing the hypothesis $|V(G)|=|V(G')|$. When $G$ and $G'$ have different orders, proving that $\C_k(G)\not\cong\C_{\chi(G')}(G')$ is highly non-trivial. The remainder of the paper is devoted to resolving this task by examining the structure of coloring graphs with surplus colors and identifying features absent in $\chi$-coloring graphs.

The following subsections introduce techniques that go beyond the scope of Algorithm~\ref{algo:link-vertex} to explore further structural properties that must hold if $ \C_k(G)\cong \C_{\chi(G')}(G')$ with $k>\chi(G)$. Ultimately, these properties will lead to a contradiction, proving that such an isomorphism is impossible.

\subsection{Link vertices are associated with unique coloring partitions of a base (sub)graph}\label{subsect:partition}
We now introduce Algorithm~\ref{algo:partitions}. 
Given $\C\cong\C_k(G)$ with $k>\chi(G)$ and an abstract link vertex $\alpha$ of $\C$, this algorithm labels certain edges incident to $\alpha$ in $\C$. This labeling corresponds to a unique {\em component-wise partition} that reflects the underlying coloring $\alpha_G$ of $G$. The resulting labeling is a key ingredient in Lemma~\ref{lem:partition-G'}. This important lemma establishes that if $\C_k(G)\cong \C_{\chi(G')}(G')$, then the underlying coloring $\alpha_{G'}$ of $G'$ corresponding to $\alpha$ has the same component-wise partition identified from $G$, when restricted to the induced copies of the components of $G$ (see Lemma~\ref{lem:subgraphs}).

\begin{definition}\label{def:comp-partition}
    Given a $k$-coloring $\alpha_G\colon V(G)\to \{1, 2, \dots, k\}$ of graph $G$, the {\em $\alpha_G$-partition of $V(G)$} is $\set{\alpha_G^{-1}(1),\dots,\alpha_G^{-1}(k)}$. If $G$ has $d$ connected components $H_1,\dots,H_d$, the {\em component-wise partition for $\alpha_G$} is  $P(\alpha_G)=(P^1,\dots,P^d)$ where $P^i=\set{\alpha_G^{-1}(1)\cap V(H_i), \dots, \alpha_G^{-1}(k)\cap V(H_i)}$ for each $i\in[d]$. 
    The \emph{component-wise integer partition for $\alpha_G$} is the tuple $\lambda(\alpha_G) = (\lambda^{1}, \dots, \lambda^{d})$ where $\lambda^{i}_j = |P^i_j|$ for all $i\in[d],j\in[\ell(P^i)]$. 
\end{definition}

The component-wise partitions define equivalence classes of $k$-colorings of a graph $G$ with connected components $H_1,\dots,H_d$. We describe this equivalence in terms of group actions. Let $S_k^d = S_k\times \dots \times S_k$ denote the $d$-fold product of the symmetric group $S_k$ on $k$ symbols. For $\sigma=(\sigma_1,\dots,\sigma_d)\in S_k^d$ and a $k$-coloring $\alpha_G$ of $G$, define $\sigma\cdot \alpha_G$ to be a $k$-coloring by assigning the colors $(\sigma\cdot \alpha_G)(v) \colonequals \sigma_i(\alpha_G(v))$ for $v\in V(H_i)$. This defines a group action of $S_k^{d}$ on the set of $k$-colorings of $G$. We say two $k$-colorings $\alpha_G$ and $\beta_G$ are {\em equivalent} if they belong to the same orbit under this $S_k^{d}$ action (that is, $\beta_G=\sigma\cdot \alpha_G$ for some $\sigma\in S_k^{d}$). Each orbit is associated with a unique component-wise partition $P$. Any $k$-coloring $\alpha_G$ consistent with $P$ can serve as a representative for its equivalence class. However, non-equivalent $k$-colorings (belonging to different orbits) may have the same component-wise {\em integer} partition. Figure~\ref{fig:paw} presents such an example.

\begin{figure}[ht]
\begin{center}
\begin{tikzpicture}[scale=1, every node/.style={draw=none, fill=none, minimum size=8mm, font=\small}]

\node at (0,-0.75) {$\alpha_G$};
\node[fill=red!30, circle, draw] (v1) at (0,1.25) {1};
\node[fill=green!30, circle, draw] (v2) at (0,0) {2};
\node[fill=blue!30, circle, draw] (v3) at (1,-1.15) {3};
\node[fill=red!30, circle, draw] (v4) at (-1,-1.15) {1};

\draw (v1) -- (v2) -- (v3) -- (v4) -- (v2);

\node[anchor=west] at (v1.east) {$v_1$};
\node[anchor=west] at (v2.east) {$v_2$};
\node[anchor=west] at (v3.east) {$v_3$};
\node[anchor=east] at (v4.west) {$v_4$};
\end{tikzpicture}
\hspace{6em}
\begin{tikzpicture}[scale=1, every node/.style={draw=none, fill=none, minimum size=8mm, font=\small}]

\node at (0,-0.75) {$\beta_G$};
\node[fill=blue!30, circle, draw] (v1) at (0,1.25) {3};
\node[fill=green!30, circle, draw] (v2) at (0,0) {2};
\node[fill=blue!30, circle, draw] (v3) at (1,-1.15) {3};
\node[fill=red!30, circle, draw] (v4) at (-1,-1.15) {1};

\draw (v1) -- (v2) -- (v3) -- (v4) -- (v2);

\node[anchor=west] at (v1.east) {$v_1$};
\node[anchor=west] at (v2.east) {$v_2$};
\node[anchor=west] at (v3.east) {$v_3$};
\node[anchor=east] at (v4.west) {$v_4$};
\end{tikzpicture}
\end{center}
    \caption{Non-equivalent link 4-colorings $\alpha_G$ and $\beta_G$ of the paw graph $G$ have distinct component-wise partitions, $P(\alpha_G)=(\set{\set{v_1,v_4},\set{v_2},\set{v_3}})$ and $P(\beta_G)=(\set{\set{v_1,v_3},\set{v_2},\set{v_4}})$.  However, while non-equivalent, these colorings have the same component-wise integer partition, $\lambda(\alpha_G)=\lambda(\beta_G)=(211)$.}\label{fig:paw}
\end{figure}

Algorithm~\ref{algo:link-vertex} assigns colors to some of the edges incident to a given abstract link vertex. The following algorithm, Algorithm~\ref{algo:partitions}, extends this labeling in order to determine the partition of base graph vertices $G$ that correspond to the coloring of link vertex $\alpha$ in graph $\mathcal C\cong \mathcal C_k(G),k>\chi(G)$. Note that Algorithm~\ref{algo:partitions} implicitly uses Algorithm~\ref{algo:link-vertex} in the first step to restrict its inputs to the set of abstract graphs $\mathcal C$ that can be interpreted as $\C\cong\C_k(G)$ with $k>\chi(G)$. Lemmas subsequent to the algorithm establish the uniqueness of the labeling of Algorithm~\ref{algo:partitions} (up to color permutations). Specifically, Lemma~\ref{lem:extract-partition} shows that when $\C\cong\C_k(G)$ for $k>\chi(G)$, Algorithm~\ref{algo:partitions} correctly identifies the unique component-wise partition $P(\alpha_G)$ associated with the link $k$-coloring $\alpha_G$. 

\begin{algorithm}[Partition Extraction]\label{algo:partitions} 
    Given a graph $\C$ and abstract link vertex $\alpha\in V(\C)$, identify a list of partitions $P$ for $\alpha_G$ as follows:
\begin{enumerate}
\item\label{step:vertex-labels} Run Algorithm~\ref{algo:link-vertex} to obtain $G$ with connected components $H_1,\dots,H_d$ and $k>\chi(G)$ such that $\C\cong\C_k(G)$, or abort if reconstruction fails. 
 \item\label{step:cc-and-color} For every $i\in [d]$ and free color $c\in [f_i]$ for $f_i$ as in Step~\ref{step:find-common-colors} of Algorithm~\ref{algo:link-vertex}, let $V_{i,c}\subseteq V(\C)$ be the set of $\abs{V(H_i)}$ neighbors of $\alpha$ corresponding to vertices in $H_i$ being recolored to color $c$ from $\alpha$. 
 \item\label{step:Pi} For every partition $P^i \vdash V_{i,c}$ with $\ell(P^i)=\chi(H_i)$, determine whether to retain $P^i$ as follows. 
\item\label{step:extend-to-hypercube} For each $j\in[\chi(H_i)]$, check whether the neighbors of $\alpha$ corresponding to $P_j^i$  extend to a hypercube in $\C$, and reject $P^i$ and go back to Step~\ref{step:Pi} if any does not. 
\item\label{step:consistent-antipode} For $\alpha_j^i$ antipodal to $\alpha$ on the hypercube associated with $P^i_j$, first check whether $\alpha_j^i$ is an abstract link vertex according to Algorithm~\ref{algo:link-vertex}. If not, reject $P^i$ and return to Step~\ref{step:Pi}. Otherwise, check whether all neighbors of $\alpha_j^i$ in the hypercube correspond to a common free color at $\alpha_j^i$. If not, reject $P^i$ and return to Step~\ref{step:Pi}. Otherwise, retain $P^i$, assign color $f_i+j$ as the color of each $v\in P^i_j$ at $\alpha_G$, and proceed to find a partition for the next connected component. 
\item Return the successfully identified component-wise partition $P=(P^1,\dots,P^d)$.
\end{enumerate}
\end{algorithm}

The following lemma states that Algorithm~\ref{algo:partitions} correctly identifies the unique component-wise partition for a link $k$-coloring when run on the surplus coloring graph $\mathcal{C}\cong \mathcal{C}_{k}(G)$ with $k>\chi(G)$. The result holds regardless of how the base graph vertices $v_1, \dots, v_{\abs{V(G)}}$ are initially labeled in Algorithm~\ref{algo:link-vertex} (provided the labeling is consistent with the adjacency matrix). Figure~\ref{fig:paw-partition} below explains how Algorithm~\ref{algo:partitions} unfolds for the $4$-coloring graph of the paw graph from Figure~\ref{fig:paw}. Note that Algorithm~\ref{algo:link-vertex} could have interchanged the $v_3$ and $v_4$ labels due to the automorphism of $G$.

 \begin{figure}[ht]
\begin{center}
\begin{tikzpicture}
    \node[circle, draw] (A1) at (0,0) {$\alpha$};
    \node[circle, draw, fill=green!70] (B1) at (-.5,1.2) {$\phantom{e}$};
    \node[circle, draw] (B2) at (.5,1.2) {$\phantom{e}$};
    \node[circle, draw,fill=green!70] (C1) at (1.3,0) {$\phantom{e}$};
     \node[circle, draw,fill=green!70] (D1) at (-1.3,0) {$\phantom{e}$};
      \node[circle, draw,fill=green!70] (E1) at (0,-1.3) {$\phantom{e}$};

      \node at (0,0.75) {$v_1$};  
      \node at (0.7,0.2) {$v_2$}; 
      \node at (0.25,-0.7) {$v_3$};
      \node at (-0.7,0.2) {$v_4$};

      \node[circle, draw,color = red] (B2C) at (1.6,1.3) {$\phantom{e}$};
      \node[circle, draw,color = red] (BD) at (-1.6,1.3) {$\phantom{e}$};
      \node[circle, draw,color = red] (B2D) at (-0.8,2.1) {$\phantom{e}$};
      \node[circle, draw,color = red] (BE) at (-0.9,-1) {$\phantom{e}$};
      \node[circle, draw,color = red] (B2E) at (0.9,-1) {$\phantom{e}$};

      \draw[red, dashed] (B2) -- (B2C);
      \draw[red, dashed] (C1) -- (B2C);
      \draw[red, dashed] (B1) -- (BD);
      \draw[red, dashed] (D1) -- (BD);
      \draw[red, dashed] (B2) -- (B2D);
      \draw[red, dashed] (D1) -- (B2D);
      \draw[red, dashed] (B1) -- (BE);
      \draw[red, dashed] (E1) -- (BE);
      \draw[red, dashed] (B2) -- (B2E);
      \draw[red, dashed] (E1) -- (B2E);

    \draw (A1) -- (B1) -- (B2) -- (A1);
    \draw (A1) -- (C1);
    \draw (A1) -- (D1);
    \draw (A1) -- (E1);

     \node at (0, -2) {Steps 1 and 2};
\end{tikzpicture}
\hspace{3pc}
\begin{tikzpicture}
    \node[circle, draw] (A1) at (0,0) {$\alpha$};
    \node[circle, draw] (B1) at (-.5,1.2) {$\phantom{e}$};
    \node[circle, draw] (B2) at (.5,1.2) {$\phantom{e}$};
    \node[circle, draw,fill=purple!20] (C1) at (1.3,0) {$\phantom{e}$};
     \node[circle, draw,fill=purple!20] (D1) at (-1.3,0) {$\gamma$};
      \node[circle, draw] (E1) at (0,-1.3) {$\phantom{e}$};
      \node[circle, draw,fill=purple!20] (BE) at (-0.9,-1) {$\beta$};

      \node at (0,0.75) {$v_1$};  
      \node at (0.7,0.2) {$v_2$}; 
      \node at (0.25,-0.7) {$v_3$};
      \node at (-0.7,0.2) {$v_4$};

      \draw[ dashed] (B1) -- (BE);
      \draw[ dashed] (E1) -- (BE);
    \draw (A1) -- (B1) -- (B2) -- (A1);
    \draw (A1) -- (C1);
    \draw (A1) -- (D1);
    \draw (A1) -- (E1);
    \node at (0, -2) {$\left(\{\{v_1,v_3\},\{v_2\},\{v_4\}\}\right)$};
\end{tikzpicture}
\hspace{3pc}
\begin{tikzpicture}
    \node[circle, draw] (A1) at (0,0) {$\alpha$};
    \node[circle, draw] (B1) at (-.5,1.2) {$\phantom{e}$};
    \node[circle, draw] (B2) at (.5,1.2) {$\phantom{e}$};
    \node[circle, draw,fill=purple!70] (C1) at (1.3,0) {$\phantom{e}$};
     \node[circle, draw] (D1) at (-1.3,0) {$\phantom{e}$};
      \node[circle, draw,fill=purple!70] (E1) at (0,-1.3) {$\phantom{e}$};

      \node at (0,0.75) {$v_1$};  
      \node at (0.7,0.2) {$v_2$}; 
      \node at (0.25,-0.7) {$v_3$};
      \node at (-0.7,0.2) {$v_4$};

      \node[circle, draw,fill=purple!70] (BD) at (-1.6,1.3) {$\phantom{e}$};

      \draw[ dashed] (B1) -- (BD);
      \draw[ dashed] (D1) -- (BD);

    \draw (A1) -- (B1) -- (B2) -- (A1);
    \draw (A1) -- (C1);
    \draw (A1) -- (D1);
    \draw (A1) -- (E1);
    \node at (0, -2) {$\left(\{\{v_1,v_4\},\{v_2\},\{v_3\}\}\right)$};
\end{tikzpicture}
\end{center}
    \caption{Algorithm~\ref{algo:partitions} for the link $k$-coloring $\alpha_G = 1231$ in $\C_4(G)$ from Figure~\ref{fig:paw}. First,  Step~\ref{step:vertex-labels} associates all incident edges at $\alpha\in V(\C_4(G))$ with vertex labels from the base graph consistent with the reconstructed graph $G$. Step~\ref{step:cc-and-color} identifies all of the green highlighted vertices in $\C_4(G)$ as changing to a common free color (in this case color 4), as indicated by the sequence of missing squares between them.  We then start testing partitions of length 3.  The partition $\left(\{\{v_2,v_3\},\{v_1\},\{v_4\}\}\right)$ fails Step~\ref{step:extend-to-hypercube} as there is no hypercube extended from the green vertices labeled $v_2$ and $v_3$. The partition $\left(\{\{v_1,v_3\},\{v_2\},\{v_4\}\}\right)$ survives this step but fails Step~\ref{step:consistent-antipode} as the vertex $\gamma$ corresponding to $\gamma_G = 1234$ in the second picture is not an abstract link vertex. Moreover, while $\beta$ with $\beta_G=4241$ is an abstract link vertex, it fails Step~\ref{step:consistent-antipode} as its return trip to $\alpha = 1231$ does not use a common free color (in this case color 3).  The correct partition $\left(\{\{v_1,v_4\},\{v_2\},\{v_3\}\}\right)$ survives every step in Algorithm~\ref{algo:partitions} as the colored vertices in the third picture are all abstract link vertices that have a common free color returning to $\alpha$.
       }\label{fig:paw-partition}
\end{figure}

 \begin{lemma}\label{lem:extract-partition}
     Let $\C\cong\C_k(G)$ for $k>\chi(G)$. Let $\alpha\in V(\C)$ be an abstract link vertex corresponding to a link $k$-coloring $\alpha_G$ of $G$. Then Algorithm~\ref{algo:partitions}, run with input $\C$ and vertex $\alpha$, returns the unique component-wise partition for $\alpha_G$. In particular, for each connected component $H_i$ of $G$ where $i\in [d]$, exactly one candidate partition $P^i$ will satisfy the conditions in Steps~\ref{step:extend-to-hypercube} and \ref{step:consistent-antipode}.
 \end{lemma}

 \begin{proof} Let $P(\alpha_G)=(P^1,\dots, P^d)$ be the component-wise partition for link $k$-coloring $\alpha_G$.  Fix $i\in[d],j\in[\chi(H_i)]$ and any color $c$ among the $f_i=k-\chi(H_i)$ common free colors for vertices in $H_i$ from coloring $\alpha_G$. Let $(\alpha^i_j)_G$ be the link $k$-coloring obtained from $\alpha_G$ by recoloring every vertex in $P^i_j$ to color $c$, which has the same component-wise partition $P(\alpha_{G})$. By Lemma~\ref{lem:link-equivalence}, the vertex $\alpha_j^i=\Phi^{-1}((\alpha_j^i)_G)$ in $\C$ must be an abstract link vertex.  Moreover, $\alpha=\Phi^{-1}(\alpha_G)$ and $\alpha_j^i$ are antipodal vertices of a $\abs{P^i_j}$-dimensional hypercube in $\mathcal{C}$. This is because the underlying link $k$-coloring $(\alpha_j^{i})_G$ differs from $\alpha_G$ by a partial transposition taking $\alpha_G$'s $j$th part of the $i$th connected component and recoloring it $c$. Furthermore, the neighbors of $\alpha_j^i$ within this hypercube correspond to changing the vertices in $P^i_j$ from color $c$ back to their original common color in $\alpha_G$. Thus, this choice of $P=P(\alpha_G)$ satisfies all the checks in Algorithm~\ref{algo:partitions}.
 
Now, consider a different candidate partition $\tilde P = (\tilde P^1, \dots, \tilde P^d)$ for $\alpha_G$. If $\tilde P^i\neq P^i$ for some $i\in [d]$, then at least one part $\tilde P^{i}_j \in \tilde P^i$ contains two vertices $v$ and $w$ with $\alpha_G(v)\neq \alpha_G(w)$. When Algorithm~\ref{algo:partitions} considers the part $\tilde P^{i}_j$, there are two possibilities. First, neighbors of $\alpha$ corresponding to changing vertices in $\tilde P^{i}_j$ may not survive Step~\ref{step:extend-to-hypercube}. Otherwise, these vertices will extend to a hypercube whose antipodal vertex $\tilde{\alpha}$ corresponds to a coloring $\tilde{\alpha}_G$ where $v$ and $w$ have color $c$. Upon running Algorithm~\ref{algo:link-vertex} at $\tilde{\alpha}$, the path of recolorings returning to $\alpha$ involves changing $v$ back to $\alpha_G(v)$ and $w$ back to $\alpha_G(w)$. Since these colors are different, the neighbors of $\tilde{\alpha}$ on the return path do not correspond to a single common free color, failing Step~\ref{step:consistent-antipode}. Therefore, $P(\alpha_G)$ is the unique component-wise partition for $\alpha_G$ returned by the algorithm. \end{proof}

When applied to an abstract link vertex $\alpha$ in a coloring graph $\C_k(G)$ with $k>\chi(G)$,  Algorithm~\ref{algo:partitions} assigns a specific proper coloring $\alpha_G\colon V(G)\to [k]$. Moreover, by providing free color labels and identifying the correct partition, Algorithm~\ref{algo:partitions} consistently assigns a proper coloring to all hypercube-adjacent abstract link vertices $\alpha^i_j$. However, this consistency is only guaranteed within this local neighborhood. Running the algorithm on an abstract link vertex $\beta$ that is not hypercube-adjacent to $\alpha$ may produce a labeling that is inconsistent with the labeling assigned to $\alpha$ and its hypercube neighborhood. To remedy this limitation, Algorithm~\ref{algo:LLG} in Section~\ref{subsect:LLG} enhances Algorithm~\ref{algo:partitions} through a more sophisticated labeling process that maintains consistency across equivalent abstract link vertices.

All subsequent results in our paper work toward establishing that $\C_k(G)$ for $k>\chi(G)$ can never be isomorphic to $\C_{\chi(G')}(G')$ for any $G,G'$.  Therefore, in the remainder of the paper, we will use the following hypothesis for the sake of contradiction, which uses the results of Lemma~\ref{lem:subgraphs}.

\medskip

\begin{hypothesis}\label{hyp:running} Suppose $\C\cong \C_k(G)\cong \C_{\chi'}(G')$ where $k>\chi(G)$ and $\chi'=\chi(G')$. This pair of isomorphisms, $\C\cong \C_k(G)$ and $\C_k(G)\cong \C_{\chi'}(G')$, will be fixed, as will the underlying colorings of $G$ and $G'$ corresponding to each vertex in $\C$. Assume $H_1, \dots, H_d$ are the connected components of $G$. If $\alpha\in V(\C)$ is an abstract link vertex, then $\alpha_G$ and $\alpha_{G'}$ will be the corresponding link $k$-coloring of $G$ and $\chi'$-coloring of $G'$, respectively.  By Lemma~\ref{lem:subgraphs}, $G'$ contains subgraphs $H'_1, \dots, H'_d$ such that $H'_i\cong H_i$ for each $i\in [d]$. We also recall the embedding $\phi\colon V(G)\to V(G')$ from Lemma~\ref{lem:subgraphs}. 
\end{hypothesis}

\medskip  

Next, we show that Algorithm~\ref{algo:partitions} successfully extracts the unique vertex partition $P^i$ for each induced component $H'_i$ in $G'$ at the same abstract link vertex of $\mathcal{C}$. 

\begin{lemma}\label{lem:partition-G'}  Assume Hypothesis~\ref{hyp:running}. For each $i\in [d]$, $\alpha_{G'}$ restricted to $V(H'_i)$ uses exactly $\chi(H_i)$ colors. Furthermore, there exists an injection $\psi_i\colon [k]\to[\chi']$  such that $\alpha_{G'}\left(\phi(v)\right)=\psi_i(\alpha_{G}(v))$ for each $v\in V(H_i)$. That is, the partition of $V(H'_i)$ induced by $\alpha_{G'}$ is equivalent (via $\phi$ and $\psi_i$) to the partition $P^i$ of $V(H_i)$ induced by $\alpha_G$.
\end{lemma}

\begin{proof}
Lemma~\ref{lem:extract-partition} ensures that Algorithm~\ref{algo:partitions}, run on $\mathcal{C}$ and $\alpha$, correctly identifies the unique component-wise $P(\alpha_{G}) = (P^1, \dots, P^d)$ associated with the link $k$-coloring $\alpha_G$. Since the algorithm relies only on the structure of $\mathcal{C}$, it must identify the same partitions when interpreting $\mathcal{C}$ as $\mathcal{C}_{\chi'}(G')$. We will confirm that the identified partition $P^i$ corresponds to a valid proper coloring of $V(H'_i)$ represented by the coloring $\alpha_{G'}$.

By Lemma~\ref{lem:subgraphs}, each $H_i'$  is an induced subgraph of $G'$. Therefore, the adjacency relations within $H'_i$ are consistent with $H_i$, and the partition $P^i$ defines independent sets \emph{within} $H'_i$. There is only one potential issue: the edges between $V(H'_i)$ and $V(G')\setminus V(H'_i)$ might conflict with the coloring pattern dictated by $P^i$.

However, the structural analysis performed by Algorithm~\ref{algo:partitions} works equally well for the $G'$ interpretation. Suppose $\alpha, \alpha'$ are antipodal abstract link vertices on a hypercube associated with changing an \emph{independent} part $T\in P^i$ from color $c$ to color $\overline{c}$. This structure in $\mathcal{C}$ implies that the corresponding $\chi'$-colorings $\alpha_{G'}$ and $\overline{\alpha}_{G'}$ differ exactly on the vertices of $\phi(T)$, where $\alpha_{G'}(\phi(v))=c$ and $\overline{\alpha}_{G'}(\phi(v))=\overline{c}$ for all $v\in T$. The same analysis holds for all such hypercube transitions identified by Algorithm~\ref{algo:partitions}. Crucially, Algorithm~\ref{algo:partitions} analyzes each component $H_i$ (and thus $H'_i$) independently based on its internal structure and common free colors. It ensures the partition $P^i$ reflects the coloring within $H'_i$. It does not, however, determine relationships between the color sets used for different components $H'_i$ and $H'_j$.

Therefore, for each induced subgraph $H'_i$ of $G'$, the identified partition $P^i$ accurately reflects the color classes within $V(H'_i)$ under the coloring $\alpha_{G'}$. This is formalized by defining, for each component $H_i$, a map $\psi_i\colon [k]\to [\chi']$ that translates the colors used by $\alpha_{G}$ on $H_i$ to those used by $\alpha_{G'}$ on $H'_i$. Concretely, $\psi_i$ maps the color $\alpha_{G}(v)$ to the color $\alpha_{G'}(\phi(v))$ for each $v\in V(H_i)$. The choice of $\psi_i$ is not unique, because it only needs to be prescribed on the set of $\chi(H_i)$ colors used by the link $k$-coloring $\alpha_G$. By the preceding discussion, $\psi_i$ is injective and the relation $\alpha_{G'}\left(\phi(v)\right)=\psi_i(\alpha_{G}(v))$ holds for each $v\in H_i$ by definition.
\end{proof}

Lemma~\ref{lem:partition-G'} guarantees that the internal coloring structure of each subgraph $H'_i$ in $G'$ matches the corresponding $H_i$ in $G$ (up to color permutation) at any abstract link vertex $\alpha$. However, it does not constrain the relationship between the actual sets of colors used. For instance, the set of $\chi(H_i)$ colors used in $H'_i$ might partially or fully overlap with the set of $\chi(H_j)$ colors used in $H'_j$. By analyzing $\mathcal{C}_{\chi'}(G')$ at an abstract link vertex, Algorithm~\ref{algo:partitions} determines the partition within each $H'_i$ consistently. However, we cannot guarantee that the color labels we have chosen for individual $H_i'$ patch together in a consistent way to yield a proper coloring $\alpha_{G'}$ of $G'$. The subtlety arises because there may be additional edges in $G'$, including edges between $u\in V(H'_i)$ and $v\in V(H'_j)$. See Figure~\ref{fig:counterexample} for an example of this phenomenon. 

Another key subtlety arises when applying Lemma~\ref{lem:partition-G'} in the $\mathcal{C}_{\chi'}(G')$ context. The graph $G'$ may contain multiple copies of the subgraph $G$ (or its components $H_i$). Merely identifying a partition corresponding to $H_i'$ at an abstract link vertex $\alpha$ does \emph{not} guarantee we are looking at the same copy of $H_i$ within $G'$ when we move to a different abstract link vertex $\beta$ (see Figure~\ref{fig:changing-G}). However, the algorithm introduced in the next subsection helps resolve this issue by analyzing how different link vertices with the same vertex partition are connected within $\mathcal{C}$. This will effectively fix the specific subgraph of $G'$ under consideration, which allows reasoning about the interaction between this fixed subgraph and the rest of $G'$. We achieve our goal by introducing a new structure, termed \emph{labeled link graph}. The vertices in this new graph are link vertices in a given equivalence class, with an edge between each pair of equivalent link vertices that are antipodal vertices of a hypercube in the underlying coloring graph. As demonstrated in the following subsection, analyzing the labeled link graph associated with a link vertex $\alpha$ in $\mathcal{C}$ removes the ambiguity about which $G$ subgraph of $G'$ is being considered.

\subsection{Equivalent link vertices and labeled link graphs}\label{subsect:LLG}
Building on Algorithm~\ref{algo:partitions}, we now extend the labeling process for a coloring graph with surplus colors. For $\C\cong\C_k(G)$ where $k>\chi(G)$, we aim for a consistent labeling of the hypercube structures connecting equivalent link vertices.
We first establish a series of lemmas demonstrating how equivalent link $k$-colorings are connected within $\mathcal{C}$. 

Let $\alpha_G$ be a link $k$-coloring of graph $G$ with connected components $H_1,\dots,H_d$ where $k>\chi(G)$. Let $i\in[d]$, and let $\sigma$ be a transposition in the symmetric group $S_k$ swapping two colors in $[k]$. 

\begin{lemma}\label{lem:cube-adjacency} Define the coloring $\sigma_i(\alpha_G)$ by applying the transposition $\sigma$ only in $H_i$ by setting $\sigma_i(\alpha_G)(v)=\sigma(\alpha_G(v))$ for $v\in V(H_i)$ and $\sigma_i(\alpha_G)(v)=\alpha_G(v)$ for $v\notin V(H_i)$. Then the corresponding vertices $\alpha$ and $\sigma_i(\alpha)$ of $\C_k(G)$ are in the same connected component. \end{lemma}
\begin{proof}
Suppose without loss of generality that $G$ is connected (so $d=1$) and $\sigma=(12)$. The argument extends to disconnected graphs because $\sigma$ does not affect vertices outside $H_i$. 

\textbf{Case 1.} Neither 1 nor 2 is in the image of $\alpha_G$.

 In this case, $\alpha_G=(12)\alpha_G$, so the corresponding vertices are identical and thus trivially connected. 
 
\textbf{Case 2.} Exactly one of 1 or 2 is in the image of $\alpha_G$. 

Then $\alpha$ and $(12)\alpha$ are antipodal vertices on a cube of dimension $\max\set{\abs{\alpha_G^{-1}(1)},\abs{\alpha_G^{-1}(2)}}$, because the vertices colored 1 in $\alpha_G$ (or 2, if none are 1 in $\alpha_G$) form the dimensions of a hypercube representing their colors independently changing from 1 to 2 (or 2 to 1, if none are 1 in $\alpha_G$) moving from $\alpha$ to $(12)\alpha$. 

\textbf{Case 3.} Both 1 and 2 are in the image of $\alpha_G$. 

Since $k>\chi(G)$, there must be at least one color not used by $\alpha_G$. Without loss of generality, assume that color 3 is not in the image of $\alpha_G$. By the same reasoning from Case 2, the sequence of link colorings $(\alpha,(13)\alpha,(12)(13)\alpha,(23)(12)(13)\alpha)$ are pairwise antipodal vertices connected via hypercubes, and hence $\alpha$ and $(23)(12)(13)\alpha=(12)\alpha$ are connected.
\end{proof}

\begin{figure}[ht]
\begin{tikzpicture}[scale=1, every node/.style={draw=none, fill=none, minimum size=8mm, font=\small}]

\node[draw=none,fill=none] at (-7,1) {\Large $G, k=2$};

\node at (-7,-0.75) {{\Large $\alpha_G$}};
\node[fill=red!30, circle, draw] (u1) at (-8,0) {1};
\node[fill=red!30, circle, draw] (u2) at (-6,0) {1};
\node[anchor=east] at (u1.west) {$v_1$};
\node[anchor=west] at (u2.east) {$v_2$};

\node[draw=none,fill=none] at (0, 0.9) {\Large $G', k'=4 $}; 
\node at (0, -0.75) {{\Large $\alpha'_{G'}$}};
\node[fill=red!30, circle, draw] (v1) at (-1, 0) {1}; 
\node[fill=blue!30, circle, draw] (v2) at (1, 0) {3}; 
\node[anchor=east] at (v1.west) {$v_1$};
\node[anchor=west] at (v2.east) {$v_2$};
\draw (v1) -- (v2);

\node[fill=blue!30, circle, draw] (v3) at (-1, -1.5) {3}; 
\node[fill=red!30, circle, draw] (v4) at (1, -1.5) {1}; 
\node[fill=purple!30, circle, draw] (v5) at (-1, -3) {4}; 
\node[fill=green!30, circle, draw] (v6) at (1, -3) {2}; 
\draw (v3) -- (v4);
\draw (v5) -- (v4);
\draw (v5) -- (v3);
\draw (v6) -- (v3);
\draw (v6) -- (v5);
\draw (v6) -- (v4);

\draw (v6) to[out=45, in=-45] (v2);
\draw (v4) -- (v2);
\draw (v5) to[out=145, in=210] (v1);
\draw (v3) -- (v1);
\end{tikzpicture}
\caption{A $2$-coloring $\alpha_G$ of $G$ consisting of two isolated vertices, where $k=2>\chi(G)$, and a $4$-coloring $\alpha'_{G'}$ of $G'$ with $k'=4=\chi'$. In both $\alpha_G$ and $\alpha'_{G'}$, only $v_1$ and $v_2$ have exactly one free color, and they can be recolored independently. The neighborhood structure around abstract link vertex $\alpha\in\C_k(G)$ and $\alpha'\in\C_{\chi'}(G')$ are identical. Applying Lemma~\ref{lem:partition-G'}, we identify the trivial partitions $\{v_1\}$ for $H_1\cong H'_1$ and $\{v_2\}$ for $H_2\cong H'_2$. This reflects $\alpha'_{G'}(v_1)=1$ and $\alpha'_{G'}(v_2)=3$ in $G'$. However, this local analysis does not reveal the edge $v_1v_2$ in $G'$, so we cannot infer whether $v_1$ and $v_2$ can take the same color.  We note that $G$ and $G'$ will not have equivalent coloring graphs; this example merely illustrates a subtlety between locally identical vertices in $\C_k(G)$ and $\C_{\chi'}(G')$.}
\label{fig:counterexample}
\end{figure}

\begin{figure}[ht]
    \begin{center}
\begin{tikzpicture}[scale=1, every node/.style={draw=none, fill=none, minimum size=8mm, font=\small}]
\node at (0,0) {{\Large $\alpha_{G'}$}};

\node[fill=green!30, circle, draw=red, very thick] (u1) at ({90+72*(1-1)}:1.5cm) {2};
\node[fill=red!30, circle, draw=red,very thick] (u5) at ({90+72*(2-1)}:1.5cm) {1};
\node[fill=green!30, circle, draw] (u4) at ({90+72*(3-1)}:1.5cm) {2};
\node[fill=blue!30, circle, draw] (u3) at ({90+72*(4-1)}:1.5cm) {3};
\node[fill=red!30, circle, draw] (u2) at ({90+72*(5-1)}:1.5cm) {1};
\node[fill=red!30, circle, draw] (u6) at (0,-3) {1};

\draw[very thick, red] (u5) -- (u1);
\draw (u1) -- (u2) -- (u3) -- (u4) -- (u5);
\draw (u4) -- (u6) -- (u3);

  \node[anchor=east]  at (u5.west)  {\{3\}};
    \node[anchor=south]  at (u1.north)  {\{3\}};

\begin{scope}[xshift=6cm]

\node at (0,0) {{\Large $\beta_{G'}$}};

\node[fill=blue!30, circle, draw=red, very thick] (u1) at ({90+72*(1-1)}:1.5cm) {3};
\node[fill=red!30, circle, draw] (u5) at ({90+72*(2-1)}:1.5cm) {1};
\node[fill=green!30, circle, draw] (u4) at ({90+72*(3-1)}:1.5cm) {2};
\node[fill=blue!30, circle, draw] (u3) at ({90+72*(4-1)}:1.5cm) {3};
\node[fill=red!30, circle, draw=red, very thick] (u2) at ({90+72*(5-1)}:1.5cm) {1};
\node[fill=red!30, circle, draw] (u6) at (0,-3) {1};

\draw[very thick, red] (u2) -- (u1);
\draw (u2) -- (u3) -- (u4) -- (u5) -- (u1);
\draw (u4) -- (u6) -- (u3);

\node[anchor=west]  at (u2.east)  {\{2\}};
    \node[anchor=south]  at (u1.north)  {\{2\}};
\end{scope}
\end{tikzpicture}
\end{center}
\caption{A 6-vertex graph $G'$ and two colorings $\alpha_{G'}$ and $\beta_{G'}$ which are adjacent in $\mathcal C_{\chi(G')}(G')$. When run on any vertex in $\mathcal C_{3}(G')$, Algorithm~\ref{algo:link-vertex} will reconstruct the graph $G=P_2$ and $k=3>\chi(G)$, as exactly two adjacent vertices in $G'$ have a common free color and all other vertices are locked. However, the corresponding subgraph $G=P_2$ in $G'$ changes depending on the underlying proper $\chi'$-coloring of $G'$.  }\label{fig:changing-G}
\end{figure}

Let $V_1=\alpha_G^{-1}(1)$ and $V_2=\alpha_G^{-1}(2)$. In Case 3 of Lemma~\ref{lem:cube-adjacency}, the path of connectivity is established by traversing hypercubes. We describe this transformation explicitly by labeling the path of hypercubes connecting $\alpha$ and $(12)\alpha$ as follows:
\begin{equation}\label{eq:cube-chain-labeling}
\alpha \xleftrightarrow{1V_13}(13)\alpha \xleftrightarrow{2V_21} (12)(13)\alpha\xleftrightarrow{3V_12} (23)(12)(13)\alpha=(12)\alpha
\end{equation}
where the notation $c_1Vc_2$ indicates that every vertex in part $V$ changes its color from $c_1$ to $c_2$.

We refer to the equivalence relation introduced after Definition~\ref{def:comp-partition}.  Lemma~\ref{lem:cube-adjacency} extends to show connectivity between \emph{any} two equivalent link $k$-colorings in $\C_k(G)$. This follows because any element in the group $S_k^{d}=S_{k}\times \dots \times S_k$ (acting on the set of link $k$-colorings) can be generated by a sequence of transpositions within components, and Lemma~\ref{lem:cube-adjacency} handles connectivity for each transposition.

\begin{corollary}\label{cor:equivalent-links-connected}
    If $\alpha_G$ and $\beta_G$ are equivalent link $k$-colorings of $G$ with $k>\chi(G)$, then the corresponding vertices in $\C_k(G)$ are in the same connected component.
\end{corollary}

The proof of Lemma~\ref{lem:cube-adjacency} also reveals how equivalent colorings differing on partial transpositions are connected via hypercubes of a certain dimension. Every nontrivial transposition of $\alpha$ is connected to $\alpha$ by either one hypercube (Case 2), or a chain of three hypercubes (Case 3). Furthermore, in Case 3 (where $\alpha_G$ uses $1$ and $2$, and we assume $3$ is free), there exists a second path of three hypercubes connecting $\alpha$ to $(12)\alpha$:
\begin{equation}\label{eq:cube-chain-labeling2}
\alpha \xleftrightarrow{2V_23}(23)\alpha \xleftrightarrow{1V_12} (12)(23)\alpha\xleftrightarrow{3V_21} (13)(12)(23)\alpha=(12)\alpha.
\end{equation}
The two paths \eqref{eq:cube-chain-labeling} and \eqref{eq:cube-chain-labeling2} are internally disjoint; their union creates a \emph{6-cycle of cubes}. Figure~\ref{fig:coloring-graph-labeling} illustrates how $\C_3(P_3)$ is a 6-cycle of cubes of alternating dimensions 1 and 2. This holds more generally for $\C_3(K_{n_1,n_2})$, noting that $P_3\cong K_{1,2}$, as formalized by the following corollary and illustrated in Figure~\ref{fig:6cyclecube}.

\begin{corollary}\label{cor:equivalent-links-connected-hypercubes}
    Let $\alpha_G$ be any link $k$-coloring of $G$ with components $H_1,\dots,H_d$ and $k>\chi(G)$. Suppose $\alpha_G$ has component-wise integer partition vector $\lambda(\alpha_G)=(\lambda^1,\dots,\lambda^d)$ where $\lambda^i\vdash |V(H_i)|$ and $\ell(\lambda^i)=\chi(H_i)$. Then for each $i\in[d]$ and any two indices $1\le j_1<j_2\le \chi(H_i)$, the coloring graph $\C_k(G)$ contains $\C_3(K_{\lambda^i_{j_1},\lambda^i_{j_2}})$ as a subgraph.
\end{corollary}

Figure~\ref{fig:6cyclecube} shows an example of a 6-cycle of cubes, where for a base graph $G$ the parts defined by $V_1=\{v_4,v_5,v_6\}$ and $V_2=\{v_1,v_3\}$ cycle through colors 1,3, and 4. As $|V_1|=3$ and $|V_2|=2$, the resulting subgraph in $\mathcal{C}_4(G)$ is isomorphic to $\mathcal{C}_3(K_{3,2})$. The six abstract link vertices are named \emph{links} because they serve as common vertices shared by a family of cycles within $\C_3(K_{|V_1|, |V_2|})$. Each cycle in this family is composed of three paths of length $|V_1|$ and three paths of length $|V_2|$, with each path traversing one of the six cubes, for a total length of $3(|V_1|+|V_2|)$. This entire structure $\C_3(K_{|V_1|, |V_2|})$, visualized as a 6-cycle of cubes, is an induced subgraph of $\mathcal{C}_k(G)$. Moreover, every 6-cycle of cubes is also ``induced" from a hypercube perspective. This means the six link vertices are only adjacent to each other (via hypercubes) as dictated by the $6$-cycle structure, with each vertex connecting to exactly two others.

\begin{figure}[ht]
\centering
\includegraphics[width=1\linewidth]{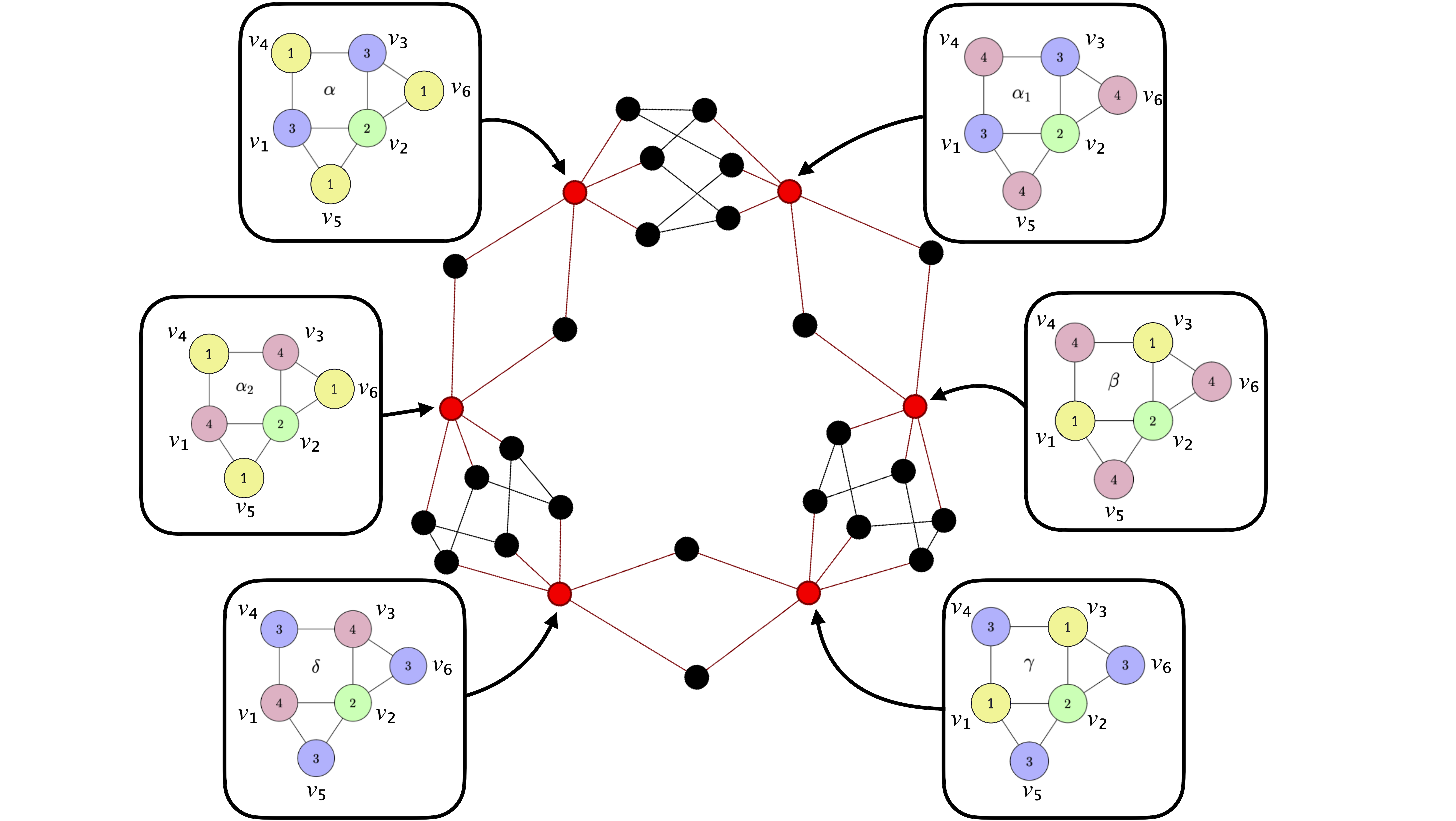}
    \caption{The 6-cycle of cubes defined by abstract link vertex $\alpha$, free color $4$, and parts $V_1=\{v_4,v_5,v_6\}$ and $V_2=\{v_1,v_3\}$.  The 6 red vertices are the 6 abstract link vertices that participate in the 6-cycle.  All black vertices are the parts $V_1$ or $V_2$ independently changing colors to reach the next abstract link vertex.  In our notation, this 6-cycle of cubes is described as  $\alpha \xleftrightarrow{1V_14}  \alpha_1\xleftrightarrow{3V_21} \beta\xleftrightarrow{4V_13}\gamma \xleftrightarrow{1V_24} \delta\xleftrightarrow{3V_11} \alpha_2\xleftrightarrow{4V_23} \alpha$. }
    \label{fig:6cyclecube}
\end{figure}

The results in this subsection so far pertain to structures in coloring graphs with surplus colors. Algorithm~\ref{algo:LLG} describes a procedure to identify hypercubes associated with some abstract link vertex $\alpha$ to propagate consistent cube labels, even without assuming surplus colors. The algorithm begins with an abstract link vertex $\alpha$ and the color and vertex labels for its base graph as provided by Algorithms~\ref{algo:link-vertex}~and~\ref{algo:partitions}. By iteratively rerunning Algorithm~\ref{algo:partitions} and analyzing 6-cycles of cubes, Algorithm~\ref{algo:LLG} ensures that color and vertex labels are consistent across neighboring link vertices. This process continues until the entire labeled link graph associated with $\alpha$ is generated and labeled. In Lemma~\ref{lem:unique-LLG-labeling}, we prove that labeling decisions made by Algorithm~\ref{algo:LLG} are uniquely determined for surplus coloring graphs (up to the choice of a bijection $\alpha\mapsto \alpha_{G}$).

Recall that the group $S_k^{d}$ acts on the set of proper $k$-colorings of $G$. If $\alpha_G$ is a link $k$-coloring, then $\sigma\cdot\alpha_{G}$ is also a link $k$-coloring. Hence, the action descends to the set of link $k$-colorings. Given a link $k$-coloring $\alpha_{G}$ in $G$, we denote by $[\alpha_{G}]$ the equivalence class of $\alpha_{G}$ (that is, the orbit of $\alpha_G$ under this group action). Suppose $\C$ is an abstract coloring graph arising from a graph $G$ with surplus colors; that is, $\C\cong\C_k(G)$ where $k>\chi(G)$. Using Lemma~\ref{lem:link-equivalence}, there is a bijection between abstract link vertices $\alpha$ in $\C$ and link $k$-colorings $\alpha_G$ in $G$. We define $[\alpha]$ as the set of vertices in $\C$ that represent the colorings in 
$[\alpha_G]$ under this bijection. 

\begin{definition}\label{def:labeled-link-graph}
Suppose we are given $\mathcal{C}\cong\C_k(G)$ along with the map $\alpha\mapsto \alpha_G$. Assume  $k>\chi(G)$ and we have an abstract link vertex $\alpha$ corresponding to link $k$-coloring $\alpha_{G}$ with component-wise partition $P(\alpha_G)=(P^1,\dots,P^d)$.  Define $\mathcal{L}_{[\alpha]}$ as an edge-labeled graph where:
   \begin{enumerate}
    \item Vertices of $\mathcal{L}_{[\alpha]}$: abstract link vertices $\beta$ in $[\alpha]$, that is, $\beta_{G}$ is equivalent to $\alpha_G$.
    \item Edges of $\mathcal{L}_{[\alpha]}$: $\beta$ and $\gamma$ are adjacent if $\beta_G$ and $\gamma_G$ differ in the color of exactly one part $P^i_j$.
    \item Edge labels: If $\beta_G$ and $\gamma_G$ differ on the part $P^i_{j}$ with $\beta_{G}(P^{i}_{j})=c_1$ and $\gamma_{G}(P^{i}_{j})=c_2$, we label the edge $\beta\gamma$ by:
    \[
    \beta\xleftrightarrow{c_1 P^i_{j}c_2}  \gamma. 
    \]
\end{enumerate}
We refer to $\mathcal{L}_{[\alpha]}$ as a \emph{labeled link graph} of $\mathcal{C}_{k}(G)$ corresponding to an equivalence class $[\alpha]$.
\end{definition}

The object $\mathcal{L}_{[\alpha]}$ is not intrinsic to the base graph $G$. There is a subtle dependence on the choice of the bijection $\alpha \mapsto \alpha_{G}$ that concretely realizes a vertex $\alpha$ in $V(\mathcal{C})$ as a function $\alpha_G\colon V(G)\to [k]$. 

\begin{lemma}
   The graph $\mathcal{L}_{[\alpha]}$, disregarding its edge labels, is isomorphic to $\mathcal{C}_k\left(\biguplus_{i\in [d]} K_{\chi(H_i)}\right)$.
\end{lemma}

A general fact \cite{BFHRS16}*{Lemma 1} states that the coloring graph of the disjoint union, $\mathcal{C}_k\left(\biguplus_{i\in [d]} G_i\right)$, is isomorphic to a Cartesian product of individual coloring graphs, $\square_{i\in[d]}\mathcal{C}_{k}(G_i)$. 

\begin{proof}
A link $k$-coloring $\alpha_G$ (and any $\beta_G \in [\alpha_G]$) uses a specific set of $\chi(H_i)$ distinct colors for each connected component $H_i$ of $G$. These $\chi(H_i)$ colors define the parts $P^i_1, \dots, P^i_{\chi(H_i)}$ of the component-wise partition $P^i$ for $H_i$. In $\mathcal{L}_{[\alpha]}$, an edge corresponds to changing the color assigned to one such part, say $P^i_j$, to a new color, while the colors of all other parts $P^s_t$ (for $(s,t) \neq (i,j)$) remain unchanged.

Consider the graph $K_{\chi(H_i)}$. Its vertices can be identified with the abstract ``slots" for the $\chi(H_i)$ colors used in component $H_i$. A $k$-coloring of $K_{\chi(H_i)}$ assigns one of $k$ colors to each of these $\chi(H_i)$ slots such that all slots receive distinct colors. An edge in $\mathcal{C}_k(K_{\chi(H_i)})$ connects two such colorings if they differ in the color assigned to exactly one slot.

This is precisely the structure of transitions within each component $H_i$ as reflected in $\mathcal{L}_{[\alpha]}$: changing the color of a part $P^i_j$ (a ``slot") to a new color (from the $k$ available, distinct from colors of other parts $P^i_m$ in $H_i$) corresponds to an edge in $\mathcal{C}_k(K_{\chi(H_i)})$. Since these changes can be made independently for each component $H_i$ of $G$ to move between colorings in $[\alpha_G]$, the overall structure of $\mathcal{L}_{[\alpha]}$ (unlabeled) is the Cartesian product of these individual coloring graphs, $\square_{i\in[d]}\mathcal{C}_{k}(K_{\chi(H_i)})$. This is isomorphic to the $k$-coloring graph of the disjoint union of these complete graphs, $\mathcal{C}_k\left(\biguplus_{i\in [d]} K_{\chi(H_i)}\right)$.
\end{proof}

This structure implies that all of the 6-cycles of cubes that connect equivalent link $k$-colorings in $\C_k(G)$ become 6-cycles in the labeled link graph.  Identifying (and ultimately labeling) these 6-cycles is a key step in constructing the labeled link graph. The following lemma states that if two consecutive edges in the 6-cycle are specified, then there is a unique continuation of the path that completes this 6-cycle. This uniqueness holds as long as the two edges are not also part of a 3- or 4-cycle.  This condition is satisfied if the labeling of these edges is of the form $\alpha_1 
 \xleftrightarrow{c_3 P^i_{j_1}c_1}  \alpha  \xleftrightarrow{c_2 P^i_{j_2} c_3} \alpha_2$, where, from $\alpha$, two different parts in the same connected component are sent to the same free common color $c_3$.

\begin{lemma}\label{lem:length-4-path}
Suppose $\alpha_1, \alpha_2\in [\alpha]$ are two vertices satisfying the following adjacency condition in $\mathcal{L}_{[\alpha]}$: 
\[ \alpha_1 
 \xleftrightarrow{c_3 P^i_{j_1}c_1}  \alpha  \xleftrightarrow{c_2 P^i_{j_2} c_3} \alpha_2.
\]
Then there is a unique shortest path from $\alpha_1$ to $\alpha_2$ in $\mathcal{L}_{[\alpha]}$ that avoids $\alpha$, which has length $4$ given by:
\[
\alpha_1\xleftrightarrow{c_2P^i_{j_2}c_1} \beta\xleftrightarrow{c_3P^i_{j_1}c_2}\gamma \xleftrightarrow{c_1P^i_{j_2}c_3} \delta\xleftrightarrow{c_2P^i_{j_1}c_1} \alpha_2
\]
for some vertices $\beta, \gamma, \delta$.
\end{lemma}

\begin{proof} 
Let $t$ denote the length of the shortest $\alpha$-avoiding path from $\alpha_1$ to  $\alpha_2$. We claim that $t=4$. Suppose, to the contrary, that $t\leq 3$. If $t=3$, then together with $\alpha$, we have an induced $5$-cycle in $\mathcal{L}_{[\alpha]}$. This is a contradiction, because $C_5$ is not an induced subgraph of any coloring graph \cite{BFHRS16}*{Corollary 12}. We may now assume $t\leq 2$. Since $\alpha_1$ and $\alpha_2$ are not adjacent, $t\neq 1$. Consequently, $t=2$ and there exists a vertex $\zeta\neq \alpha$ between $\alpha_1$ and $\alpha_2$ in $\mathcal{L}_{[\alpha]}$:  
\[ \alpha_1 
 \xleftrightarrow{?}  \zeta  \xleftrightarrow{?} \alpha_2.
\]
Now, $(\alpha_1)_{G}$ and $(\alpha_2)_{G}$ differ exactly on the two parts $P^i_{j_1}$ and $P^i_{j_2}$. By inspection, no such $\zeta$ can exist, as it would require both $P_{j_1}^i$ and $P_{j_2}^i$ to be colored $c_3$, producing a proper coloring that uses fewer colors than the chromatic number. Thus, $t=4$. The following $\alpha$-avoiding path of length $4$ from $\alpha_1$ to $\alpha_2$ certainly exists:
\begin{equation}\label{eq:length-4-path}
\alpha_1\xleftrightarrow{c_2P^i_{j_2}c_1} \beta\xleftrightarrow{c_3P^i_{j_1}c_2}\gamma \xleftrightarrow{c_1P^i_{j_2}c_3} \delta\xleftrightarrow{c_2P^i_{j_1}c_1} \alpha_2
\end{equation}
Next, we show that \eqref{eq:length-4-path} is the unique such path of length $4$. 

Consider an arbitrary $\alpha$-avoiding path of length $4$ from $\alpha_1$ to $\alpha_2$. When combined with the edges $(\alpha, \alpha_1)$ and $(\alpha_2, \alpha)$, we obtain an induced $6$-cycle $(\alpha, \alpha_1, \beta, \gamma, \delta, \alpha_2, \alpha)$ in $\mathcal{L}_{[\alpha]}$. By \cite{ABFR18}, an induced 6-cycle in a coloring graph arises in one of two situations: (a) two vertices alternately swapping between three colors, or (b) three vertices independently swapping between two colors. Case (b) always occurs as part of a $3$-cube. Our setup involves two parts $P_{j_1}^{i}, P_{j_2}^{i}$, and specifies two colors for each of these parts (colors $c_1$ and $c_3$ for $P_{j_1}^i$ and colors $c_2$ and $c_3$ for $P_{j_2}^i$).  If we had case (b), the 3-cube allows all color changes independently, so we would have a vertex where both $P_{j_1}^i$ and $P_{j_2}^i$ are colored $c_3$, which would result in a proper coloring using fewer colors than the chromatic number, a contradiction. Therefore, we must have case (a). The intermediate vertices of the 6-cycle are uniquely determined by the two involved parts and three colors. Thus, the remaining vertices $\beta, \gamma, \delta$ in any induced 6-cycle formed this way are unique, proving that \eqref{eq:length-4-path} is the only shortest $\alpha$-avoiding path of length 4 from $\alpha_1$ to $\alpha_2$. \end{proof}

For an abstract link vertex $\alpha\in V(\C)$, a modified use of Algorithm~\ref{algo:partitions} can be applied to find all equivalent link vertices, which are precisely the vertices of $\mathcal{L}_{[\alpha]}$.  After identifying the partition $P(\alpha_G)$ of $\alpha_G$ using Algorithm~\ref{algo:partitions}, it is clear that changing every vertex in any one part $P^i_j$ to any common free color results in an equivalent link vertex, $\beta$.  This process can be repeated at $\beta$ to discover additional equivalent link vertices. Given the connectivity of the labeled link graph, conducting this search in breadth-first order guarantees that we will find all vertices in $[\alpha]$.

The subtlety in this approach is that although Algorithm~\ref{algo:partitions} works equally well on both $\alpha$ and $\beta$, the exact labeling of the partitions is not necessarily consistent.  Even with only one connected component, there is a potential mismatch. Suppose $P(\alpha_G)=P^1$ with the partition $P^1=P^1_1\uplus P^1_2 \uplus P^1_3$, where $\beta_G$ is obtained by recoloring vertices in $P^1_1$. There is no guarantee that Algorithm~\ref{algo:partitions} run on $\beta$ will label $P^1_2$ as $P^1_2$ again; it might instead be labeled $P^1_3$. 

The following algorithm, Algorithm~\ref{algo:LLG}, addresses this subtlety by not only identifying all of the vertices in $\mathcal{L}_{[\alpha]}$ but also consistently labeling the parts and colors (and thus the edges in $\mathcal{L}_{[\alpha]}$).   Given an initial abstract link vertex $\alpha$, Algorithms~\ref{algo:link-vertex} and~\ref{algo:partitions} are run to obtain a canonical labeling of vertices, colors, and the partition in Steps~\ref{step:alg1-labels} and~\ref{step:alg2-labels}, aborting if the input graph is not isomorphic to a coloring graph with surplus colors. A breadth-first search of all vertices in $\mathcal{L}_{[\alpha]}$ is then initialized in Steps~\ref{step:initQ} and~\ref{step:forloop}.  In these steps, every edge incident to $\alpha$ in $\mathcal{L}_{[\alpha]}$ is labeled using the canonical labeling from the previous steps.  

The final Step~\ref{step:findalledges} runs the breadth-first search until all vertices are found, while consistently labeling all edges of $\mathcal{L}_{[\alpha]}$. Given a complete set of labeled edges incident to a vertex $\gamma$ in $\mathcal{L}_{[\alpha]}$, including the edge $\gamma\beta$, we analyze certain cycles to label all incident edges to $\beta$ consistently. There are four such cases, which are depicted in Figure~\ref{fig:cycleoptions} and described in the ensuing algorithm. In the algorithm, the breadth-first ordering guarantees that when labeling the edges incident to some $\beta$, there is at least one predecessor $\gamma$ (i.e., a neighbor of $\beta$ that is closer to $\alpha$) whose incident edges have already been labeled.  

\begin{algorithm}[Labeled Link Graph Identification]\label{algo:LLG}
    Given a graph $\C$ and an abstract link vertex $\alpha\in V(\C)$, identify $[\alpha]$ and construct its labeled link graph $\mathcal{L}_{[\alpha]}$ as follows:
\begin{enumerate}

\item\label{step:alg1-labels} Run Algorithm~\ref{algo:link-vertex} to obtain $G$ (or abort) with components $H_1,\dots,H_d$, labels on incident edges indicating color-changing vertices and the common free color palette $[f_i]$ for each $v\in H_i$ from $\alpha$.
\item\label{step:alg2-labels}  Run Algorithm~\ref{algo:partitions} to assign canonical color labels $\{f_i+1,\dots,f_i+\chi(H_i)\}$ to the colors used in $\alpha$ for vertices within each component $H_i$, consistent with its component-wise partition $P_\alpha$.
\item\label{step:initQ} Initialize the vertex set $V=\{\alpha\}$ and the edge set $E=\emptyset$ for $\mathcal{L}_{[\alpha]}$.  Initialize an empty queue $Q$.
\item\label{step:forloop}  Iterate for each $i\in[d]$, each color $c\in [f_i]$, and  each index $1\leq j\leq \chi(H_i)$: Let $c_j=f_i+j$ denote the color of $P_{j}^i$ at $\alpha$ assigned by Algorithm~\ref{algo:partitions}.  Let $\beta$ be the abstract link vertex corresponding to recoloring part $j$ with $c$ from $\alpha$.
 Add $\beta$ to $V$ and $Q$, and $\alpha\beta$ to $E$ with label  $\alpha \xleftrightarrow{c_jP^i_{j}c}  \beta$. 
\item\label{step:findalledges} While $Q$ is nonempty: remove the next element $\beta$ from queue $Q$ and find a predecessor element $\gamma$ in $V\setminus Q$ such that $\beta\gamma\in E$. Let $c_jP^i_jc$ be the existing label of this edge $\gamma\beta$ so that $c_j$ is the color of $P^i_j$ at $\gamma$ and $c$ is the color of $P^{i}_{j}$ at $\beta$. For each $1\leq \ell\leq d$, let $\mathcal{F}^{\ell}$ be the common free color palette at $\gamma$ for $H_{\ell}$; note that $c\in\mathcal{F}^i$ and $c_j\notin\mathcal{F}^i$.  We label all currently unlabeled incident edges to $\beta$ (and possibly other edges) by analyzing incident edges to $\gamma$ as follows:
\begin{itemize}
    \item[(a)]\label{step:alg3:six-cycles} {\bf Iterate through incident edges to $\gamma$ where other parts in $H_i$ change to the same free color $c$}: 
    
    For each $1\le k\le \chi(H_i)$ with $k\neq j$: use existing labels on edges incident to $\gamma$ to identify the unique $\delta$ such that $\gamma \xleftrightarrow{c_kP^i_{k}c}  \delta$.  By exploring hypercube adjacent neighbors, and using Algorithm~\ref{algo:partitions} to confirm that each vertex we consider is in $[\alpha]$, find the shortest sequence of vertices in $[\alpha] \setminus \{\gamma\}$ from $\beta$ to $\delta$ where two consecutive vertices are antipodal points of hypercubes of alternating dimensions $|P_{j}^{i}|$ and $|P_{k}^i|$.  By Lemma~\ref{lem:length-4-path},  this path is $\beta,\beta_1,\beta_2,\beta_3,\delta$. If not already in $V$, add $\beta_1,\beta_2,\beta_3$ (in this order) to $V$ and $Q$.  Add the following edges to $E$ with labels 
\[\gamma \xleftrightarrow{c_jP^i_{j}c}  \beta\xleftrightarrow{c_kP^i_{k}c_j} \beta_1\xleftrightarrow{cP^i_{j}c_k}\beta_2 \xleftrightarrow{c_jP^i_{k}c} \beta_3\xleftrightarrow{c_kP^i_{j}c_j} \delta\xleftrightarrow{cP^i_{k}c_k}  \gamma.
\]
\item[(b)]\label{step:alg3:squares-in-component}{\bf Iterate through incident edges to $\gamma$ where other parts in $H_i$ change to a different free color $\tilde c\ne c$}: 

For each $1\le k\le \chi(H_i)$ with $k\neq j$ and $\tilde{c}\in\mathcal{F}^i\setminus\{c\}$: use existing labels on edges incident to $\gamma$ to identify the unique $\delta$ such that $\gamma \xleftrightarrow{c_kP^i_{k}\tilde{c}}  \delta$. As this color change is independent of recoloring $c_jP^i_jc$ on the edge $\gamma\beta$, there is a unique vertex $\beta_1\in[\alpha]$ hypercube adjacent to both $\beta$ and $\delta$.  If not already in $V$, add $\beta_1$ to $V$ and $Q$, and add the following edges to $E$ with labels
\[\gamma \xleftrightarrow{c_jP^i_{j}c}  \beta\xleftrightarrow{c_kP^i_{k}\tilde{c}} \beta_1\xleftrightarrow{cP^i_{j}c_j}\delta\xleftrightarrow{\tilde{c}P^i_{k}c_k}  \gamma.
\]
\item[(c)]\label{step:alg3:cliques}{\bf Iterate through incident edges to $\gamma$ where $P^i_j$ changes to a different free color $\tilde c\ne c$}: \\ For each index $\tilde{c}\in\mathcal{F}^i\setminus\{c\}$: use existing labels on edges incident to $\gamma$ to identify the unique $\delta$ such that $\gamma \xleftrightarrow{c_j P^i_{j}\tilde{c}}  \delta$.  This abstract link vertex $\delta$ is necessarily hypercube adjacent to $\beta$, so we add the edge $\delta\xleftrightarrow{\tilde{c}P^i_{j}c}  \beta$ to $E$.
\item[(d)]\label{step:alg3:squares-out-component}{\bf Iterate through incident edges to $\gamma$ where parts not in $H_i$ change to any free color}: 

For each $\ell\in[d]\setminus\{i\}$ and $\tilde{c}\in\mathcal{F}^\ell$: use existing labels on edges incident to $\gamma$ to identify the unique $\delta$ such that $\gamma \xleftrightarrow{c_sP^\ell_{k}\tilde{c}}  \delta$. As this color change is independent of recoloring $c_jP^i_jc$ on the edge $\gamma\beta$, there is a unique vertex $\beta_1\in[\alpha]$ hypercube adjacent to both $\beta$ and $\delta$.  If not already in $V$, add $\beta_1$ to $V$ and $Q$, and add the following edges to $E$ with labels
\[\gamma \xleftrightarrow{c_jP^i_{j}c}  \beta\xleftrightarrow{c_sP^\ell_{k}\tilde{c}} \beta_1\xleftrightarrow{cP^i_{j}c_j}\delta\xleftrightarrow{\tilde{c}P^\ell_{k}c_s}  \gamma.
\]
\end{itemize}
\end{enumerate}
\end{algorithm}

\begin{figure}[ht]
\begin{tikzpicture}[
    vertex/.style={},
    edge_label/.style={
        sloped,    
        midway,    
        auto,      
        font=\small 
    },
    edge_style/.style={
        <->, 
        >=Latex  
    }
  ]
  
  \def\diagramradius{1.7cm} 

  \node[vertex] (gamma)  at (90:\diagramradius) {$\gamma$};
  \node[vertex] (beta) at (30:\diagramradius) {$\beta$};
  \node[vertex] (beta_1)   at (-30:\diagramradius) {$\beta_1$};
  \node[vertex] (beta_2)  at (-90:\diagramradius) {$\beta_2$};
  \node[vertex] (beta_3)  at (-150:\diagramradius) {$\beta_3$};
  \node[vertex] (delta) at (150:\diagramradius) {$\delta$}; 
   
  \draw[edge_style, blue] (gamma) -- node[edge_label] {$c_jP^i_{j}c$} (beta);
  \draw[edge_style, dashed, red] (beta) -- node[edge_label] {$c_kP^i_{k}c_j$} (beta_1);
  \draw[edge_style, dashed] (beta_1) -- node[edge_label, swap] {$c_kP^i_{j}c$} (beta_2); 
  \draw[edge_style, dashed] (beta_2) -- node[edge_label, swap] {$cP^i_{k}c_j$} (beta_3); 
  \draw[edge_style, dashed] (beta_3) -- node[edge_label] {$c_k P^i_{j}c_j$} (delta);
  \draw[edge_style] (delta) -- node[edge_label] {$cP^i_{k}c_k$} (gamma);
\end{tikzpicture}
\quad 
\begin{tikzpicture}[
    vertex/.style={},
    edge_label/.style={
        sloped,    
        midway,    
        auto,      
        font=\small 
    },
    edge_style/.style={
        <->, 
        >=Latex  
    }
  ]
  
  \def\diagramradius{1.7cm} 

  \node[vertex] (gamma_sq) at (90:\diagramradius)  {$\gamma$};
  \node[vertex] (beta_sq)  at (0:\diagramradius)   {$\beta$};
  \node[vertex] (beta1_sq) at (-90:\diagramradius) {$\beta_1$};
  \node[vertex] (delta_sq) at (180:\diagramradius) {$\delta$};
  
  \draw[edge_style, blue] (gamma_sq) -- node[edge_label] {$c_jP^i_{j}c$} (beta_sq);
  \draw[edge_style, dashed, red] (beta_sq)  -- node[edge_label, swap] {$\tilde{c} P^i_{k}c_k$} (beta1_sq);
  \draw[edge_style, dashed] (beta1_sq) -- node[edge_label, swap] {$c_j P^i_{j}c$} (delta_sq);
  \draw[edge_style] (delta_sq) -- node[edge_label] {$\tilde{c}P^i_{k}c_k$} (gamma_sq);
\end{tikzpicture}
\quad 
\begin{tikzpicture}[
    vertex/.style={}, 
    edge_label/.style={
        sloped,    
        midway,    
        auto,      
        font=\small 
    },
    edge_style/.style={
        <->, 
        >=Latex  
    }
  ]
  
  \def\diagramradius{1.6cm} 

  \node[vertex] (delta_tri) at (90:\diagramradius)    {$\delta$};
  \node[vertex] (beta_tri)  at (-30:\diagramradius)   {$\beta$};
  \node[vertex] (gamma_tri) at (210:\diagramradius)   {$\gamma$};
  
  \draw[edge_style, red, dashed] (delta_tri) -- node[edge_label] {$\tilde{c}P^i_{j}c$} (beta_tri);
  \draw[edge_style, blue] (beta_tri)  -- node[edge_label, swap] {$c_j P^i_{j}c$} (gamma_tri);
  \draw[edge_style] (gamma_tri) -- node[edge_label] {$c_jP^i_{j}\tilde{c}$} (delta_tri);
\end{tikzpicture}
\quad 
\begin{tikzpicture}[
    vertex/.style={},
    edge_label/.style={
        sloped,    
        midway,    
        auto,      
        font=\small 
    },
    edge_style/.style={
        <->, 
        >=Latex  
    }
  ]
  
  \def\diagramradius{1.7cm} 
  \node[vertex] (gamma_sq_new) at (90:\diagramradius)  {$\gamma$};
  \node[vertex] (beta_sq_new)  at (0:\diagramradius)   {$\beta$};
  \node[vertex] (beta1_sq_new) at (-90:\diagramradius) {$\beta_1$};
  \node[vertex] (delta_sq_new) at (180:\diagramradius) {$\delta$};
  \draw[edge_style, blue] (gamma_sq_new) -- node[edge_label] {$c_jP^i_{j}c$} (beta_sq_new);
  \draw[edge_style, red, dashed] (beta_sq_new)  -- node[edge_label, swap] {$\tilde{c} P^{\ell}_{k}c_s$} (beta1_sq_new);
  \draw[edge_style, dashed] (beta1_sq_new) -- node[edge_label, swap] {$c_j P^i_{j}c$} (delta_sq_new);
  \draw[edge_style] (delta_sq_new) -- node[edge_label] {$\tilde{c}P^{\ell}_{k}c_s$} (gamma_sq_new);
\end{tikzpicture}

\caption{Illustration of Algorithm~\ref{algo:LLG}, Step~\ref{step:findalledges}, parts (a)-(d). The blue arrow indicates the edge from $\beta$ to its predecessor $\gamma$. Solid edges are already labeled; dashed edges are in the process of being labeled (although some may have been previously labeled). Processing the red dashed edge $\beta\leftrightarrow \beta_1$ at each step eventually exhausts all incident edges to $\beta$.}
\label{fig:cycleoptions}
\end{figure}

As with our previous algorithms, we first show that the labeling produced by this algorithm is unique in some sense with respect to a fixed labeling of $\C_k(G)$ in the case when $k>\chi(G)$. Subsequently, we will show that it implies something structural about a version of the labeled link graph in the $\C_{\chi'}(G')$, $\chi'=\chi(G')$ interpretation under Hypothesis~\ref{hyp:running}.

In Definition~\ref{def:labeled-link-graph}, the edge labeling in $\mathcal{L}_{[\alpha]}$ is uniquely determined because we implicitly fix a bijection $\alpha\mapsto\alpha_G$. Since we want Algorithm~\ref{algo:LLG} to apply to an abstract coloring graph $\mathcal{C}$, there could be multiple labelings of $\mathcal{L}_{[\alpha]}$, depending on the choice of bijection $\alpha\mapsto\alpha_G$. Algorithms~\ref{algo:link-vertex} and \ref{algo:partitions} fix such a mapping, and Algorithm~\ref{algo:LLG} labels the labeled link graph consistent with this choice in its Step~\ref{step:alg1-labels} and Step~\ref{step:alg2-labels}.

\begin{lemma}\label{lem:unique-LLG-labeling} Let $\C\cong \C_k(G)$ with $k>\chi(G)$ and let $\alpha$ be an abstract link vertex. Then Algorithm~\ref{algo:LLG}, when run on $\mathcal{C}$ and $\alpha$, exactly identifies $\mathcal{L}_{[\alpha]}$, where the edge labeling in $\mathcal{L}_{[\alpha]}$ is consistent with the labeling of $\mathcal{C}$ produced by Algorithm~\ref{algo:link-vertex} in Step~\ref{step:alg1-labels}. 
\end{lemma}

\begin{proof}
First, Algorithm~\ref{algo:LLG} finds all vertices of $\mathcal{L}_{[\alpha]}$, because all equivalent link $k$-colorings of $G$ are connected by a sequence of hypercubes by Corollaries~\ref{cor:equivalent-links-connected} and \ref{cor:equivalent-links-connected-hypercubes}. Second, Algorithm~\ref{algo:LLG}  correctly identifies each abstract link vertex as being equivalent to $\alpha$ by the correctness of Algorithm~\ref{algo:partitions}. It remains to show that Algorithm~\ref{algo:LLG} consistently labels all edges in $\mathcal{L}_{[\alpha]}$.

Every vertex in $\mathcal{L}_{[\alpha]}$ has degree $\sum_{\ell=1}^d\chi(H_{\ell})(k-\chi(H_{\ell}))$. Indeed, for each connected component $H_{\ell}$, any of its $\chi(H_{\ell})$ parts can change to any of the $k-\chi(H_{\ell})$ free colors. Step~\ref{step:findalledges} finds and labels all of these edges, as (a) labels $\chi(H_i)-1$ edges; (b) labels $(\chi(H_i)-1)(k-\chi(H_i)-1)$ edges; (c) labels $(k-\chi(H_i)-1)$ edges; (d) labels $\sum_{\ell\neq i}\chi(H_{\ell})(k-\chi(H_{\ell}))$ edges.  These edges are all distinct. Summing these counts, and including the edge that connects to the predecessor, gives the desired total degree.

Moreover, we claim that these edge labels are uniquely determined and hence correct. For Step~\ref{step:findalledges}(a), we consistently label every 6-cycle by Lemma~\ref{lem:length-4-path}.  For Step~\ref{step:findalledges}(b), we are recoloring two parts in the same connected component to different colors. There is a unique fourth vertex where both color changes happen, completing the induced square.  Therefore, we can consistently label these color changes.  Similar logic applies to Step~\ref{step:findalledges}(d), since color changes on parts in two distinct components $H_i$ and $H_{\ell}$ ($\ell\neq i$) are independent. Lastly, for Step~\ref{step:findalledges}(c), if two edges of a $3$-cycle are already labeled (representing a single part changing to two distinct free colors), we can uniquely label the third edge.

Finally, this labeling is consistent across all vertices $\beta$ in $[\alpha]$ because Algorithm~\ref{algo:LLG} relies on a breadth-first search (BFS) in Step~\ref{step:findalledges}. Starting with the canonical labeling at $\alpha$, the BFS ensures that when any vertex $\beta$ is processed, a neighboring vertex (its predecessor) has already been fully and consistently labeled, providing a valid reference.  To be exact, in parts (b) and (d), only $\beta_1$ could be added to $Q$, and we are guaranteed a predecessor $\beta$ as we are currently processing it. In part (a), we add $\beta_1,\beta_2,\beta_3$ to the queue in this order; since $\beta_1$ is adjacent to $\beta$ (which is being processed), each vertex has a predecessor when dequeued. \end{proof}

We have shown that, given a labeling of a link $k$-coloring and all of its hypercube neighbors, we can uniquely construct the entire labeled link graph. Shifting to Hypothesis~\ref{hyp:running}, where $\mathcal{C}_k(G)\cong\mathcal{C}_{\chi'}(G')$, we aim to understand the corresponding proper $\chi'$-colorings of $G'$ and their transformations across this labeled link graph structure. By Lemma~\ref{lem:partition-G'}, every abstract link vertex of $\mathcal{C}_k(G)$ defines an analogous partition on a subgraph of $G'$. However, the precise color changes throughout the entire labeled link graph in the $G'$ context are not yet clear. In particular, it is not yet apparent whether the subgraph of $G'$ undergoing color changes remains fixed across all vertices of the labeled link graph, or whether more than $k$ colors are involved. The following definition introduces $\mathcal{L}'_{[\alpha]}$, an alternative version of the labeled link graph, whose edge labels record the color changes of the underlying proper $\chi'$-colorings in $G'$.

Each labeled edge in $\mathcal{L}_{[\alpha]}$ indicates that there is a hypercube between the two link vertices in $\mathcal{C}$. By Lemma~\ref{lem:partition-G'}, the same hypercube structure implies that the underlying colorings of $G'$ also differ in the color of exactly one part. To keep track of this information, we create a new labeled graph.

\begin{definition}
    Assume Hypothesis~\ref{hyp:running}. Given an abstract link vertex $\alpha$ in $\mathcal{C}$, we define $\mathcal{L}'_{[\alpha]}$ as follows. The graph $\mathcal{L}'_{[\alpha]}$ shares the same vertex set and (unlabeled) edge set with $\mathcal{L}_{[\alpha]}$. For each edge $\beta\xleftrightarrow{c_1 P^i_{j}c_2} \gamma$ in $\mathcal{L}_{[\alpha]}$, if $\beta_{G'}$ and $\gamma_{G'}$ differ in the color on part $Q^{i}$, we label the edge $\beta\xleftrightarrow{c'_1 Q^{i} c'_2} \gamma$ accordingly in $\mathcal{L}'_{[\alpha]}$. 
\end{definition}

Our next goal is to show that the edge labels in $\mathcal{L}'_{[\alpha]}$ have a natural correspondence to the edge labels in $\mathcal{L}_{[\alpha]}$. The identification of $3$-, $4$-, and $6$-cycles plays a crucial role for labeling $\mathcal{L}_{[\alpha]}$ via Algorithm~\ref{algo:LLG}. The following two lemmas assert that our interpretation of these cycles in $\mathcal{L}'_{[\alpha]}$ is the natural one. Lemma~\ref{lem:hexagons-in-linked-graph} shows that the $6$-cycles (which are not part of a 3-cube) in $\mathcal{L}'_{[\alpha]}$ correspond to two fixed parts cycling between three colors. Lemma~\ref{lem:squares-in-linked-graph} states that $4$-cycles in $\mathcal{L}'_{[\alpha]}$ arise from two fixed parts independently changing colors, while $3$-cycles involve one fixed part changing among three colors.

\begin{lemma}\label{lem:hexagons-in-linked-graph} Assume Hypothesis~\ref{hyp:running}. Let $\alpha, \alpha_1, \alpha_2 \in V(\mathcal{L}'_{[\alpha]})$ be vertices with  the edge $\alpha\alpha_1$ labeled $c_1Q_{1}^ic_3$ and the edge $\alpha\alpha_2$ labeled $c_2 Q_{2}^ic_3$. If $\alpha_1$ and $\alpha_2$ are non-adjacent, and $\alpha$ is their \emph{only} common neighbor, then $\mathcal{L}'_{[\alpha]}$ has a unique 6-cycle containing the path $\alpha_1\alpha\alpha_2$, and its edge labels are given by:

\centering
\begin{tikzpicture}[
    vertex/.style={},
    edge_label/.style={
        sloped,   
        midway,   
        auto,    
        font=\small 
    },
    edge_style/.style={
        <->, 
        >=Latex 
    }
  ]
  
  \def\hexagonradius{2.1cm}

  \node[vertex] (alpha)  at (90:\hexagonradius) {$\alpha$};
  \node[vertex] (alpha1) at (30:\hexagonradius) {$\alpha_1$};
  \node[vertex] (beta)   at (-30:\hexagonradius) {$\beta$};
  \node[vertex] (gamma)  at (-90:\hexagonradius) {$\gamma$};
  \node[vertex] (delta)  at (-150:\hexagonradius) {$\delta$};
  \node[vertex] (alpha2) at (150:\hexagonradius) {$\alpha_2$}; 
  \draw[edge_style] (alpha) -- node[edge_label, swap] {$c_1Q^i_{1}c_3$} (alpha1);
  \draw[edge_style] (alpha1) -- node[edge_label, swap] {$c_2Q^i_{2}c_1$} (beta);
  \draw[edge_style] (beta) -- node[edge_label] {$c_2Q^i_{1}c_3$} (gamma);
  \draw[edge_style] (gamma) -- node[edge_label] {$c_3Q^i_{2}c_1$} (delta);
  \draw[edge_style] (delta) -- node[edge_label, swap] {$c_2Q^i_{1}c_1$} (alpha2);
  \draw[edge_style] (alpha2) -- node[edge_label, swap] {$c_3Q^i_{2}c_2$} (alpha);
\end{tikzpicture}
\end{lemma}
\begin{proof}
The vertices $\alpha, \alpha_1, \alpha_2$ in $\mathcal{L}'_{[\alpha]}$ have corresponding elements in $\mathcal{L}_{[\alpha]}$. By Lemma~\ref{lem:length-4-path} applied to $\mathcal{L}_{[\alpha]}$, the path $\alpha_1\alpha\alpha_2$ extends uniquely to an induced 6-cycle $(\alpha, \alpha_1, \beta, \gamma, \delta, \alpha_2, \alpha)$. For every vertex in the 6-cycle in $\mathcal{L}_{[\alpha]}$, the two outgoing incident edges in the 6-cycle represent changes to the same free color. This property continues to hold in the labeling of this 6-cycle in $\mathcal{L}'_{[\alpha]}$, as common free colors can be identified in the coloring graph.   
In particular, a part changing to a free color in $\alpha_1\to \beta$ agrees with the free color $\alpha_1 \to \alpha$, which is $c_1$. Figure~\ref{fig:6-cycle-analysis} displays this partial information as $ \alpha_1\xleftrightarrow{?Wc_1}  \beta$ for some subset $W$ of vertices of $G'$.  We can similarly deduce the partial information about edge  $ \alpha_2\xleftrightarrow{?Zc_2}  \delta$. By the structure of $6$-cycle of cubes in $\mathcal{L}_{[\alpha]}$, the dimensions of the cubes alternate. Thus, $|Z|=|Q_{1}^{i}|$ and $|W|=|Q_{2}^{i}|$. 
\begin{figure}[htbp]
\centering
\begin{tikzpicture}[
    vertex/.style={},
    edge_label/.style={
        sloped,   
        midway,   
        auto,    
        font=\small 
    },
    edge_style/.style={
        <->, 
        >=Latex 
    }
  ]
  
  \def\hexagonradius{2.1cm}

  \node[vertex] (alpha)  at (90:\hexagonradius) {$\alpha$};
  \node[vertex] (alpha1) at (30:\hexagonradius) {$\alpha_1$};
  \node[vertex] (beta)   at (-30:\hexagonradius) {$\beta$};
  \node[vertex] (gamma)  at (-90:\hexagonradius) {$\gamma$};
  \node[vertex] (delta)  at (-150:\hexagonradius) {$\delta$};
  \node[vertex] (alpha2) at (150:\hexagonradius) {$\alpha_2$}; 
  \draw[edge_style] (alpha) -- node[edge_label, swap] {$c_1Q^i_{1}c_3$} (alpha1);
  \draw[edge_style] (alpha1) -- node[edge_label, swap] {?$Wc_1$} (beta);
  \draw[edge_style] (beta) -- node[edge_label] {???} (gamma);
  \draw[edge_style] (gamma) -- node[edge_label] {???} (delta);
  \draw[edge_style] (delta) -- node[edge_label, swap] {$c_2Z?$} (alpha2);
  \draw[edge_style] (alpha2) -- node[edge_label, swap] {$c_3Q^i_{2}c_2$} (alpha);

\end{tikzpicture}
\caption{Partial information}
\label{fig:6-cycle-analysis}
\end{figure}

We claim that $W=Q_{2}^{i}$. Suppose, to the contrary, that $W\neq Q_2^i$. Since $|W|=|Q_{2}^{i}|$, it follows that 
\begin{equation}\label{eq:6-cycle-size-comparison}
    |W\cap Q_{2}^{i}|<|Q_{2}^{i}| \quad \text{ and } \quad  Q_{2}^{i}\setminus W\neq\emptyset.
\end{equation} 

At vertex $\beta$, we have $\beta_{G'}(Q_1^i)=c_3$, $\beta_{G'}(Q_2^i\setminus W)=c_2$, and $\beta_{G'}(W)=c_1$. For the cycle to reach $\alpha_2$, the colors of these three distinct vertex sets ($Q_1^i$, $Q_2^i\setminus W$, and $W$) must all eventually change along the path segment $\beta \to \gamma \to \delta \to \alpha_2$. Specifically, over these three edges $\beta\gamma$, $\gamma\delta$ and $\delta\alpha_2$, there must be exactly one edge where vertices colored $c_1$ (those in $W$) change, one where vertices colored $c_2$ (those in $Q_2^i\setminus W$) change, and one where vertices colored $c_3$ (those in $Q_1^i$) change. Since the edge $\delta\alpha_2$ is already labeled by taking vertices colored $c_2$ to a new color, the vertices in $Q_2^i\setminus W$ must change on this edge back to $c_3$.  So, we may write $Z=(Q_2^i\setminus W)\uplus Z'$ for some set $Z'$ (allowing for the possibility that $Z'=\emptyset$).

Next, observe that the shortest path between abstract link vertices $\alpha$ and $\delta$ in the actual coloring graph $\C$ is $|Q_1^i|+|Q_2^i|$. There cannot be a shorter path, since $\alpha_G$ and $\delta_G$ differ on exactly $|Q_1^i|+|Q_2^i|$ vertices; these edges are correctly labeled as changing two distinct parts of dimensions $|Q_1^i|$ and $|Q_2^i|$ in $\mathcal{L}_{[\alpha]}$. However, with our assumption $W\neq Q_{2}^{i}$, we construct a shorter path from $\alpha$ to $\delta$. Let $\tilde{Q}_2^i=W\cap Q_2^i$.  Starting at $\alpha$, consider the sequence of color changes given by $c_2\tilde{Q}_2^ic_3$ and then $c_3 Z'c_2$ to reach $\delta$. Note that $|\tilde{Q}^{i}_2| = |W\cap Q_{2}^{i}|<|Q_{2}^{i}|$ and $|Z'|<|Z|=|Q_{1}^{i}|$ using $Q_2^{i}\setminus W\neq\emptyset$ from \eqref{eq:6-cycle-size-comparison}. We produced a path from $\alpha$ to $\delta$ of length $|Z'|+|\tilde{Q}_{2}^{i}|<|Q_{1}^{i}|+|Q_{2}^{i}|$, a contradiction. 

\begin{figure}
\centering 
\begin{tikzpicture}[remember picture, 
    vertex/.style={},
    edge_label/.style={
        sloped,    
        midway,    
        auto,      
        font=\small 
    },
    edge_style/.style={
        <->, 
        >=Latex 
    }
  ]
 
  \def\hexagonradius{2.1cm}

  \node[vertex] (alpha)  at (90:\hexagonradius) {$\alpha$};
  \node[vertex] (alpha1) at (30:\hexagonradius) {$\alpha_1$};
  \node[vertex] (beta)   at (-30:\hexagonradius) {$\beta$};
  \node[vertex] (gamma)  at (-90:\hexagonradius) {$\gamma$};
  \node[vertex] (delta)  at (-150:\hexagonradius) {$\delta$};
  \node[vertex] (alpha2) at (150:\hexagonradius) {$\alpha_2$}; 
  \draw[edge_style] (alpha) -- node[edge_label, swap] {$c_1Q^i_{1}c_3$} (alpha1);
  \draw[edge_style] (alpha1) -- node[edge_label, swap] {$c_2 Q_2^{i} c_1$} (beta);
  \draw[edge_style] (beta) -- node[edge_label] {$c_2$??} (gamma);
  \draw[edge_style] (gamma) -- node[edge_label] {??$c_1$} (delta);
  \draw[edge_style] (delta) -- node[edge_label, swap] {$c_2Q_{1}^{i} c_1$} (alpha2);
  \draw[edge_style] (alpha2) -- node[edge_label, swap] {$c_3 Q^i_{2}c_2$} (alpha);
  \coordinate (pic1right) at (current bounding box.east);

\end{tikzpicture}
\qquad \qquad
\begin{tikzpicture}[remember picture, 
    vertex/.style={},
    edge_label/.style={
        sloped,    
        midway,    
        auto,      
        font=\small 
    },
    edge_style/.style={
        <->, 
        >=Latex 
    }
  ]
 
  \def\hexagonradius{2.1cm}

  \node[vertex] (alpha)  at (90:\hexagonradius) {$\alpha$};
  \node[vertex] (alpha1) at (30:\hexagonradius) {$\alpha_1$};
  \node[vertex] (beta)   at (-30:\hexagonradius) {$\beta$};
  \node[vertex] (gamma)  at (-90:\hexagonradius) {$\gamma$};
  \node[vertex] (delta)  at (-150:\hexagonradius) {$\delta$};
  \node[vertex] (alpha2) at (150:\hexagonradius) {$\alpha_2$}; 
  \draw[edge_style] (alpha) -- node[edge_label, swap] {$c_1Q^i_{1}c_3$} (alpha1);
  \draw[edge_style] (alpha1) -- node[edge_label, swap] {$c_2 Q_2^{i} c_1$} (beta);
  \draw[edge_style] (beta) -- node[edge_label] {$c_2 Q^{i}_1 c_3$} (gamma);
  \draw[edge_style] (gamma) -- node[edge_label] {$c_3 Q^{i}_2 c_1$} (delta);
  \draw[edge_style] (delta) -- node[edge_label, swap] {$c_2Q_{1}^{i} c_1$} (alpha2);
  \draw[edge_style] (alpha2) -- node[edge_label, swap] {$c_3 Q^i_{2}c_2$} (alpha);

  \coordinate (pic2left) at (current bounding box.west);

\end{tikzpicture}

\begin{tikzpicture}[remember picture, overlay]
  \path (pic1right) -- (pic2left) node[midway] (arrowmidpoint) {};
  \draw[->, >=Latex, thick, black]
    ($(pic1right)+(0.1cm,0)$) -- node[above, font=\small, black] {} ($(pic2left)+(-0.1cm,0)$); 
\end{tikzpicture}
\caption{Identifying the final two edges in the $6$-cycle}\label{fig:6-cycle-analysis-final-edges}
\end{figure}

Therefore, $W=Q_2^i$, and the edge $\alpha_1\beta$ is labeled $c_2Q_2^ic_1$. By identical reasoning, we conclude $Z=Q_{1}^{i}$ and the edge $\alpha_2\delta$ is labeled $c_1Q_1^ic_2$. We only have two remaining edges to identify as in Figure~\ref{fig:6-cycle-analysis-final-edges} (left). Both $Q_{1}^{i}$ and $Q_{2}^{i}$ must return to their original color as we go around the cycle. As the $\beta\gamma$ and $\gamma\delta$ edges already have one of the colors fixed in their edge label, there is only one possible way to label these two edges $\delta \gamma$ and $\gamma\beta$; see Figure~\ref{fig:6-cycle-analysis-final-edges} (right). \end{proof}
 
\begin{lemma}\label{lem:squares-in-linked-graph} Assume Hypothesis~\ref{hyp:running}.  Let $\alpha, \alpha_1, \alpha_2 \in V(\mathcal{L}'_{[\alpha]})$ be distinct vertices such that $\alpha$ is adjacent to both $\alpha_1$ and $\alpha_2$. 
\begin{enumerate}
    \item If $\alpha\alpha_1$ has label $c_1Q_{1}^ic_3$ and $\alpha\alpha_2$ has label $c_2 Q_{2}^ic_4$, where $c_1,c_2,c_3$, and $c_4$ are all distinct, and $Q_{1}^i$ and $Q_2^{i}$ are disjoint sets, then there is a unique induced $4$-cycle in $\mathcal{L}'_{[\alpha]}$ that contains these two edges and its labels are given by:
     \[\alpha \xleftrightarrow{c_1Q^i_{1}c_3}  \alpha_1\xleftrightarrow{c_2Q^i_{2}c_4} \beta\xleftrightarrow{c_3Q^i_{1}c_1}\alpha_2\xleftrightarrow{c_4Q^i_{2}c_2}  \alpha.
\]
    \item If $\alpha\alpha_1$ has label $c_1Q_{1}^ic_3$ and $\alpha\alpha_2$ has label $c_1 Q_{1}^ic_4$, where $c_1,c_3$, and $c_4$ are all distinct, then there is a unique $3$-cycle in $\mathcal{L}'_{[\alpha]}$ that contains these two edges and its labels are given by: 
    \[\alpha \xleftrightarrow{c_1Q^i_{1}c_3}  \alpha_1\xleftrightarrow{c_3Q^i_{1}c_4} \alpha_2\xleftrightarrow{c_4Q^i_{1}c_1}  \alpha.
\]
    \item If $\alpha\alpha_1$ has label $c_1Q_{1}^ic_3$ and $\alpha\alpha_2$ has label $c_2 Q_{2}^jc_4$ for $i\neq j$, where necessarily $c_1\neq c_3$ and $c_2\neq c_4$, and $Q_1^i$ and $Q_2^j$ are disjoint sets, then there is a unique induced $4$-cycle in $\mathcal{L}'_{[\alpha]}$ that contains these two edges and its labels are given by: \[\alpha \xleftrightarrow{c_1Q^i_{1}c_3}  \alpha_1\xleftrightarrow{c_2Q^j_{2}c_4} \beta\xleftrightarrow{c_3Q^i_{1}c_1}\alpha_2\xleftrightarrow{c_4Q^j_{2}c_2}  \alpha.
\]
\end{enumerate}
\end{lemma}
\begin{proof}
We prove each statement independently below.
\begin{enumerate}
    \item We are given $\alpha_{G'}(Q^i_1)=c_1, \alpha_{G'}(Q^i_2)=c_2$, with edges $\alpha \xleftrightarrow{c_1Q^i_1c_3} \alpha_1$ and $\alpha \xleftrightarrow{c_2Q^i_2c_4} \alpha_2$. Thus, $(\alpha_1)_{G'}(Q^i_1)=c_3, (\alpha_1)_{G'}(Q^i_2)=c_2$ and $(\alpha_2)_{G'}(Q^i_1)=c_1, (\alpha_2)_{G'}(Q^i_2)=c_4$.
    Since $Q^i_1, Q^i_2$ are disjoint parts and $c_1,c_2,c_3,c_4$ are distinct, the recoloring operations $Q^i_1\colon c_1 \to c_3$ and $Q^i_2\colon c_2 \to c_4$ are independent. Let $\beta$ be the coloring where $\beta_{G'}(Q^i_1)=c_3$ and $\beta_{G'}(Q^i_2)=c_4$ (other parts as in $\alpha$). This $\beta$ is proper as $c_3 \neq c_4$. The edges $\alpha_1 \xleftrightarrow{c_2Q^i_2c_4} \beta$ and $\beta \xleftrightarrow{c_3Q^i_1c_1} \alpha_2$ complete the stated unique induced $4$-cycle $\alpha \xleftrightarrow{c_1Q^i_1c_3} \alpha_1 \xleftrightarrow{c_2Q^i_2c_4} \beta \xleftrightarrow{c_3Q^i_1c_1} \alpha_2 \xleftrightarrow{c_4Q^i_2c_2} \alpha$.
    \item  We are given $\alpha_{G'}(Q^i_1)=c_1$, with edges $\alpha \xleftrightarrow{c_1Q^i_1c_3} \alpha_1$ and $\alpha \xleftrightarrow{c_1Q^i_1c_4} \alpha_2$. Thus, $(\alpha_1)_{G'}(Q^i_1)=c_3$ and $(\alpha_2)_{G'}(Q^i_1)=c_4$, while all other parts are colored as in $\alpha$. Since $c_3 \neq c_4$, $\alpha_1$ and $\alpha_2$ differ only on $Q^i_1$ and are thus adjacent via the edge $\alpha_1 \xleftrightarrow{c_3Q^i_1c_4} \alpha_2$. Together with the edge $\alpha_2 \xleftrightarrow{c_4Q^i_1c_1} \alpha$ (returning $Q^i_1$ to its color in $\alpha$), this forms the unique induced $3$-cycle $\alpha \xleftrightarrow{c_1Q^i_1c_3} \alpha_1 \xleftrightarrow{c_3Q^i_1c_4} \alpha_2 \xleftrightarrow{c_4Q^i_1c_1} \alpha$.
    \item We are given $\alpha_{G'}(Q^i_1)=c_1, \alpha_{G'}(Q^j_2)=c_2$, with initial edges $\alpha \xleftrightarrow{c_1Q^i_1c_3} \alpha_1$ and $\alpha \xleftrightarrow{c_2Q^j_2c_4} \alpha_2$. Here $Q^i_1$ and $Q^j_2$ are disjoint vertex sets, $c_1 \neq c_3$, and $c_2 \neq c_4$. The recolorings $Q^i_1\colon c_1 \to c_3$ and $Q^j_2\colon c_2 \to c_4$ are independent. If $c_3 \neq c_4$, the reasoning from part (1) applies directly: the independent operations with distinct target colors define the unique induced 4-cycle and its labeling as stated.
    
   Consider the case $c_3 = c_4 = c$. We have $\alpha_{G'} \equiv (Q^i_1\colon c_1, Q^j_2\colon c_2)$, $(\alpha_1)_{G'} \equiv (Q^i_1\colon c, Q^j_2\colon c_2)$, and $(\alpha_2)_{G'} \equiv (Q^i_1\colon c_1, Q^j_2\colon c)$. These vertices complete to a 4-cycle in $\mathcal{L}'_{[\alpha]}$ because the matching $4$-cycle exists in $\mathcal{L}_{[\alpha]}$. Let $\beta$ be the fourth vertex in this 4-cycle $(\alpha, \alpha_1, \beta, \alpha_2,\alpha)$.
    The $\alpha$-avoiding path from $\alpha_1$ to $\alpha_2$ is $\alpha_1 \to \beta \to \alpha_2$. To change from $(\alpha_1)_{G'} \equiv (Q^i_1\colon c, Q^j_2\colon c_2)$ to $(\alpha_2)_{G'} \equiv (Q^i_1\colon c_1, Q^j_2\colon c)$, part $Q^i_1$ must change from $c$ to $c_1$, and part $Q^j_2$ must change from $c_2$ to $c$.
    These two changes must occur over the two edges $(\alpha_1, \beta)$ and $(\beta, \alpha_2)$. There are two ways to assign these changes to the edges:
    \begin{enumerate}
        \item Path $T_1$: $\alpha_1 \xleftrightarrow{cQ^i_1c_1} \beta \xleftrightarrow{c_2Q^j_2c} \alpha_2$.
        Here, $\beta$ would have $Q^i_1$ colored $c_1$ and $Q^j_2$ colored $c_2$. Thus, $\beta_{G'} \equiv (Q^i_1\colon c_1, Q^j_2\colon c_2)$, which is $\alpha$. This is a contradiction, because $\beta$ is distinct from $\alpha$.
        \item Path $T_2$: $\alpha_1 \xleftrightarrow{c_2Q^j_2c} \beta \xleftrightarrow{cQ^i_1c_1} \alpha_2$.
        Here, from $(\alpha_1)_{G'} \equiv (Q^i_1\colon c, Q^j_2\colon c_2)$, recoloring $Q^j_2\colon c_2 \to c$ yields $\beta_{G'} \equiv (Q^i_1\colon c, Q^j_2\colon c)$. This path $T_2$ avoids $\alpha$, so this is the correct labeling.
    \end{enumerate}
    Therefore, $T_2$ provides the unique labeling for the $\alpha$-avoiding path segment $\alpha_1 \to \beta \to \alpha_2$. Reinstating $c$ as $c_3$ (for $Q^i_1$'s target) and $c_4$ (for $Q^j_2$'s target, which is also $c_3$ in this subcase), the labels become $\alpha_1 \xleftrightarrow{c_2Q^j_2c_4} \beta$ and $\beta \xleftrightarrow{c_3Q^i_1c_1} \alpha_2$.
    This completes the 4-cycle as stated in the lemma: $\alpha \xleftrightarrow{c_1Q^i_1c_3} \alpha_1 \xleftrightarrow{c_2Q^j_2c_4} \beta \xleftrightarrow{c_3Q^i_1c_1} \alpha_2 \xleftrightarrow{c_4Q^j_2c_2} \alpha$. In retrospect, the properness of $\beta$, ensured by the existence of this 4-cycle structure, implies that if $c_3=c_4$, no edge exists between $Q^i_1$ and $Q^j_2$. \qedhere
\end{enumerate}
\end{proof}

Building upon Lemmas~\ref{lem:hexagons-in-linked-graph} and \ref{lem:squares-in-linked-graph}, we find a natural correspondence between $\mathcal{L}_{[\alpha]}$ and $\mathcal{L}'_{[\alpha]}$. To formalize this correspondence, the following lemma introduces functions describing how a proper $k$-coloring $\alpha_{G}$ is lifted to a $\chi'$-coloring $\alpha_{G'}$ on a subgraph of $G'$ isomorphic to $G$. Recall that $H_1,\dots, H_d$ are the connected components of $G$, identified and labeled by Algorithm~\ref{algo:link-vertex} run on $\mathcal{C}_k(G)$ and $\alpha$. For each $H_i$, we define two mappings: $\psi_i$ and $\phi_i$. The function $\psi_i$ translates colors from the palette $[k]$ used in the $G$-interpretation of $\C$ to distinct colors in a specific palette $\mathcal{P}_i\subseteq [\chi']$ for the $G'$-interpretation. The function $\phi_i$ maps each vertex in $H_i$ to a vertex in $G'$. We encountered these two functions in Lemma~\ref{lem:partition-G'} at a single abstract link vertex. The following key lemma establishes that $\psi_i$ and $\phi_i$ can be chosen consistently for every link vertex within the same equivalence class. 

\begin{lemma}\label{lem:gluing-lifting-data}
Assume Hypothesis~\ref{hyp:running}.  For each $i\in[d]$ there exists a subset $\mathcal P_i\subseteq [\chi']$ with $\abs{\mathcal P_i}=k$ and two functions $\psi^{[\alpha]}_{i}$  and $\phi^{[\alpha]}_{i}$ such that:
\begin{enumerate}
    \item $\psi^{[\alpha]}_{i}\colon [k]\to \mathcal{P}_i$ is a bijection.
    \item $\phi^{[\alpha]}_{i}\colon H_i\to H'_i$ is a graph isomorphism. 
    \item\label{item:LLG:comp} For each labeled edge $\beta\xleftrightarrow{c_1 P^i_{j}c_2} \gamma$ in $\mathcal{L}_{[\alpha]}$, the corresponding labeled edge in $\mathcal{L}'_{[\alpha]}$ is given by $\beta\xleftrightarrow{\psi_i^{[\alpha]}(c_1) \phi_i^{[\alpha]}(P^i_{j})\psi_i^{[\alpha]}(c_2)} \gamma$.
\end{enumerate}
We use $(\psi^{[\alpha]}, \phi^{[\alpha]})$ as a shorthand for the collection of pairs $(\psi_i^{[\alpha]}, \phi_i^{[\alpha]})$ for $1\leq i\leq d$.
\end{lemma}

\begin{proof}
By Lemma~\ref{lem:subgraphs}, there exist vertex-disjoint subgraphs $H'_1, \dots, H'_d$ in $G'$ such that $H'_i\cong H_i$ for each $i\in [d]$. The map $\phi_i^{[\alpha]}\colon H_i\to H'_{i}$ can be constructed by comparing edges incident to $\alpha$ in $\mathcal{C}_{k}(G)\cong \mathcal{C}_{\chi'}(G')$, as in the proof of Lemma~\ref{lem:subgraphs}. By Lemma~\ref{lem:partition-G'}, for each $i\in[d]$ there exists an injection $\psi_i^{[\alpha]}\colon [k]\to [\chi']$ such that $\alpha_{G'}(\phi_i^{[\alpha]}(v))=\psi_i^{[\alpha]}\left(\alpha_G(v)\right)$ for each vertex $v\in V(H_i)$. As of yet, $\psi_i^{[\alpha]}$ is not uniquely defined for any of the free colors, those colors that do not appear in $\alpha_G$.  
By inspection of incident edges to $\alpha$ in $\mathcal{C}_{k}(G)\cong \mathcal{C}_{\chi'}(G')$, we can define $\psi_i^{[\alpha]}$ to be compatible for free colors. More precisely, for each $i\in [d]$, and each free color $c$ at $\alpha_G$ within $H_i$, consider $\beta_{G}$ obtained from $\alpha_{G}$ by recoloring all the vertices in a single independent part $P_j^{i}$ from $\alpha_{G}(P_{j}^{i})$ to $c$. We define $\psi_{i}^{[\alpha]}(c)$ to be $\beta_{G'}(\phi_i^{[\alpha]}(P_{j}^{i}))$. This ensures the relation $\beta_{G'}(\phi_i^{[\alpha]}(v))=\psi_i^{[\alpha]}\left(\beta_G(v)\right)$ for every $\beta$ that is a neighbor of $\alpha$ in $\mathcal{L}_{[\alpha]}$. After defining $\mathcal{P}_{i} = \psi_i^{[\alpha]}([k])$, the map $\psi_{i}^{[\alpha]}\colon [k]\to \mathcal{P}_i$ is a bijection. 

By construction, the desired condition~\eqref{item:LLG:comp} holds for all incident edges to $\alpha$ in $\mathcal{L}_{[\alpha]}$, meaning that:
$$
\alpha\xleftrightarrow{c_1 P^i_{j}c_2} \beta \text{ in } \mathcal{L}_{[\alpha]}  \quad\quad \Longrightarrow \quad\quad \alpha\xleftrightarrow{\psi_i^{[\alpha]}(c_1) \phi_i^{[\alpha]}(P^i_{j})\psi_i^{[\alpha]}(c_2)} \beta \text{ in } \mathcal{L}'_{[\alpha]}.
$$
Choose some neighbor $\beta$ in $\mathcal{L}'_{[\alpha]}$.  Using the same iteration process as Algorithm~\ref{algo:LLG} (Step~\ref{step:findalledges}), label all incident edges to $\beta$ in $\mathcal{L}'_{[\alpha]}$. These edges can be found through a cycle of appropriate length. Then Lemmas~\ref{lem:hexagons-in-linked-graph} and \ref{lem:squares-in-linked-graph} propagate the desired correspondence \eqref{item:LLG:comp} from labeled edges incident to $\alpha$ to those incident to $\beta$.
By continuing this process in a breadth-first ordering as in Algorithm~\ref{algo:LLG} (Step~\ref{step:findalledges}), which is correct by Lemma~\ref{lem:unique-LLG-labeling}, all edges in $\mathcal{L}'_{[\alpha]}$ will be labeled using the correspondence given by the pairs $(\psi_i^{[\alpha]}, \phi_i^{[\alpha]})$.
\end{proof}

Examining the proof of Lemma~\ref{lem:gluing-lifting-data}, we observe that $\psi^{[\alpha]}$ is constructed precisely in a way to satisfy $\beta_{G'} \circ \phi_i^{[\alpha]} = \psi_{i}^{[\alpha]} \circ \beta_{G}$ for each $\beta_{G}\in [\alpha_{G}]$. This \emph{compatibility relation} turns out to be important in later arguments, primarily because it lends itself well to computation. 

\begin{corollary}\label{cor:gluing-relations} Assume Hypothesis~\ref{hyp:running}. Consider the functions $(\psi_i^{[\alpha]}, \phi_i^{[\alpha]})$ from Lemma~\ref{lem:gluing-lifting-data}. For each $\beta$ in $[\alpha]$, the relation  $\beta_{G'} \circ \phi_i^{[\alpha]} = \psi_{i}^{[\alpha]} \circ \beta_{G}$ holds for each $1\leq i\leq d$.
\end{corollary}

The compatibility condition in Corollary~\ref{cor:gluing-relations} can be visualized as a commutative diagram.

\begin{center}
\begin{tikzcd}[column sep=3em]
V(H_i)  \arrow[d,swap, "\phi_i^{[\alpha]}"] \arrow[r, "\beta_{G}"] & {[k]} \arrow[d, "\psi_i^{[\alpha]}"] \\
V(H'_i) \arrow[r, swap, "\beta_{G'}"] & {\mathcal{P}_i}
\end{tikzcd}
\end{center}

Lemma~\ref{lem:gluing-lifting-data} establishes the coloring of the vertices in $V(H'_1)\uplus V(H'_2)\uplus \dots \uplus V(H'_d)$ throughout the equivalence class $[\alpha]$ when viewed inside $G'$. Let $L\subseteq V(G')$ be the complement of $V(H'_1)\uplus V(H'_2)\uplus \dots \uplus V(H'_d)$. The next corollary asserts that equivalent link colorings agree pointwise on $L$. We refer to $L$ as the \emph{locked set} associated with the equivalence class $[\alpha]$. For each vertex $\beta\in[\alpha]$, the coloring of $L$ in $\beta_{G'}$ is determined by an orbit representative $\alpha_{G'}$. This fixed assignment of colors to $L$ must be proper with respect to the way $H_1',\dots,H'_d$ are colored (as derived from $\beta_{G}$), since $\beta$ exists as a vertex in the coloring graph.

\begin{corollary}\label{cor:equivalent-links-agree-on-L} Assume Hypothesis~\ref{hyp:running}.
Suppose $\beta\in [\alpha]$. Then $\beta_{G'}(w)=\alpha_{G'}(w)$ for every $w\in L$.
\end{corollary}
\begin{proof}
From Lemma~\ref{lem:gluing-lifting-data},  $\alpha$ and $\beta$ share the same functions $(\psi^{[\alpha]},\phi^{[\alpha]})$, with $\beta_{G'}(\phi^{[\alpha]}_{i}(v))=\psi_i^{[\alpha]}(\beta_{G}(v))$ for every $v\in V(H_i)$ and $i\in [d]$. Since $\alpha$ and $\beta$ are connected in $\mathcal{L}_{[\alpha]}$, it suffices to reduce (by induction on the length of path between them) to the case when $\alpha$ and $\beta$ are adjacent in $\mathcal{L}_{[\alpha]}$. In this case, we have an edge of the form $\alpha\xleftrightarrow{c_1 P^i_{j}c_2}  \beta$ where $P^{i}_{j}\subset V(H_i)$ is the unique part on which $\alpha_{G}$ and $\beta_{G}$ differ. Specifically, $\alpha_{G}(P^{i}_{j}) = c_1$, $\beta_{G}(P^{i}_{j}) = c_2$, and $\alpha_{G}(v)=\beta_{G}(v)$ for all $v\in V(G)\setminus P^i_{j}$. Using Corollary~\ref{cor:gluing-relations}, the colorings $\alpha_{G'}$ and $\beta_{G'}$ satisfy the compatibility relations with the same pair $(\psi^{[\alpha]}, \phi^{[\alpha]})$:
$$
\alpha_{G'} \circ \phi_i^{[\alpha]} = \psi_{i}^{[\alpha]} \circ \alpha_{G} \quad \text{ and } \quad \beta_{G'} \circ \phi_i^{[\alpha]} = \psi_{i}^{[\alpha]} \circ \beta_{G}
$$
Applying these identities on the input $V(H_i)$, we obtain that $\alpha_{G'}$ and $\beta_{G'}$ differ on a single part $\phi_i^{[\alpha]}(P^{i}_{j})$. By counting the number of edges between the vertices $\alpha$ and $\beta$ in $\mathcal{C}$, we see that $\alpha_{G'}$ and $\beta_{G'}$ must agree on all vertices in $V(G')\setminus \phi_{i}^{[\alpha]}(P^{i}_{j})$. Since $L$ is disjoint from $V(H'_i)$, we get $\alpha_{G'}(w)=\beta_{G'}(w)$ for every $w\in L$.
\end{proof}

An example of the pair $(\psi^{[\alpha]},\phi^{[\alpha]})$ and the associated locked set $L$ is illustrated in the top row of Figure~\ref{fig:surgery}. We leverage the previous results to deduce more information about the $\chi'$-colorings of $G'$ at the abstract link vertices. From Lemma~\ref{lem:subgraphs}, we know that $G'$ contains $G$ as a subgraph where each connected component $H_i$ of $G$ appears as an induced subgraph $H_i'$. However, it is certainly possible for there to be an edge connecting $H'_{i}$ and $H'_{j}$ for $i\neq j$. The following corollary imposes strong restrictions in this case. It concludes that if $H_i'$ and $H_j'$ are connected by an edge, their respective color palettes $\mathcal{P}_i$ and $\mathcal{P}_{j}$ must be disjoint. Furthermore, no vertex in the locked set $L$ adjacent to a vertex in $H_i'$ can have a color from $\mathcal{P}_i$. Finally, the locked set $L$ must use every color in the $[\chi']$ palette. Figure~\ref{fig:surgery} depicts these properties. In the statement below, $N(v)$ denotes the set of neighbors of a vertex $v$ in $G'$.

\begin{corollary}\label{cor:locked-vertices}
   Assume Hypothesis~\ref{hyp:running}. Let $\alpha$ be an abstract link vertex.  Then
    \begin{enumerate}
        \item\label{item:disjoint-palettes} for every $1\le i<j\le d$, if there exist $u\in H'_i,v\in H'_j$ with $uv\in E(G')$, then $\mathcal P_i\cap \mathcal P_j=\emptyset$; 
        \item\label{item:locked-palette} for every $i\in[d], v\in H'_i$, $\alpha_{G'}(N(v)\cap L)\subseteq [\chi']\backslash \mathcal P_i$;
        \item\label{item:every-color-locked} $\alpha_{G'}(L)=[\chi']$. 
        \end{enumerate}
\end{corollary}

\begin{proof}
By Lemma~\ref{lem:gluing-lifting-data}, all $\alpha, \beta$ with $[\beta]=[\alpha]$ share the same pair $(\psi^{[\alpha]},\phi^{[\alpha]})$ defining a fixed partition $(H'_1,\dots,H'_d, L)$ of $G'$ where $H_i'\cong H_i$. 
\begin{enumerate}
    \item Let $uv\in E(G')$ with $u\in V(H'_i)$ and $v\in V(H'_j)$ for $i\neq j$. Suppose, to the contrary, that $\mathcal{P}_i\cap\mathcal{P}_j\neq \emptyset$. Take any color $c\in\mathcal{P}_i\cap\mathcal{P}_j$.  Within $[\alpha]$ we may independently permute the colors (within $k$ available colors) inside each $H_i$ and $H_j$. Consequently, we can find a link $k$-coloring $\beta_{G}$ such that the associated $\chi'$-coloring $\beta_{G'}$ simultaneously satisfies: $\beta_{G'}(u)=c$ and $\beta_{G'}(v)=c$. This contradicts the fact that $\beta_{G'}$ is a proper coloring, and $uv\in E(G')$.
    \item We show that no color $c\in \mathcal{P}_i$ appears on $N(v)\cap L$. Otherwise, choose $w\in N(v)\cap L$ with $\alpha_{G'}(w)=c$. Find $\beta_G\in [\alpha_G]$ such that $\beta_{G'}(v)=c$. This is possible because $v\in H'_i$. Since $\beta_{G'}|_{L} = \alpha_{G'}|_{L}$ by Corollary~\ref{cor:equivalent-links-agree-on-L}, we have $\beta_{G'}(w)=\alpha_{G'}(w)=c$. The resulting equality $\beta_{G'}(v)=\beta_{G'}(w)$ leads to a contradiction as $\beta_{G'}$ is a proper coloring and $vw\in E(G')$. 
    \item Suppose $\tilde{c}\in [\chi']$ is not present in $\alpha_{G'}(L)$. Find $\beta\in [\alpha]$ such that $\tilde{c}\notin \beta_{G'}(V(H'_i))$ for each $i\in [d]$. This is possible, as $\beta_{G'}(V(H'_i))=\psi_{i}^{[\alpha]}(\beta_{G}(V(H_i)))$ and $G$ has surplus colors. Now, $\tilde{c}\notin \alpha_{G'}(L) = \beta_{G'}(L)$. As $\tilde{c}\notin \beta_{G'}(V(G'))$, it follows that $\beta_{G'}$ is a proper $(\chi'-1)$-coloring of $G'$, a contradiction. \qedhere
\end{enumerate}
\end{proof}

As a final result in this subsection, we use the accumulated information regarding the proper coloring $\alpha_{G'}$ to apply the pair $(\psi^{[\alpha]},\phi^{[\alpha]})$ to any abstract link vertex, not just those in $[\alpha]$.  In particular, we can convert {\it any} link $k$-coloring of $G$ to a proper $\chi'$-coloring of $G'$ by applying the pair $(\psi^{[\alpha]},\phi^{[\alpha]})$, while retaining the colors of the locked set $L$ as in $\alpha_{G'}$.  Figure~\ref{fig:surgery} shows an example of applying $(\psi^{[\alpha]},\phi^{[\alpha]})$ to an abstract link vertex not in $[\alpha]$ and still obtaining a proper coloring that is itself an abstract link vertex.  This technique produces \emph{many} abstract link vertices from an equivalence class $[\alpha]$ of a \emph{single} abstract link vertex. This process will demonstrate that the set of abstract link vertices of $\mathcal{C}_k(G)$ cannot be in bijection with the set of abstract link vertices of $\mathcal{C}_{\chi'}(G')$, ultimately showing that Hypothesis~\ref{hyp:running} is untenable.

\begin{lemma}
\label{lem:all_extend}
Assume Hypothesis~\ref{hyp:running}. Let $\alpha,\beta$ be two abstract link vertices in $\mathcal{C}$ such that $\beta\notin[\alpha]$. Let $(\psi^{[\alpha]},\phi^{[\alpha]})$ be a pair for $[\alpha]$ from Lemma~\ref{lem:gluing-lifting-data}. Define $\widetilde{\beta}_{G'}\colon V(G')\to [\chi']$ by specifying $\widetilde{\beta}_{G'}(\phi^{[\alpha]}(v))=\psi^{[\alpha]}(\beta_G(v))$ for $v\in V(G)$ and $\widetilde{\beta}_{G'}(w)=\alpha_{G'}(w)$ for $w\in L$. Then $\widetilde{\beta}_{G'}$ is a proper coloring of $G'$.  Moreover, the corresponding vertex $\widetilde{\beta}\in\C$ (which may or may not be $\beta$) is an abstract link vertex satisfying  $\widetilde{\beta}\notin[\alpha]$.
\end{lemma}

\begin{proof}

We first establish that $\widetilde{\beta}_{G'}$ is a proper $\chi'$-coloring. The domain of $\widetilde{\beta}_{G'}$ is 
$$
V(G') = V(H'_1) \uplus \dots \uplus V(H'_{d})\uplus L.
$$
We carefully check properness of $\widetilde{\beta}_{G'}$ on the set $E(G')$ by separating the edges into four types.
\begin{itemize}
    \item \emph{Edges within $H'_{i}$:} For each $u\in H'_i$, write $u=\phi_i^{[\alpha]}(v)$ for some $v\in V(H_i)$. Then $\widetilde{\beta}_{G'}(u)=\psi_i^{[\alpha]}(\beta_G(v))$. Since $\beta_G$ is proper on $H_i$, $\phi_i^{[\alpha]}\colon H_i\to H'_i$ is an isomorphism, and $\psi_i^{[\alpha]}$ is injective, it follows that $\widetilde{\beta}_{G'}$ preserves properness on $E(H'_i)$.
    \item \emph{Edges between $H'_{i}$ and $H'_{j}$ for $i\neq j$:} Suppose $uv\in E(G')$ with $u\in H'_i$ and $v\in H'_j$. Corollary~\ref{cor:locked-vertices}\eqref{item:disjoint-palettes} implies $\mathcal P_{i}\cap \mathcal P_{j}=\emptyset$. Since $\widetilde{\beta}_{G'}(u)\in \mathcal P_{i}$ and $\widetilde{\beta}_{G'}(v)\in \mathcal P_{j}$, we deduce that $\widetilde{\beta}_{G'}(u)\neq \widetilde{\beta}_{G'}(v)$. 
    \item \emph{Edges between $H'_{i}$ and $L$:} Suppose $uw\in E(G')$ with $u\in H'_{i}$ and $w\in L$. By Corollary~\ref{cor:locked-vertices}\eqref{item:locked-palette}, the image set $\alpha_{G'}(N(u)\cap L)=\widetilde{\beta}_{G'}(N(u)\cap L)$ is disjoint from $\mathcal{P}_{i}$. Since $w\in N(u)\cap L$ and $\widetilde{\beta}_{G'}(u)\in \mathcal{P}_{i}$, it follows that $\widetilde{\beta}_{G'}(w)\neq \widetilde{\beta}_{G'}(u)$.
    \item \emph{Edges within $L$:} The coloring $\widetilde{\beta}_{G'}|_{L}$ is identical to $\alpha_{G'}|_{L}$. Since $\alpha_{G'}$ is a proper coloring, $\widetilde{\beta}_{G'}$ is also proper on the subgraph induced by $L$.
\end{itemize}
Next, we argue that $\widetilde{\beta}\in \mathcal{C}$ (the vertex corresponding to the $\chi'$-coloring $\widetilde{\beta}_{G'}$) is an abstract link vertex. When verifying that $\beta$ is an abstract link vertex, Algorithm~\ref{algo:link-vertex} checks the clique-neighborhood around the vertex $\beta$ corresponding to the coloring $\beta_G$. The same check applies to $\widetilde{\beta}$ because $\phi^{[\alpha]}$ is a graph isomorphism, $\psi^{[\alpha]}$ is a bijection, and $\widetilde{\beta}_{G'}\circ \phi^{[\alpha]}$ and $\psi^{[\alpha]}\circ \beta_{G}$ agree on $V(G)$. Therefore, $\widetilde{\beta}$ is also an abstract link vertex. It remains to show that $\widetilde{\beta}\notin[\alpha]$. The relation $\widetilde{\beta}_{G'}(\phi^{[\alpha]}(v))=\psi^{[\alpha]}(\beta_G(v))$ for all $v\in V(G)$ implies that the corresponding $k$-coloring $\widetilde{\beta}_{G}$ coincides with $\beta_{G}$. Since $\beta\notin [\alpha]$, it follows that $\widetilde{\beta}\notin [\alpha]$, as desired. \end{proof}

\begin{figure}[ht]
\begin{center}
\begin{tikzpicture}[scale=1, every node/.style={draw=none, fill=none, minimum size=8mm, font=\small}]

\node[font=\LARGE] (Ag) at (-10,0) {$\alpha_G$};
\node[font=\LARGE] (Agp) at (-6.5,0) {$\alpha_{G'}$};
\draw[->, thick] (Ag) -- (Agp) node[midway, above,font=\Large] {$(\psi^{[\alpha]},\phi^{[\alpha]})$};

\node at (-11.5,-0.75) {$H_1$};
\node[fill=red!30, circle, draw] (w1) at (-11.5,1.25) {1};
\node[fill=green!30, circle, draw] (w2) at (-11.5,0) {2};
\node[fill=blue!30, circle, draw] (w3) at (-10.5,-1.15) {3};
\node[fill=red!30, circle, draw] (w4) at (-12.5,-1.15) {1};

\draw (w1) -- (w2) -- (w3) -- (w4) -- (w2);
\node[anchor=east] at (w1.west) {$u_1$};
\node[anchor=east] at (w2.west) {$u_2$};
\node[anchor=west] at (w3.east) {$u_3$};
\node[anchor=east] at (w4.west) {$u_4$};

\node at (-5,-0.75) {$H_1'$};
\node at (0,0) {$L$};

\node[fill=red!30, circle, draw] (v1) at (-5,1.25) {1};
\node[fill=yellow!30, circle, draw] (v2) at (-5,0) {5};
\node[fill=blue!30, circle, draw] (v3) at (-4,-1.15) {3};
\node[fill=red!30, circle, draw] (v4) at (-6,-1.15) {1};

\draw (v1) -- (v2) -- (v3) -- (v4) -- (v2);

\node[fill=red!30, circle, draw] (u1) at ({90+72*(1-1)}:1.5cm) {1};
\node[fill=yellow!30, circle, draw] (u5) at ({90+72*(2-1)}:1.5cm) {5};
\node[fill=black!30, circle, draw] (u4) at ({90+72*(3-1)}:1.5cm) {4};
\node[fill=blue!30, circle, draw] (u3) at ({90+72*(4-1)}:1.5cm) {3};
\node[fill=green!30, circle, draw] (u2) at ({90+72*(5-1)}:1.5cm) {2};

  \foreach \i in {1,...,5} {
    \foreach \j in {1,...,5} {
      \ifnum\i<\j
        \draw (u\i) -- (u\j);
      \fi
    }
  }
\node[anchor=west] at (v1.east) {$v_1$};
\node[anchor=west] at (v2.east) {$v_2$};
\node[anchor=west] at (v3.east) {$v_3$};
\node[anchor=east] at (v4.west) {$v_4$};

\draw (v4) to[out=-35, in=210] (u4);
\draw (v3) to[out=-25, in=185] (u4);
\draw (v2) -- (u4);
\draw (v1) -- (u4);
\end{tikzpicture}
\begin{tikzpicture}[scale=1, every node/.style={draw=none, fill=none, minimum size=8mm, font=\small}]

\node[font=\LARGE] (Ag) at (-10,0) {$\beta_G$};
\node[font=\LARGE] (Agp) at (-6.5,0) {$\widetilde{\beta}_{G'}$};
\draw[->, thick] (Ag) -- (Agp) node[midway, above,font=\Large] {$(\psi^{[\alpha]},\phi^{[\alpha]})$};

\node at (-11.5,-0.75) {$H_1$};
\node[fill=black!30, circle, draw] (w1) at (-11.5,1.25) {4};
\node[fill=red!30, circle, draw] (w2) at (-11.5,0) {1};
\node[fill=black!30, circle, draw] (w3) at (-10.5,-1.15) {4};
\node[fill=green!30, circle, draw] (w4) at (-12.5,-1.15) {2};

\draw (w1) -- (w2) -- (w3) -- (w4) -- (w2);
\node[anchor=east] at (w1.west) {$u_1$};
\node[anchor=east] at (w2.west) {$u_2$};
\node[anchor=west] at (w3.east) {$u_3$};
\node[anchor=east] at (w4.west) {$u_4$};

\node at (-5,-0.75) {$H_1'$};
\node at (0,0) {$L$};

\node[fill=green!30, circle, draw] (v1) at (-5,1.25) {2};
\node[fill=red!30, circle, draw] (v2) at (-5,0) {1};
\node[fill=green!30, circle, draw] (v3) at (-4,-1.15) {2};
\node[fill=yellow!30, circle, draw] (v4) at (-6,-1.15) {5};

\draw (v1) -- (v2) -- (v3) -- (v4) -- (v2);

\node[fill=red!30, circle, draw] (u1) at ({90+72*(1-1)}:1.5cm) {1};
\node[fill=yellow!30, circle, draw] (u5) at ({90+72*(2-1)}:1.5cm) {5};
\node[fill=black!30, circle, draw] (u4) at ({90+72*(3-1)}:1.5cm) {4};
\node[fill=blue!30, circle, draw] (u3) at ({90+72*(4-1)}:1.5cm) {3};
\node[fill=green!30, circle, draw] (u2) at ({90+72*(5-1)}:1.5cm) {2};

  \foreach \i in {1,...,5} {
    \foreach \j in {1,...,5} {
      \ifnum\i<\j
        \draw (u\i) -- (u\j);
      \fi
    }
  }
\node[anchor=west] at (v1.east) {$v_1$};
\node[anchor=west] at (v2.east) {$v_2$};
\node[anchor=west] at (v3.east) {$v_3$};
\node[anchor=east] at (v4.west) {$v_4$};

\draw (v4) to[out=-35, in=210] (u4);
\draw (v3) to[out=-25, in=185] (u4);
\draw (v2) -- (u4);
\draw (v1) -- (u4);
\end{tikzpicture}
\end{center}
    \caption{Proper colorings of two graphs $G$, with $k=4>\chi$, and $G'$, with $k'=5=\chi'$. Consider the pair of functions $(\psi^{[\alpha]},\phi^{[\alpha]})$ defined by $\phi(u_i)=v_i$ and $\psi(1)=1$, $\psi(2)=5$, $\psi(3)=3$, and $\psi(4)=2$ with the palette $\mathcal{P}=\{1,2,3,5\}$.  The top row shows a proper coloring $\alpha_{G'}$ of $G'$ that relates to $\alpha_G$ via the map $(\psi^{[\alpha]},\phi^{[\alpha]})$. These two vertices look identical locally in their coloring graphs. The bottom row shows that this pair $(\psi^{[\alpha]},\phi^{[\alpha]})$, when applied to an abstract link vertex $\beta\not\in[\alpha]$, still creates a proper coloring $\widetilde{\beta}_{G'}$ that locally resembles an abstract link vertex. We note that $G$ and $G'$ do not have equivalent coloring graphs; this example merely illustrates Properties 2 and 3 in Corollary~\ref{cor:locked-vertices} and Lemma~\ref{lem:all_extend}.
    }\label{fig:surgery}
\end{figure}

\subsection{\texorpdfstring{Abstract link vertices of $\C_k(G),k>\chi(G)$ and $\C_{\chi'}(G'),\chi'=\chi(G')$ can never be in bijection if $\mathcal C_k(G)\cong \C_{\chi'}(G')$}{Abstract link vertices of C{k}(G),k>chi(G) and C{chi'}(G'),chi'=chi(G') can never be in bijection if C{k}(G)=C{chi'}(G')}}\label{subsect:final}

In this final subsection, we use the structural results accumulated by Algorithms~\ref{algo:link-vertex}, \ref{algo:partitions}, and \ref{algo:LLG} to establish our main result: if a coloring graph $\mathcal{C}$ arises from a graph with surplus colors, then $\mathcal{C}$ does \emph{not} arise from any other base graph. 

Theorem~\ref{thm:reconstruction} establishes that if $\mathcal{C}_k(G)$ with $k>\chi(G)$ is isomorphic to $\mathcal{C}_{k'}(G')$, then $k'=\chi'\colonequals\chi(G')$ (meaning $G'$ has no surplus colors). Theorem~\ref{thm:main} below proves this remaining scenario is impossible. Assuming, for contradiction, that $\mathcal{C}_k(G)\cong \mathcal{C}_{\chi'}(G')$, we use properties of abstract link vertices. Our analysis reveals that if the shared graph $\mathcal{C}$ is interpreted as $\mathcal{C}_{\chi'}(G')$, it must contain $\chi'!$ times the number of abstract link vertices defined by its structure as $\mathcal{C}_{k}(G)$. This discrepancy yields a contradiction, ruling out the isomorphism. The theorem upgrades Corollary~\ref{cor:equal-number-of-vertices} by removing the dependence on the number of vertices.

\medskip 

\begin{theorem}\label{thm:main}
    Given a graph $G$ and $k>\chi(G)$, there is no graph $G'$ such that $\C_k(G)\cong \C_{\chi(G')}(G')$.
\end{theorem}
\begin{proof}
    Assume, to the contrary, that $\C\cong \C_k(G)\cong\C_{\chi'}(G')$ for some graphs $G$ and $G'$, where $k>\chi(G)$ and $\chi'=\chi(G')$. First, we show that $\chi'>1$. If $\chi'=1$, then $G'$ is the edgeless graph $N_{r}$ (with $r\geq 1$ vertices). The $1$-coloring graph $\mathcal{C}_{1}(N_{r})$ contains exactly one vertex (representing the single coloring that assigns color $1$ to all vertices). In contrast, since $k>\chi(G)$, any proper $\chi(G)$-coloring of $G$ gives at least two distinct proper $k$-colorings of $G$ by embedding its color set into $[k]$ in two different ways. Thus, $\mathcal{C}_{k}(G)$ has at least two vertices, contradicting that $\mathcal{C}_{1}(N_{r})$ has only one vertex. So, we must have $\chi'>1$.
    
    Let $A$ be the set of abstract link vertices in $\mathcal{C}$. Since $\mathcal{C}\cong \mathcal{C}_{k}(G)$ where $k>\chi(G)$, abstract link vertices exist by Lemma~\ref{lem:link-equivalence}, so $|A|\geq 1$. Fix a vertex $\alpha\in A$, and let $(\psi^{[\alpha]},\phi^{[\alpha]})$ be a pair associated with the equivalence class $[\alpha]$ as in Lemma~\ref{lem:gluing-lifting-data}. As before, let $L\subseteq V(G')$ denote the locked set for $(\psi^{[\alpha]},\phi^{[\alpha]})$. For every $\beta\in A$, define $\widetilde{\beta}_{G'}$ by $
    \widetilde{\beta}_{G'}(\phi^{[\alpha]}(v))=\psi^{[\alpha]}(\beta_G(v))$ for $v\in V(G)$ and $\widetilde{\beta}_{G'}(w)=\alpha_{G'}(w)$ for $w\in L$. If $\beta\in [\alpha]$, we have $\widetilde{\beta}_{G'} = \beta_{G'}$ by Corollaries~\ref{cor:gluing-relations} and \ref{cor:equivalent-links-agree-on-L}, which corresponds to an abstract link vertex. If $\beta\notin [\alpha]$, then $\widetilde{\beta}_{G'}$ is a proper $\chi'$-coloring of $G'$ by Lemma~\ref{lem:all_extend}, corresponding to an abstract link vertex in $\C$. The construction produces $|A|$ distinct proper $\chi'$-colorings of $G'$, each representing a different abstract link vertex in $\mathcal{C}$, that all agree pointwise on the coloring of $L$.  
    
    For each such coloring $\widetilde{\beta}_{G'}$, consider the set of $\chi'!$ colorings obtained by applying every permutation $\sigma\in S_{\chi'}$ to the global color set $[\chi']$. Each resulting coloring $\sigma\circ \widetilde{\beta}_{G'}$ is a proper $\chi'$-coloring of $G'$. Since $\widetilde{\beta}_{G'}$ corresponds to an abstract link vertex (a property preserved under permutation of colors), each $\sigma\circ\widetilde{\beta}_{G'}$ must also correspond to an abstract link vertex. These $\chi'!$ colorings generated from a single $\widetilde{\beta}_{G'}$ are all pairwise distinct. Indeed, if $\sigma\circ \widetilde{\beta}_{G'}=\tau\circ \widetilde{\beta}_{G'}$, then their restrictions to $L$ are equal: $\sigma(\alpha_{G'}|_{L})=\tau(\alpha_{G'}|_{L})$. Since $\alpha_{G'}(L)=[\chi']$ by Corollary~\ref{cor:locked-vertices} \eqref{item:every-color-locked}, we deduce that $\sigma=\tau$.
    
    Moreover, the resulting sets of $\chi'!$ colorings generated by two \emph{distinct} initial colorings, say $\widetilde{\beta}_{G'}$ and $\widetilde{\gamma}_{G'}$ (where $\beta, \gamma\in A$ with $\beta\neq \gamma$), are themselves disjoint. Indeed, if $\sigma\circ \widetilde{\beta}_{G'}=\tau\circ \widetilde{\gamma}_{G'}$, then considering the colors on $L$, we again get $\sigma(\alpha_{G'}|_{L})=\tau(\alpha_{G'}|_{L})$. The same reasoning from the previous paragraph yields $\sigma=\tau$. This implies $\widetilde{\beta}_{G'}=\widetilde{\gamma}_{G'}$, and consequently, $\widetilde{\beta}_{G'}(\phi^{[\alpha]}(v))=\widetilde{\gamma}_{G'}(\phi^{[\alpha]}(v))$ for all $v\in V(G)$. By definition, $\psi^{[\alpha]}(\beta_{G}(v)) = \psi^{[\alpha]}(\gamma_{G}(v))$ for all $v\in V(G)$.  As $\psi^{[\alpha]}$ is injective, $\beta_G(v)=\gamma_G(v)$ for all $v\in V(G)$, leading to $\beta_{G}=\gamma_{G}$. This equality between $k$-colorings of $G$ implies $\beta=\gamma$ as vertices in $\mathcal{C}$, yielding a contradiction.

    Therefore, the overall process generates exactly $|A|\cdot \chi'!$ distinct $\chi'$-colorings of $G'$, each corresponding to an abstract link vertex in $\mathcal{C}$. Hence, the number of abstract link vertices in $\mathcal{C}\cong \mathcal{C}_{\chi'}(G')$ is at least $|A|\cdot \chi'!$. By definition, the number of abstract link vertices in $\mathcal{C}\cong \mathcal{C}_{k}(G)$ is exactly $|A|$. Comparing the two counts, we obtain $|A|\geq |A|\cdot \chi'!$. This inequality yields a final contradiction, because $|A|\geq 1$ and $\chi'>1$.
\end{proof}

The progression of results in Section~\ref{sec:unique-if-surplus}, which collectively establish Theorem~\ref{thm:main}, is visualized in Figure~\ref{fig:flowchart} on the next page. Combining Theorem~\ref{thm:reconstruction} and Theorem~\ref{thm:main}, our main theorem, Theorem~\ref{thm:main-intro}, is now proved. We conclude that coloring graphs with surplus colors are complete graph invariants. 

\medskip 

\textbf{Acknowledgments.} We thank the referees for their helpful comments and suggestions, which improved the exposition of the paper.

\newpage 

\begin{figure}[ht]
    \centering
\begin{tikzpicture}[
    node_style/.style={
        rectangle,
        rounded corners,
        draw=blue!50!black!100,
        fill=blue!15,
        text width=3.2cm, 
        align=center,
        minimum height=1.1cm,
        font=\scriptsize
    },
    thm_style/.style={
        node_style,
        fill=red!20,
        draw=red!50!black!100,
        font=\scriptsize\bfseries,
        text width=3.9cm
    },
     algo_style/.style={
        node_style,
        draw=orange!50!black!100,
        fill=orange!25,
        text width=3.4cm,
        align=center,
        minimum height=1cm,
        font=\scriptsize
    },
    hyp_style/.style={
        node_style,
        draw=gray!60!black,
        fill=gray!23,
        text width=3.4cm,
        align=center,
        minimum height=1cm,
        font=\scriptsize
    },
    arrow_style/.style={
        -Stealth,
        thick,
        draw=black!70
    },
    node distance=1.5cm and 1cm
]

\node[algo_style] (A1) {Algorithm 1: \\ Reconstruction + Link Vertices};
\node[hyp_style, below right = 1cm and 0.5 cm of A1] (H1) {Hypothesis 1: $\mathcal{C}_k(G)\cong\mathcal{C}_{\chi'}(G')$};
\node[algo_style, right = 1cm of A1] (A2) {Algorithm 2: \\ Partition Extraction};
\node[node_style, left=0.5cm of H1] (L43) {Lem 4.3: Abstract link vertices $\leftrightarrow$ link $k$-colorings};
\node[node_style, right=0.5cm of H1] (L411) {Lem 4.11: Alg 2 identifies partitions};

\node[node_style, right=1cm of L411] (L419) {Lem 4.19: Alg 3 labels link graphs uniquely};

\node[node_style, below=1.2cm of L419] (L418) {Lem 4.18: Certain labeled $2$-paths imply a unique $6$-cycle of cubes};

\node[node_style, below=1.2cm of L43] (T44) {Thm 4.4: Unique $G$ with surplus colors generates $\C$};

\node[node_style, below=1.2cm of T44] (L46) {Lem 4.6: Components of $G$ are induced in $G'$};
\node[node_style, right=of L46, xshift=1.2cm] (L412) {Lem 4.12 Partitions induced by a link coloring are the same for $G$ and $G'$};

\node[node_style, below=1.5cm of L412, xshift=-2.4cm] (L423) {Lem 4.23: A pair ($\psi^{[\alpha]},\phi^{[\alpha]})$ relates links in $G$ and $G'$};

\node[node_style, below=0.8cm of L423, xshift=-1cm] (L421) {Lem 4.21 and 4.22: cycles in $\mathcal{L}'_{[\alpha]}$};

\node[algo_style, right=1.6cm of A2] (A3) {Algorithm 3: Identify Labeled Link Graph};

\node[node_style, right=1.5cm of L423] (C425) {Cor 4.25: Colorings in $[\alpha]$ agree on $L$};
\node[node_style, right=1.2cm of C425] (C426) {Cor 4.26: Palette and locked set properties};
\node[node_style, below=1.2cm of C426] (L427) {Lem 4.27: Permuting colors creates new links};

\node[thm_style, below=1.8cm of L423, yshift=-0.5cm, xshift=3cm] (T428) {Thm 4.28: $\mathcal{C}_k(G) \not\cong \mathcal{C}_{\chi'}(G')$};

\draw[arrow_style] (H1.south) -- (L46.north east); %
\draw[arrow_style] (H1.south) -- (L412.north); %

\draw[arrow_style] (L418) -- (L419);

\draw[arrow_style] (L43) -- (T44);

\draw[arrow_style] (A3) -- (L419);

\draw[arrow_style] (L411) to (L412);

\draw[arrow_style] (A1) to (A2);

\draw[arrow_style] (A2) to (A3);

\draw[arrow_style] (A2) to (L411);

\draw[arrow_style] (A3) to[bend left=50] (L423);

\draw[arrow_style] (T44) -- (L46);

\draw[arrow_style] (L46) -- (L412);
\draw[arrow_style] (L46) to (L423);

\draw[arrow_style] (L412) to (L423);

\draw[arrow_style] (A1) to (L43);

\draw[arrow_style] (L423) -- (C425);
\draw[arrow_style] (L423) to[bend right=15] (C426);
\draw[arrow_style] (L423.south) -- (T428.north);

\draw[arrow_style] (L421) to (L423);

\draw[arrow_style] (C425) to (C426);

\draw[arrow_style] (C426) -- (L427.north);
\draw[arrow_style] (C426) -- (T428.north east);

\draw[arrow_style] (L427) -- (T428.east);

\end{tikzpicture}
\caption{Flowchart for Section~\ref{sec:unique-if-surplus}}\label{fig:flowchart}
\end{figure}

\bibliographystyle{alpha}
\bibliography{biblio}

\end{document}